\documentclass[12pt]{article}
 \usepackage{color,graphicx,amsmath,amsfonts,amssymb}
\usepackage[numbers]{natbib}
 
\newcommand{\rs}[1]{{\mbox{\scriptsize \sc #1}}}
\newcommand{\vc}[1]{{\boldsymbol #1}} 
\newcommand{\vcn}[1]{{\bf #1}}

\newcommand{\sr}[1]{{\cal #1}}
 \newcommand{\dd}[1]{\mathbb{#1}}
 \newcommand{\edge}{{\rm e}}
 \newcommand{\cp}{{\rm c}}
 \newcommand{\rmn}[1]{\if#11I\else {\if#12I\hspace{-0.12ex}I\hspace{-0.85ex}\else {\if #13I\hspace{-0.16ex}I\hspace{-0.16ex}I\hspace{-1.6ex}\else I\hspace{-1.2ex}V \fi} \fi} \fi}

\newcommand{\eqn}[1]{(\ref{eqn:#1})}
\newcommand{\lem}[1]{Lemma~\ref{lem:#1}}
\newcommand{\cor}[1]{Corollary~\ref{cor:#1}}
\newcommand{\thr}[1]{Theorem~\ref{thr:#1}}

\newcommand{\pro}[1]{Proposition~\ref{pro:#1}}

\newcommand{\rem}[1]{Remark~\ref{rem:#1}}
\newcommand{\fig}[1]{Figure~\ref{fig:#1}}
\newcommand{\app}[1]{Appendix~\ref{app:#1}}
\newcommand{\sectn}[1]{Section~\ref{sect:#1}}

\newcommand{\lemt}[1]{\ref{lem:#1}}

\newcommand{\thrt}[1]{\ref{thr:#1}}

\newcommand{\figt}[1]{\ref{fig:#1}}

\newcommand{\sect}[1]{\ref{sect:#1}}

\newcommand{\br}[1]{\langle #1 \rangle}

\newcommand{\ol}{\overline}
\newcommand{\ul}{\underline}
\newcommand{\pend}{\hfill \thicklines \framebox(6.6,6.6)[l]{}}

\newenvironment{proof}{\noindent {\sc  Proof.} \rm}{\pend}
\newenvironment{proof*}[1]{\noindent {\sc  #1} \rm}{\pend}
\newenvironment{keyword}[1]{ {\bf #1} \rm}{}

\newtheorem{theorem}{Theorem}[section]
\newtheorem{lemma}{Lemma}[section]

\newtheorem{proposition}{Proposition}[section]

\newtheorem{remark}{Remark}[section]

\newtheorem{corollary}{Corollary}[section]
\newtheorem{definition}{Definition}[section]

\newenvironment{mylist}[1]{\begin{list}{}
{\setlength{\itemindent}{#1mm}}
{\setlength{\itemsep}{0ex plus 0.2ex}}
{\setlength{\parsep}{0.5ex plus 0.2ex}}
{\setlength{\labelwidth}{10mm}}
}{\end{list}}

 \newcommand{\setnewcounter} {
 \setcounter{subsection}{0}
 \setcounter{equation}{0}
 \setcounter{conjecture}{0}
 \setcounter{assumption}{0}
 \setcounter{question}{0}
 \setcounter{definition}{0}
 \setcounter{theorem}{0}
 \setcounter{corollary}{0}
 \setcounter{lemma}{0}
 \setcounter{proposition}{0}
 \setcounter{remark}{0}
}

\begin{document}
 \title{Tail asymptotics of the stationary distribution of\\ a two dimensional reflecting random walk\\ with unbounded upward jumps}

\author{Masahiro Kobayashi and Masakiyo Miyazawa\\ Tokyo University of Science}
\date{May 5, 2014, proofed version}

\maketitle

\begin{abstract}
  We consider a two dimensional reflecting random walk on the nonnegative integer quadrant. This random walk is assumed to be skip free in the direction to the boundary of the quadrant, but may have unbounded jumps in the opposite direction, which are referred to as upward jumps. We are interested in the tail asymptotic behavior of its stationary distribution, provided it exists. Assuming the upward jump size distributions have light tails, we completely find the rough tail asymptotics of the marginal stationary distributions in all directions. This generalizes the corresponding results for the skip free reflecting random walk in \cite{Miya2009}. We exemplify these results for a two node network with exogenous batch arrivals.
\end{abstract}

\begin{keyword}{AMS classification:}
Primary; 60K25, 60K25, 
Secondary; 60F10, 60G50
\end{keyword}

\begin{keyword}{Keywords:}
two dimensional reflecting random walk, stationary distribution, tail asymptotics, linear convex order, unbounded jumps, two node queueing network, batch arrival
\end{keyword}

\section{Introduction}
\label{sect:Introduction}

  We consider a two dimensional reflecting random walk on the nonnegative integer quadrant. This random walk is assumed to be skip free toward the boundary of the quadrant but may have unbounded jumps in the opposite direction, which we call upward jumps. Here, the boundary is composed of the origin and two half coordinate axes, which are called boundary faces. The transitions on each boundary face are assumed to be homogeneous. This reflecting process is referred to as a double $M/G/1$-type process. This process naturally arises in queueing networks with two dimensional compound Poisson arrivals and exponentially distributed service times. Here, customers may simultaneously arrive in batch at different nodes. It also has own interest as a multidimensional reflecting random walk on the nonnegative integer quadrant.
  
  A stationary distribution is one of the most important characteristics for this reflecting random walk in application. However, it is known very hard to analytically derive it except for some special cases. Thus, theoretical interest has been directed to its tail asymptotic behaviors. Borovkov and Mogul'ski{\u\i} \cite{BoroMogu2001} made great contributions to this problem. They proposed the so called partially homogenous chain, which includes the present random walk as a special case, and studied the tail asymptotic behavior of its stationary distribution. Their results are very general, but have the following limitations.
\begin{mylist}{3}
\item [(a)] The results are not very explicit. That is, it is hard to see how the modeling primitives, that is, the parameters which describe the model, influence the tail asymptotics.
\item [(b)] The tail asymptotics are only obtained for small rectangles. No marginal distribution is considered. Furthermore, some extra technical conditions which seem to be unnecessary are assumed.
\end{mylist}

  For the skip free two-dimensional reflecting random walk, these issues have been addressed by Miyazawa \cite{Miya2009}, and its tail asymptotics have been recently answered for the marginal distributions in \cite{KobaMiya2012}, which considers the stationary equation using generating or moment generating functions and applies classical results of complex analysis (see, e.g., \cite{Miya2011} and references therein). This classical approach has been renewed in combination with other methods in recent years (e.g., see \cite{GuilLeeu2011,LiZhao2011}).  However, its application is limited to skip free processes because of technical reasons (see \rem{kernel method}).
  
  In this paper, our primary interest is to see how the tail asymptotics are changed when upward jumps may be unbounded. For this, we consider the asymptotics of the stationary probability of the tail set:
\begin{eqnarray*}
  \{ (i_{1}, i_{2}) \in \dd{Z}_{+}^{2}; c_{1} i_{1} + c_{2} i_{2} \ge x \}
\end{eqnarray*}
  as $x$ goes to infinity for each $c_{1}, c_{2} \ge 0$, where $\dd{Z}_{+}$ is the set of all nonnegative integers, where $\vc{c} = (c_{1}, c_{2})$ is called a directional vector if $c_{1}^{2} + c_{2}^{2} = 1$. We are interested in its decay rate, where $\alpha$ is said to be the decay rate of function $p(x)$ if
\begin{eqnarray*}
  \alpha = - \lim_{n \to \infty} \frac 1{x} \log p(x).
\end{eqnarray*}
  We aim to derive the tail decay rates of the marginal distributions in all directions. These decay rates will be geometrically obtained from the curves which are produced by the modeling primitives, that is, one step transitions of the reflecting random walk. In this way, we answer our question, which simultaneously resolves the issues (a) and (b). 
  
  Obviously, if the tail decay rates are positive, then the one step transitions of the random walk must have light tails, that is, they decay exponentially fast. Thus, we assume this light tail condition. Then, the mean drifts of the reflecting random walk in the interior and on the boundary faces are finite. Using these means, Fayolle, Iasnogorodski and Malyshev \cite{FayoMalyMens1995} characterize stability, that is, the existence of the stationary distribution. However, their result is incomplete as they did not consider all cases, which we redress in \lem{stability d=2}. If the mean drifts vanish in all directions, then it is not hard to see that the stationary distribution does not have light tails in all directions. We also formally verify this fact. Thus, we assume that not all the mean drifts vanish in addition to the light tail condition for the one step transitions.
  
   Under these conditions, we completely solve the decay rate problem for the marginal stationary distribution in each direction, provided the stability conditions hold. The decay rate may be considered as rough asymptotic, and we also study some finer asymptotics, called exact asymptotics. Here, the tail probability is said to have exact asymptotics for some function $f$ if the ratio of the tail probability at level $x$ to $f(x)$ converges to a positive constant as $x$ goes to infinity. In particular, if $f$ is exponential (or geometric), it is said to be exactly exponential (or geometric). We derive some sufficient conditions for the tail asymptotic to be exactly exponential. 
  
  The difficulty of the tail asymptotic problem mainly arises from reflections at the boundary of the quadrant. We have two major boundary faces corresponding to the coordinate axes. The reflections on these faces may or may not influence the tail asymptotics. Foley and McDonald \cite{FoleMcDo2005} classified them as jitter, branch and cascade cases. We need to consider all of those influence to find the tail asymptotics. This problem becomes harder because of the unbounded jumps. 
  
  To overcome these difficulties, we here take a new approach. Based on the stability conditions, we first derive the convergence domain of the moment generating function of the stationary distribution. For this, we use a stationary inequality, which was recently introduced by the second author \cite{Miya2011} (see also \cite{DaiMiya2011}), and a lower bound for the large deviations of the stationary distribution due to \cite{BoroMogu2001}. Once the domain is obtained, it is expected that the decay rate would be obtained through the boundary of the domain. However, this is not immediate. We need sharper lower bounds for the large deviations in coordinate directions. For this, we use a method based on Markov additive processes (e.g., see \cite{MiyaZhao2004}).
  
   We apply one of our main results, Theorems \thrt{decay rate 0} and \thrt{decay rate 1}, to a batch arrival network with two nodes to see how the modeling primitives influence the tail asymptotics. We use linear convex order for this. We also show that the stochastic upper bound for the stationary distribution obtained by Miyazawa and Taylor \cite{MiyaTayl1997} is not tight unless one of the nodes has no batch arrival.
  
  This paper is made up by six sections. In \sectn{Double}, we formally introduce the reflecting random walk, and discuss its basic properties including stability and stationary equations. In \sectn{Convergence}, main results on the domain and tail asymptotics, Theorems \thrt{domain D}, \thrt{decay rate 0} and \thrt{decay rate 1}, are presented. We also discuss about linear convex order. Those theorems are proved in \sectn{Proofs}. We apply them to a two node network with exogenous batch arrivals in \sectn{Application}. We give some remarks on extensions and further work in \sectn{Concluding}.

\section{Double $M/G/1$-type process}
\label{sect:Double}
\setnewcounter

 Denote the state space by $S=\dd{Z}^{2}_{+}$. Recall that $\dd{Z}_{+}$ is the set of all nonnegative integers. Similarly, $\dd{Z}$ denotes the set of all integers. Define the boundary faces of $S$ as
\begin{eqnarray*}
S_{0} = \{ (0,0) \}, \hspace{3mm} S_{1} = \{(i,0) \in \dd{Z}^{2}_{+}; i \geq 1 \}, \hspace{3mm} S_{2} = \{(0,i) \in \dd{Z}^{2}_{+}; i \geq 1 \}.
\end{eqnarray*}
  Let $\partial S = \cup_{i=0}^{2} S_{i}$ and $S_{+} = S \setminus \partial S$. We refer to $\partial S$ and $S_{+}$ as the boundary and interior of $S$, respectively.
  
  Let $\{\vc{Y}(\ell); \ell=0,1,\ldots\}$ be the random walk on $\dd{Z}^{2}$. That is, its one step increments $\vc{Y}(\ell+1) - \vc{Y}(\ell)$ are independent and identically distributed. We denote a random vector subject to the distribution of these increments by $\vc{X}^{(+)}$. We generate a reflecting process $\{\vc{L}(\ell)\} \equiv \{(L_{1}(\ell),L_{2}(\ell))\}$ from this random walk $\{\vc{Y}(\ell)\}$ in such a way that it is a discrete Markov chain with state space $S$ and the transition probabilities $p(\vc{i}, \vc{j})$ given by
\begin{eqnarray*}
P(\vc{L}(\ell+1) = \vc{j} |\vc{L}(\ell) = \vc{i}) = \left\{ 
\begin{array}{ll}
P(\vc{X}^{(+)} = \vc{j}-\vc{i}),  & \vc{j} \in S, \vc{i} \in S_{+},\\
P(\vc{X}^{(k)} = \vc{j}-\vc{i}), & \vc{j} \in S, \vc{i} \in S_{k}, k=0,1,2,
\end{array}
\right.
\end{eqnarray*}
where $\vc{X}^{(k)}$ is a random vector taking values in $\dd{Z}^{2}$. For this process to be non defective, it is assumed that $\vc{X}^{(+)} \geq (-1,-1)$, $\vc{X}^{(1)} \geq (-1,0)$, $\vc{X}^{(2)} \geq (0,-1)$, and $\vc{X}^{(0)} \geq (0,0)$ respectively. Here, inequalities of vectors are meant to those in component-wise.

Thus, $\vc{L}(\ell)$ is reflected at the boundary $\partial S$ and skip free in the direction to $\partial S$. In particular, their entries $L_{1}(\ell)$ and $L_{2}(\ell)$ have similar transitions to the queue length process of the $M/G/1$ queue. So, we refer to $\{\vc{L}(\ell)\}$ as a double $M/G/1$-type process.

 Let $\dd{R}$ be the set of all real numbers. Similar to $\dd{Z}_{+}$, let $\dd{R}_{+}$ be the all nonnegative real numbers. We denote the moment generating functions of $\vc{X}^{(+)}$ and $\vc{X}^{(k)}$ by $\gamma$ and $\gamma_{k}$, that is, for $\vc{\theta} \equiv (\theta_{1},\theta_{2}) \in \dd{R}^{2}$,
\begin{eqnarray*}
\gamma(\vc{\theta}) = E(e^{\br{\vc{\theta}, \vc{X}^{(+)}}}), & \gamma_{k}(\vc{\theta}) = E(e^{\br{\vc{\theta}, \vc{X}^{(k)}}}), & k=0,1,2,
\end{eqnarray*} 
where $\br{\vc{a}, \vc{b}}$ is inner product of vectors $\vc{a}$ and $\vc{b}$. As usual, $\dd{R}^{2}$ is considered to be a metric space with Euclidean norm $\| \vc{a} \| \equiv \sqrt{\br{\vc{a}, \vc{a}}}$. In this paper, we assume that
\begin{mylist}{0}
\item [(i)] The random walk $\{\vc{Y}(\ell)\}$ is irreducible and aperiodic. 
\item [(ii)] The reflecting process $\{\vc{L}(\ell)\}$ is irreducible and aperiodic.
\item [(iii)] For each $\vc{\theta} \in \dd{R}^{2}$ satisfying $\theta_{1} > 0$ or $\theta_{2} > 0$, there exist $t > 0$ and $t_{k} > 0$ such that $1 < \gamma(t\vc{\theta}) < \infty$ and $1 < \gamma_{k}(t_{k} \vc{\theta}) < \infty$ for $k=0,1,2$. 
\item [(iv)] Either $E(X^{(+)}_{1}) \ne 0$ or $E(X^{(+)}_{2}) \ne 0$ for $\vc{X}^{(+)} = (X^{(+)}_{1}, X^{(+)}_{2})$.
\end{mylist}
The conditions (i) and (iii) are stronger than what are actually required, but we use them for simplifying arguments. For example, except for the exact asymptotics, (iii) can be weaken to the following condition (see \rem{Gamma}). 
\begin{mylist}{2}
\item [(iii)'] $\gamma(\vc{\theta})$ and $\gamma_{k}(\vc{\theta})$ for $k=0,1,2$ are finite for some $\vc{\theta} > \vc{0}$. 
\end{mylist}

  In the rest of this section, we discuss three basic topics. First, we consider necessary and sufficient conditions for stability of the reflecting random walk $\{\vc{L}(\ell)\}$, that is, the existence of its stationary distribution, and explain why (iv) is assumed. We then formally define rough and exact asymptotics.

\subsection{Stability condition and tail asymptotics}
\label{sect:Stability}

It is claimed in the book of Fayolle, Malyshev and Menshikov \cite{FayoMalyMens1995} that necessary and sufficient conditions for stability are obtained (see Theorem 3.3.1 of the book). However, the proof of their Theorem 3.3.1 is incomplete because important steps are omitted. A better proof be found in \cite{Fayo1989}. Furthermore, some exceptional cases are missing as we will see. Nevertheless, those necessary and sufficient conditions can be valid under minor amendment.

In \cite{FayoMalyMens1995}, those necessary and sufficient conditions are separately considered according to whether all the mean drifts are null, that is, $E(X^{(+)}_{1}) = E(X^{(+)}_{2}) = 0$ or not. One may easily guess that the null drift implies that the stationary distribution has a heavy tail in all directions, that is, the tail decay rates in all directions vanish. We formally prove this fact in \rem{infinite domain}. Thus, we can assume (iv) to study the light tail asymptotics.

  We now present the stability conditions of \cite{FayoMalyMens1995} under the assumption (iv). We will consider their geometric interpretations. For this, we introduce some notation. Let, for $k=1,2$,
\begin{eqnarray*}
  m_{k} = E(X^{(+)}_{k}), \qquad m^{(1)}_{k} = E(X^{(1)}_{k}), \qquad m^{(2)}_{k} = E(X^{(2)}_{k})
\end{eqnarray*}
  Define vectors
\begin{eqnarray*}
 && \vc{m} = (m_{1}, m_{2}), \qquad \vc{m}^{(1)} = (m^{(1)}_{1}, m^{(1)}_{2}), \qquad \vc{m}^{(2)} = (m^{(2)}_{1}, m^{(2)}_{2}),\\
 && \vc{m}^{(1)}_{\bot} = (m^{(1)}_{2}, - m^{(1)}_{1}), \qquad \vc{m}^{(2)}_{\bot} = (- m^{(2)}_{2}, m^{(2)}_{1}).
\end{eqnarray*}
  Note that $\vc{m} \ne \vc{0}$ by condition (iv). Obviously, $\vc{m}^{(k)}_{\bot}$ is orthogonal to $\vc{m}^{(k)}$ for each $k=1,2$. We present the stability conditions of \cite{FayoMalyMens1995} using these vectors below, in which we make some minor corrections for the missing cases.
  
\begin{lemma}[Corrected Theorem 3.3.1 of \cite{FayoMalyMens1995}]
\label{lem:stability d=2}
  If $\vc{m} \ne \vc{0}$, the reflecting random walk $\{\vc{Z}(\ell)\}$ has the stationary distribution, that is, it is stable, if and only if either one of the following three conditions hold.
\begin{mylist}{-3}
\item [(\rmn{1})] $m_{1} < 0, m_{2} < 0$, $\br{\vc{m}, \vc{m}^{(1)}_{\bot}} < 0$, and $\br{\vc{m}, \vc{m}^{(2)}_{\bot}} < 0$.
\item [(\rmn{2})] $m_{1} \ge 0, m_{2} < 0$, $\br{\vc{m}, \vc{m}^{(1)}_{\bot}} < 0$. In addition to these conditions, $m^{(2)}_{2} < 0$ is required if $m^{(2)}_{1} = 0$. 
\item [(\rmn{3})] $m_{1} < 0, m_{2} \ge 0$, $\br{\vc{m}, \vc{m}^{(2)}_{\bot}} < 0$.  In addition to these conditions, $m^{(1)}_{1} < 0$ is required if $m^{(1)}_{2} = 0$.
\end{mylist}
\end{lemma}
\begin{remark}
\label{rem:stability d=2}
  The additional conditions $m^{(2)}_{2} < 0$ for $m^{(2)}_{1} = 0$ in (\rmn{2}) and $m^{(1)}_{1} < 0$ for $m^{(1)}_{2} = 0$ in (\rmn{3}) are missing in Theorem 3.3.1 of \cite{FayoMalyMens1995}. To see their necessity, let us assume $m^{(2)}_{1} = 0$ in (\rmn{2}). This implies that the reflecting random walk can not get out from the 2nd coordinate axis except for the origin once it hits the axis. On the other hand, $m^{(2)}_{2}$ may take any value because of no constrain on $m^{(2)}_{2}$ in the first three conditions of (\rmn{2}). Hence, $m^{(2)}_{2} < 0$ is necessary for the stability. By symmetry, the extra condition is similarly required for (\rmn{3}). It is not hard to see the sufficiency of these conditions if the proof in \cite{Fayo1989,FayoMalyMens1995} is traced, so we omit its verification.
\end{remark}
  
  We will see that the stability conditions are closely related to the curves $\gamma(\vc{\theta}) = 1$ and $\gamma_{k}(\vc{\theta}) = 1$. So, we introduce the following notation.
\begin{eqnarray*}
 && \Gamma = \{ \vc{\theta} \in \dd{R}^{2}; \gamma(\vc{\theta}) < 1 \},\qquad  \partial \Gamma = \{ \vc{\theta} \in \dd{R}^{2}; \gamma(\vc{\theta}) = 1 \},\\
 && \Gamma_{k} = \{ \vc{\theta} \in \dd{R}^{2}; \gamma_{k}(\vc{\theta}) < 1 \}, \quad \; \partial \Gamma_{k} = \{ \vc{\theta} \in \dd{R}^{2}; \gamma_{k}(\vc{\theta}) = 1 \}, \quad k=1,2.
\end{eqnarray*}
\begin{remark}
\label{rem:Gamma}
  (a) $\Gamma$ and $\Gamma_{k}$ are convex sets because $\gamma$ and $\gamma_{k}$ are convex functions. Furthermore, $\Gamma$ is a bounded set by the condition (i). (b) By the condition (iii), $\partial \Gamma$ and $\partial \Gamma_{k}$ are continuous curves. If we replace (iii) by the weaker assumption (iii)', then $\partial \Gamma$ and $\partial \Gamma_{k}$ may be empty sets. In these cases, we redefine them as the boundaries of $\Gamma$ and $\Gamma_{k}$, respectively. This does not change our arguments as long as the solutions of $\gamma(\vc{\theta}) = 1$ or $\gamma_{k}(\vc{\theta}) = 1$ are not analytically used. Thus, we will see that the rough asymptotics are still valid under (iii)', but this is not the case for the exact asymptotics.
\end{remark}
    
\begin{figure}[h]
 	\centering
	\includegraphics[height=4.9cm]{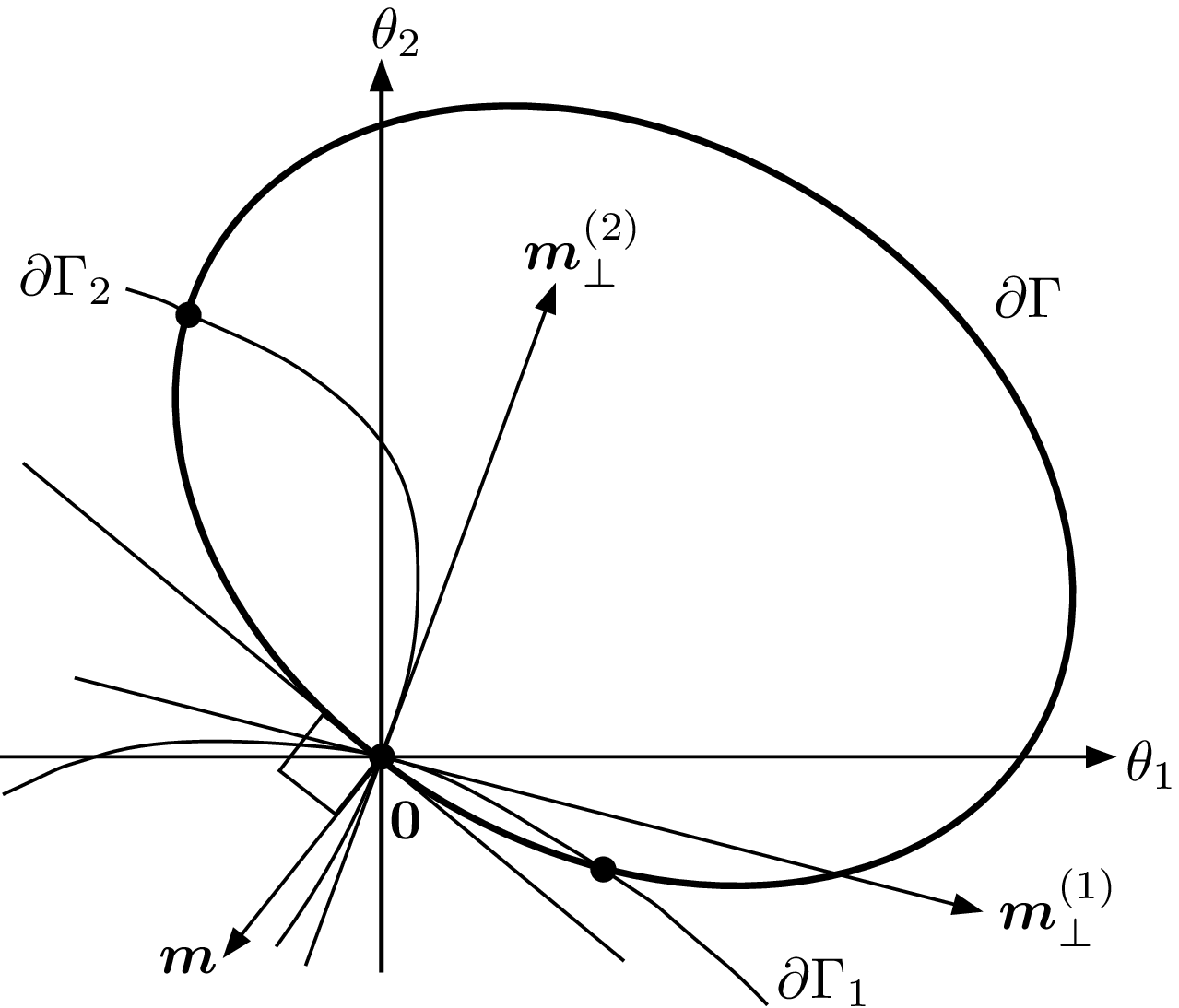} $\;$
	\includegraphics[height=4.4cm]{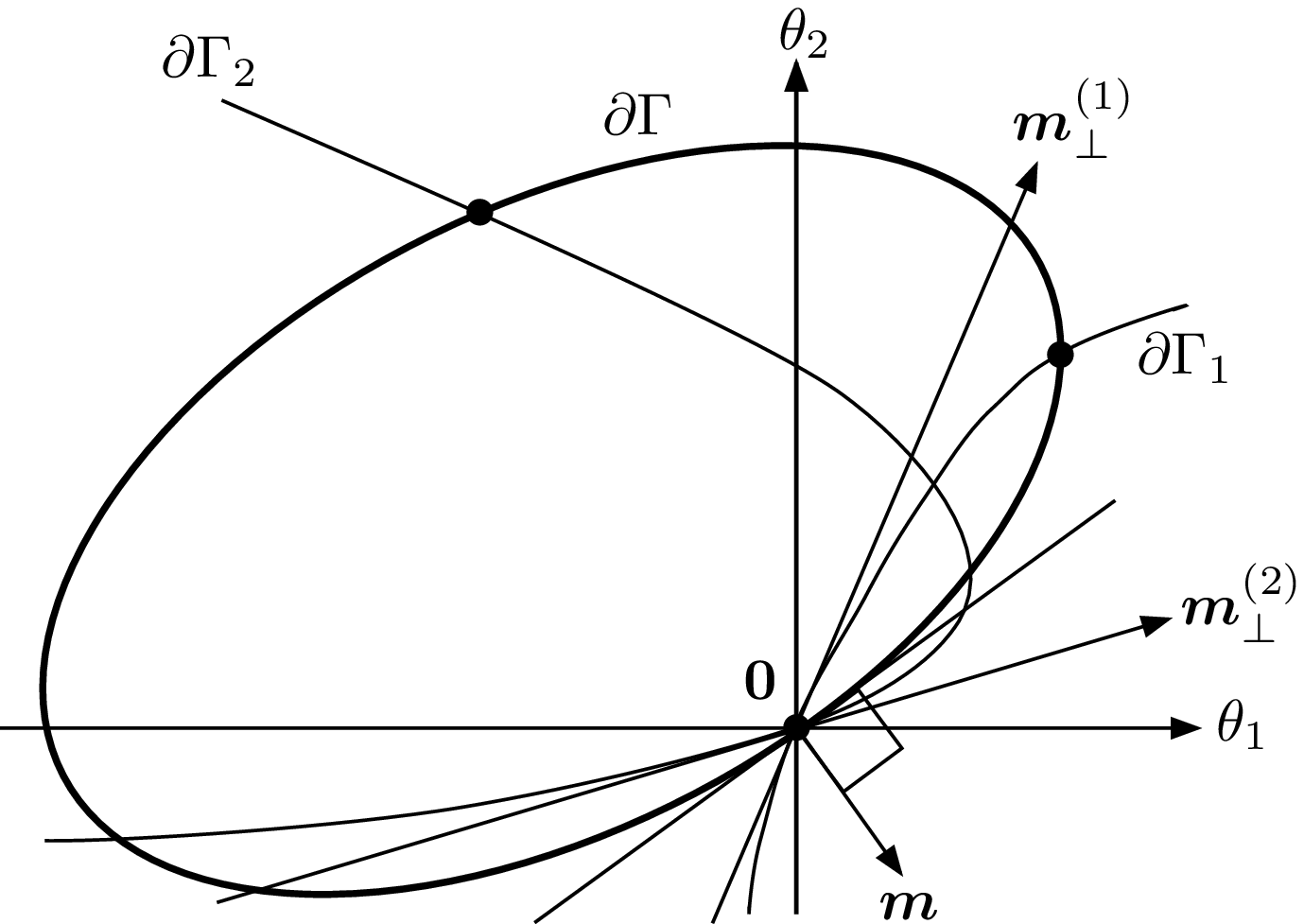}
	\caption{Vectors for conditions (\rmn{1}) and (\rmn{2})}
	\label{fig:stability condition}
\end{figure}
  
  Note that $\vc{m}$ and $\vc{m}^{(k)}$ are normal to the tangents of $\partial \Gamma$ and $\partial \Gamma_{k}$ at the origin and toward the outside of $\Gamma$ and $\Gamma_{k}$, respectively (see \fig{stability condition}). From this observation, we can get the following lemma, which gives a geometric interpretation of the stability condition. We prove it in \app{geometric 1}.
\begin{lemma}
\label{lem:geometric 1}
  Under conditions (iii) and (iv), the reflecting random walk $\{\vc{L}(\ell)\}$ has the stationary distribution if and only if $\Gamma \cap \Gamma_{k}$ contains a vector $\vc{\theta}$ such that $\theta_{k} > 0$ for each $k=1,2$. Furthermore, if this is the case, at least for either one of $k = 1, 2$, there exists a $\vc{\theta} \in \Gamma \cap \Gamma_{k}$ such that $\theta_{k} > 0$ and $\theta_{3-k} \le 0$.
\end{lemma}

  Throughout the paper, we assume stability, that is, either one of the stability conditions (\rmn{1}), (\rmn{2}) and (\rmn{3}), in addition to conditions (i)--(iv). We are now ready to formally introduce tail asymptotics for the stationary distribution of the double $M/G/1$-type process. We denote this stationary distribution by $\nu$. Let $\vc{L} \equiv (L_{1}, L_{2})$ be the random vector subject to the distribution $\nu$, that is,
\begin{eqnarray*}
  \nu(\vc{i}) = P(\vc{L} = \vc{i}), & \vc{i} \in S
\end{eqnarray*}
We are interested in the tail asymptotic behavior of this distribution. For this, we define the rough and exact asymptotics. We refer to vector $\vc{c} \in \dd{R}^{2}$ as a direction vector if $\| \vc{c} \| = 1$. For an arbitrary direction vector $\vc{c} \ge \vc{0}$, we define $\alpha_{\vc{c}}$ as
\begin{eqnarray*}
  \alpha_{\vc{c}} = - \lim_{x \rightarrow \infty} \frac{1}{x} \log P(\br{\vc{c}, \vc{L}} \geq x),
\end{eqnarray*}
  as long as it exists. This $\alpha_{\vc{c}}$ is referred to as a decay rate in the direction $\vc{c}$. Thus, if the decay rate $\alpha_{\vc{c}}$ exists, $P(\br{\vc{c}, \vc{L}} \geq x)$ is lower and upper bounded by $e^{- (\alpha_{\vc{c}} + \epsilon) x}$ and $e^{- (\alpha_{\vc{c}} - \epsilon) x}$, respectively, for any $\epsilon > 0$ and sufficiently large $x \in \dd{R}$. If there exists a function $f$ and a positive constant $b$ such that $P(\br{\vc{c}, \vc{L}} \geq x) \sim bf(x)$, that is,
\begin{eqnarray}
\label{eqn:exact asymptotic}
\lim_{x \rightarrow \infty} \frac{P(\br{\vc{c}, \vc{L}} \geq x)}{f(x)} = b,
\end{eqnarray}
then $P(\br{\vc{c}, \vc{L}} \geq x)$ is said to have exact asymptotic $f(x)$. In particular, if $f$ is exponential, that is, $f(x) = e^{-\alpha x}$ for some $\alpha > 0$, then it is called an exactly exponential asymptotic. It is notable that random variable $\br{\vc{c}, \vc{L}}$ only takes a countable number of real values. Hence, we must be careful about their periodicity. 

\begin{definition}
\label{dfn:periodicity}
(a) A countable set $A$ of real numbers is said to be $\delta$-arithmetic at infinity for some $\delta > 0$ if $\delta$ is the largest number such that, for some $x_{0} > 0$, $\{x \in A; x \ge x_{0} \} $ is a subset of $\{\delta n; n \in \dd{Z}_{+}\}$. On the other hand, $A$ is said to be asymptotically dense at infinity if there is a positive number $a$ for each $\epsilon > 0$ such that, for all $x \ge a$, $|x - y| < \epsilon$ for some $y \in A$. (b) A random variable $X$ taking countably many real values at most is said to be $\delta$-arithmetic for some integer $\delta > 0$ if $\delta$ is the largest positive number such that $\{x \in \dd{R}; P(X=x) > 0\} \subset \{\delta n; n \in \dd{Z}\}$. 
\end{definition}

 The following fact is an easy consequence of Lemma 2 and Corollary in Section V.4a of \cite{Fell1971}, but we prove it in \app{K c} for completeness.  
 
\begin{lemma}
\label{lem:K c}
  For a directional vector $\vc{c} \ge \vc{0}$, let $K_{\vc{c}} = \{ \br{\vc{c}, \vc{n}}; \vc{n} \in \dd{Z}_{+}^{2}\}$. Then, $K_{\vc{c}}$ is asymptotically dense at infinity if and only if neither $c_{1}$ nor $c_{2}$ vanishes and $c_{1}/c_{2}$ is irrational. Otherwise, $K_{\vc{c}}$ is arithmetic at infinity.
\end{lemma}
  Because of this lemma, the $x$ in \eqn{exact asymptotic} runs over either arithmetic numbers $\delta n$ for some $\delta > 0$ or real numbers. Particularly, if $K_{\vc{c}}$ is 1-arithmetic at infinity, then we replace $x$ by integer $n$. For example, this $n$ is used for the asymptotics: $P(L_{k} = n, L_{3-k} = i) \sim f(n,i)$ for each fixed $i \in \dd{Z}_{+}$.

\subsection{Moment generating function and stationary equation}
\label{sect:Moment}

  There are two typical ways to represent the stationary distribution of the double $M/G/1$-type process. One is a traditional expression using either generating or moment generating functions. Another is a matrix analytic expression viewing one of coordinates as a background state. In this paper, we will use both of then because they have their own merits. We first consider the stationary equation using moment generating functions. Since the states are vectors of nonnegative integers, it may be questioned why moment generating functions are not used. This question will be answered at the end of this section. 
  
  We denote the moment generating function of the stationary random vector $\vc{L}$ by
\begin{eqnarray*}
\varphi(\vc{\theta}) \equiv E(e^{\br{\vc{\theta}, \vc{L}}}), \qquad \vc{\theta} \in \dd{R}^{2}.
\end{eqnarray*}
  We define a light tail for the stationary distribution $\nu$ according to \cite{Miya2011}.
\begin{definition}
\label{den:light tail}
  The $\nu$ is said to have a light tail in all directions if there is a positive $\vc{\theta} \in \dd{R}^{2}$ such that $\varphi(\vc{\theta}) < \infty$. Otherwise, it is said to have a heavy tail in some direction.
\end{definition}

Define the convergence domain $\sr{D}$ of the moment generating function $\varphi$  as
\begin{eqnarray*}
  \sr{D} = \mbox{the interior of } \{\vc{\theta} \in \dd{R}; \varphi(\vc{\theta}) < \infty \}.
\end{eqnarray*}
Then, we can expect that the tail asymptotic of the stationary distribution is obtained through the boundary of the domain $\sr{D}$. Obviously, $\sr{D}$ is a convex set because $\varphi$ is a convex function on $\dd{R}^{2}$. Let us derive the stationary equation for this $\sr{D}$. Let
\begin{eqnarray*}
 && \varphi_{+}(\vc{\theta}) = E(e^{\br{\vc{\theta}, \vc{L}}} 1(\vc{L} > \vc{0})),\\
 && \varphi_{k}(\theta_{k}) = E(e^{\theta_{k} L_{k}} 1(L_{k} \geq 1, L_{3-k} = 0)), \quad k=1,2,
\end{eqnarray*}
where $1(\cdot)$ is the indicator function. Then, 
\begin{eqnarray*}
  \varphi(\vc{\theta}) = \varphi_{+}(\vc{\theta}) + \varphi_{1}(\theta_{1}) + \varphi_{2}(\theta_{2}) + \varphi_{0}(0),
\end{eqnarray*}
where $\varphi_{0}(0) = P(\vc{L} = \vc{0})$. From this relation and the stationary equation:
\begin{eqnarray}
\label{eqn:stationary equation 0}
  \vc{L} \simeq \vc{L} + \vc{X}^{(+)} 1(\vc{L} \in S_{+}) + \sum_{k \in \{0, 1, 2\}} \vc{X}^{(k)} 1(\vc{L} \in S_{k}),
\end{eqnarray}
  where $\simeq$ stands for the equality in distribution, and the random vectors in the right hand side are assumed to be independent, we have
\begin{eqnarray}
\label{eqn:stationary equation 1}
 && (1 - \gamma(\vc{\theta}))\varphi_{+}(\vc{\theta}) = \sum_{k \in \{1, 2\}} (\gamma_{k}(\vc{\theta}) - 1) \varphi_{k}(\theta_{k}) + (\gamma_{0}(\vc{\theta}) - 1) \varphi_{0}(0),
\end{eqnarray}
 as long as $\varphi(\vc{\theta}) < \infty$. This equation holds at least for $\vc{\theta} \leq \vc{0}$. 
 
  Equation \eqn{stationary equation 1} is equivalent to the stationary equation of the Markov chain $\{\vc{L}(\ell)\}$, and therefore characterizes the stationary distribution. Thus, the stationary distribution can be obtained if we can solve \eqn{stationary equation 1} for unknown function $\varphi$, equivalently, $\varphi_{+}$, $\varphi_{1}$, $\varphi_{2}$ and $\varphi_{0}$. However, this is known as a notoriously hard problem. This is relatively relaxed when jumps are skip free (see, e.g., \cite{FayoIasnMaly1999,KobaMiya2012,Miya2011}). This point is detailed below.
  
\begin{remark}[Kernel method and generating function]
\label{rem:kernel method}
  To get useful information from \eqn{stationary equation 1}, it is a key idea to consider it on the surface obtained from $1 - \gamma(\vc{\theta}) = 0$. This enables us to express $\varphi_{i}(\theta_{i})$ in terms of the other $\varphi_{j}(\theta_{j})$'s under the constraint that $1 - \gamma(\vc{\theta}) = 0$. Then, we may apply analytic extensions for $\varphi_{i}(\theta_{i})$ using complex variables for $\vc{\theta}$. This analytic approach is called a kernel method (see, e.g., \cite{GuilLeeu2011,LiZhao2011,Miya2011}). In the kernel method, generating function is more convenient than moment generating function. Let $\tilde{\gamma}(z_{1}, z_{2}) = \gamma(\log z_{1}, \log z_{2})$, then $\tilde{\gamma}(z_{1}, z_{2})$ is the generating function corresponding to $\gamma(\vc{\theta})$. Note that $z_{1} z_{2} \tilde{\gamma}(z_{1},z_{2})$ is a polynomial of $z_{1}$ and $z_{2}$. Particularly for the skip free two-dimensional reflecting random walk, $z_{1} z_{2} (1 - \tilde{\gamma}(z_{1},z_{2})) = 0$, which corresponds with  $1 - \gamma(\vc{\theta}) = 0$, is a quadratic equation of $z_{i}$ for each fixed $z_{3-i}$. Hence, we can algebraically solve it, which is thoroughly studied in the book of \cite{FayoIasnMaly1999}. However, the problem is getting hard if the random walk is not skip free. If the jumps are unbounded, the equation $1 - \tilde{\gamma}(z_{1},z_{2}) = 0$ has infinitely many solutions in the complex number field for each fixed $z_{2}$ (or $z_{1}$), and there may be no hope to solve the equation. Even if these roots are found, it would be difficult to analytically extend $\varphi_{i}(z_{i})$.
\end{remark}
 
   This remark suggests that the kernel method based on complex analysis is hard to use for the $M/G/1$-type process. We look at the problem in a different way, and consider the equation $1 - \gamma(\vc{\theta}) = 0$ in the real number field. In this case, it has at most two solutions of $\theta_{i}$ for each fixed $\theta_{3-i}$ because $\gamma(\vc{\theta})$ is a two variable convex function. However, we have a problem on the stationary equation \eqn{stationary equation 1} because we only know its validity for $\vc{\theta} \le \vc{0}$. To overcome this difficulty, we will introduce a new tool, called stationary inequality, and work on $\Gamma$ and $\Gamma_{k}$ for $k=1,2$. For this, moment generating functions are more convenient because $\Gamma$ and $\Gamma_{k}$ will be convex. This convexity may not be true for generating functions because two variable generating functions may not be convex.
  
   Although we mainly use moment generating functions, we do not exclude to use generating functions. In fact, they are convenient to consider the tail asymptotics in coordinate directions. For other directions, we again need moment generating function because $c_{1} L_{1} + c_{2} L_{2}$ may not be periodic. Thus, we will use both of them.

\section{Convergence domain and main results}
\label{sect:Convergence}
\setnewcounter

The aim of this section is to present main results on the domain $\sr{D}$ and the tail asymptotics. They will be proved in \sectn{Proofs}. We first give a key tool for finding the domain $\sr{D}$, which allows us to extend the valid region of \eqn{stationary equation 1} from $\{ \vc{\theta} \in \dd{R}^{2}; \vc{\theta} \le \vc{0}\}$.

\begin{lemma}
\label{lem:stationary equation 2}
  For $\vc{\theta} \in \dd{R}^{2}$, $\varphi(\vc{\theta}) < \infty$ and \eqn{stationary equation 1} holds true if either one of the following conditions is satisfied.
\begin{mylist}{0}
\item [(\sect{Double}a)] $\vc{\theta} \in \Gamma$ and $\varphi_{k}(\vc{\theta}) < \infty$ for $k=1,2$.
\item [(\sect{Double}b)] $\vc{\theta} \in \Gamma \cap \Gamma_{1}$ and $\varphi_{2}(\theta_{2}) < \infty$.
\item [(\sect{Double}c)] $\vc{\theta} \in \Gamma \cap \Gamma_{2}$ and $\varphi_{1}(\theta_{1}) < \infty$.
\end{mylist}
\end{lemma}  

  This lemma is a version of Lemma 6.1 of \cite{Miya2011}, and proved in \app{stationary inequality}. The lemma suggests that $\Gamma$, $\Gamma \cap \Gamma_{1}$ and $\Gamma \cap \Gamma_{2}$ are important to find for $\varphi(\vc{\theta})$ to be finite. They are not empty by \lem{geometric 1}. Obviously, these sets are convex sets. We will use the following extreme points of them.
\begin{eqnarray*}
 && \vc{\theta}^{(k, \max)} = \arg \max_{(\theta_{1},\theta_{2})} \{ \theta_{k}; \gamma(\theta_{1},\theta_{2}) = 1 \},  \quad \vc{\theta}^{(k, \min)} = \arg \min_{(\theta_{1},\theta_{2})} \{ \theta_{k}; \gamma(\theta_{1},\theta_{2}) = 1 \},\\
 && \vc{\theta}^{(k, \cp)} = \arg \sup_{(\theta_{1},\theta_{2})} \{ \theta_{k}; \vc{\theta} \in \Gamma \cap \Gamma_{k} \},  \qquad \vc{\theta}^{(k, \edge)} = \arg \max_{(\theta_{1},\theta_{2})} \{ \theta_{k}; \vc{\theta} \in \partial \Gamma \cap \partial \Gamma_{k} \},
\end{eqnarray*}
  where c and e in the superscripts stand for convergence parameter and edge, respectively (see \fig{case 1}). Their meanings will be clarified in the context of the Markov additive process $\{\vc{Z}^{(1)}(\ell)\}$. 

  From these definitions, it is easy to see that, for $k=1,2$,
\begin{eqnarray*}
\vc{\theta}^{(k, \cp)} = \left\{ 
\begin{array}{ll}
\vc{\theta}^{(k, \edge)}, \quad & \gamma_{k}(\vc{\theta}^{(k, \max)}) > 1,\\
\vc{\theta}^{(k, \max)}, & \gamma_{k}(\vc{\theta}^{(k, \max)}) \leq 1.
\end{array}
\right.
\end{eqnarray*}
  Using these points, we classify their configurations into three categories as:
\begin{center}
(D1) $\;$ $\theta_{1}^{(2, \cp)} < \theta_{1}^{(1, \cp)}$ and $\theta_{2}^{(1, \cp)} < \theta_{2}^{(2, \cp)}$, \quad(D2) $\vc{\theta}^{(2, \cp)} \le \vc{\theta}^{(1, \cp)}$, \quad
(D3) $\vc{\theta}^{(1, \cp)} \le \vc{\theta}^{(2, \cp)}$,
\end{center}
  where we exclude the case that $\theta_{1}^{(2, \cp)} \ge \theta_{1}^{(1, \cp)}$ and $\theta_{2}^{(1, \cp)} \ge \theta_{2}^{(2, \cp)}$ because it is impossible (see Figures \figt{case 1} and \figt{case 2-3}). These categories have been firstly introduced in \cite{Miya2009} for the double QBD processes, and shown to be useful for the tail asymptotic problems.

\begin{figure}[h]
 	\centering
	\includegraphics[height=4.3cm]{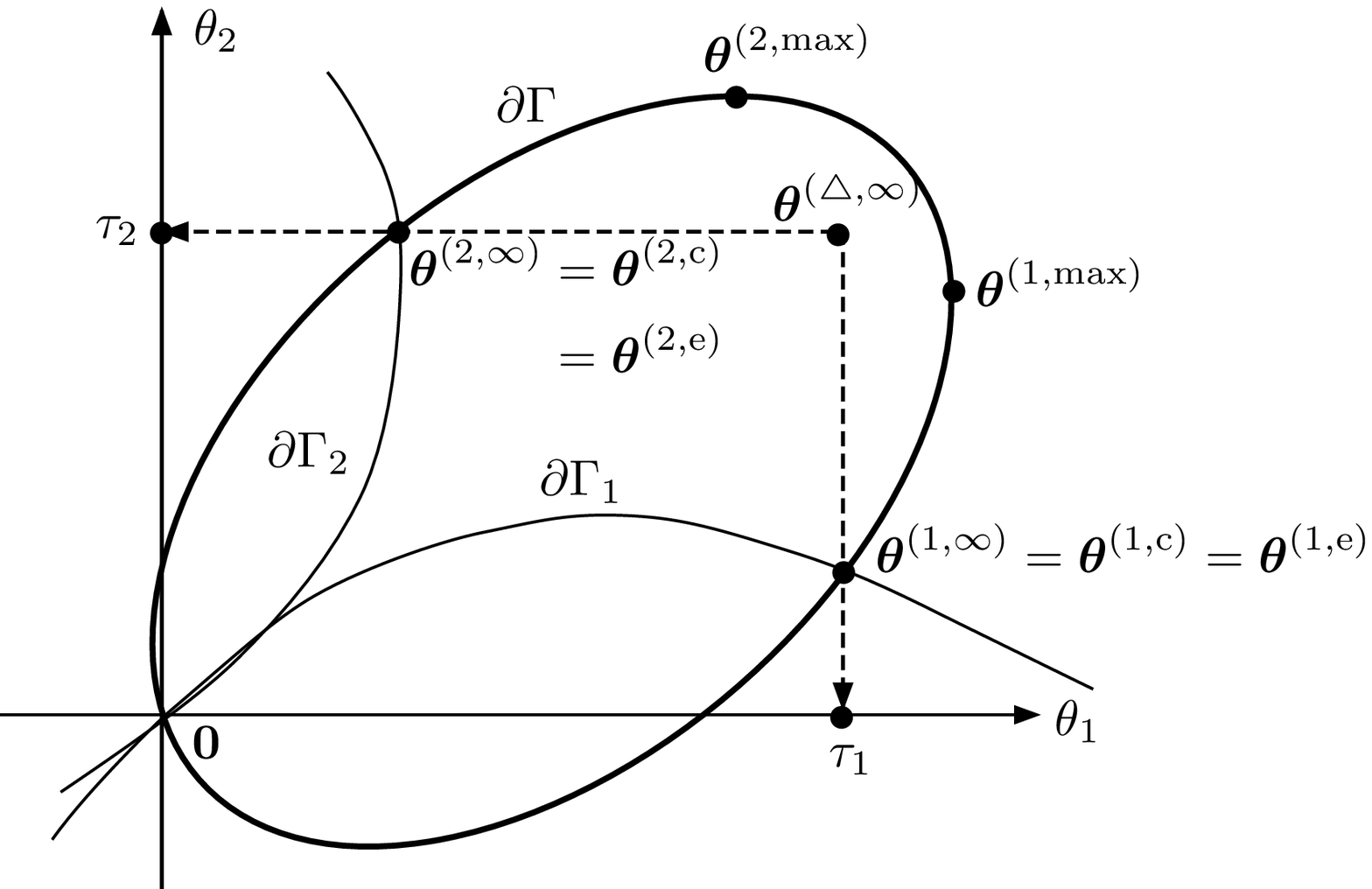} 
	\includegraphics[height=4.3cm]{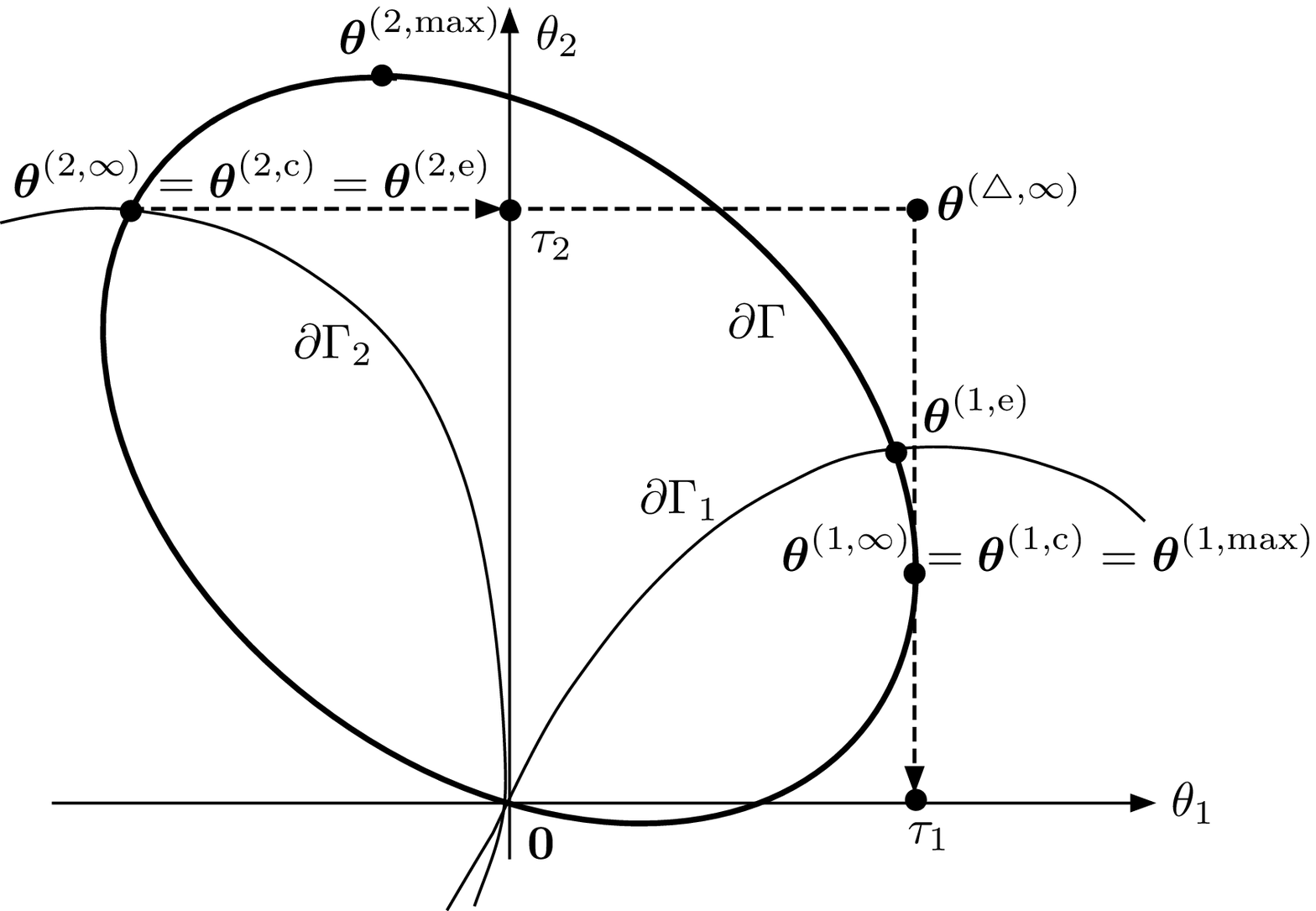}
	\caption{Typical figures for (D1)}
	\label{fig:case 1}
\end{figure}
\begin{figure}[h]
	\centering
	\includegraphics[height=4.3cm]{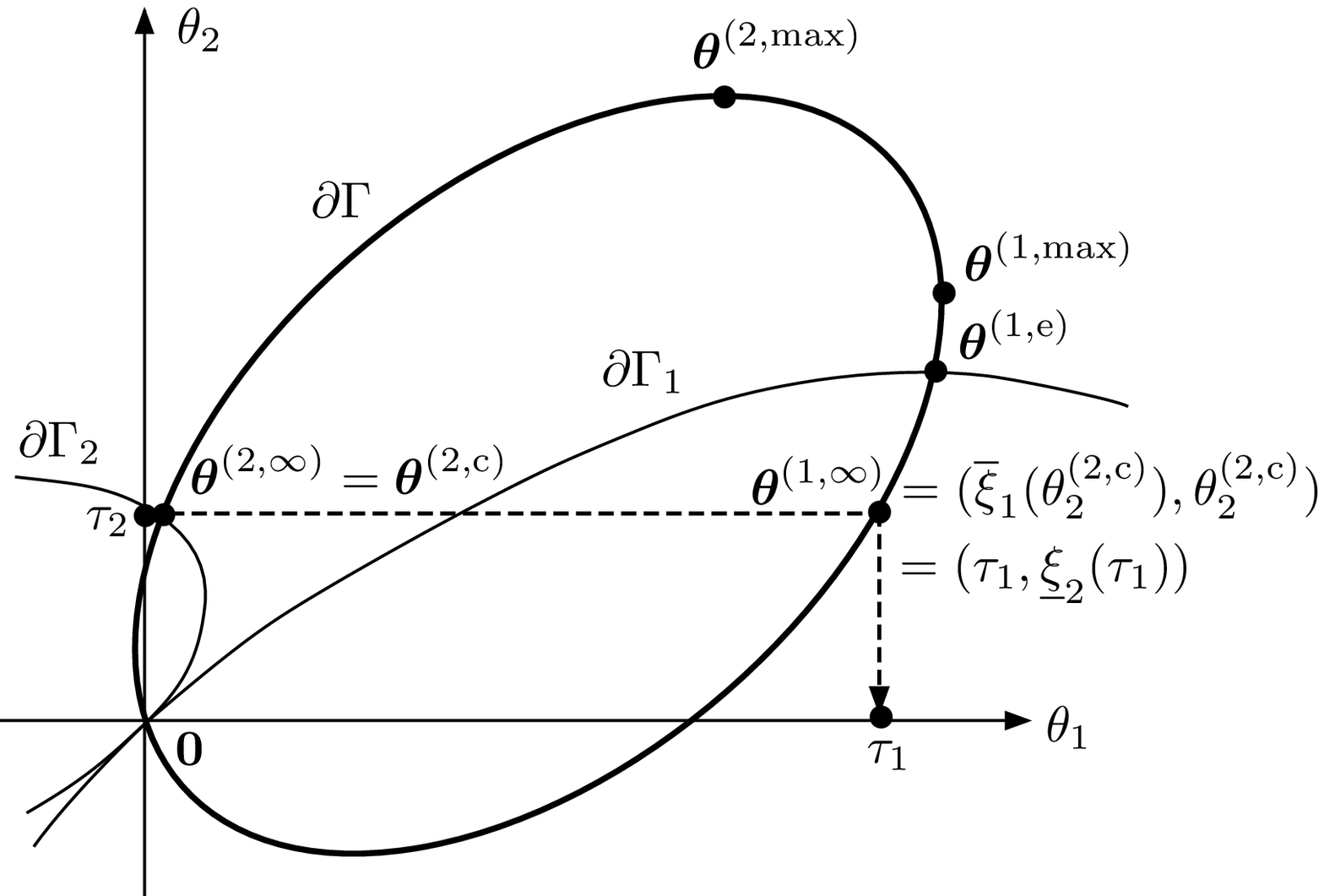}
	\includegraphics[height=4.3cm]{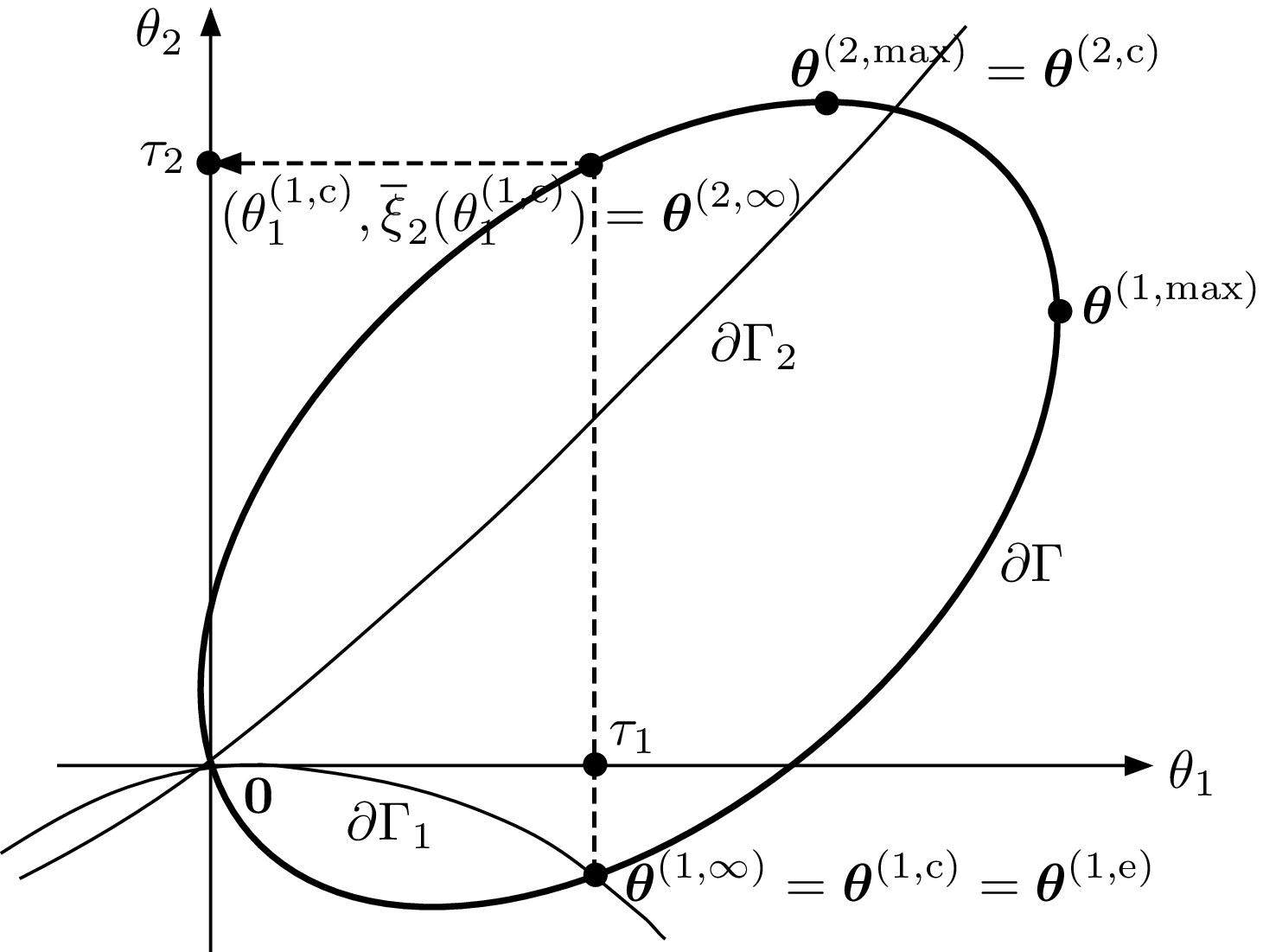}
	\caption{Typical figures for (D2) and (D3)}
	\label{fig:case 2-3}
\end{figure}

  Using this classification, we define vector $\vc{\tau} \equiv (\tau_{1}, \tau_{2})$ as
\begin{eqnarray}
\label{eqn:tau}
  \vc{\tau} = \left \{
\begin{array}{ll}
(\theta_{1}^{(1, \cp)},\theta_{2}^{(2, \cp)})          & \mbox{if (D1) holds},\\
(\ol{\xi}_{1}(\theta_{2}^{(2, \cp)}),\theta_{2}^{(2, \cp)}) & \mbox{if (D2) holds}, \\
(\theta_{1}^{(1, \cp)},\ol{\xi}_{2}(\theta_{1}^{(1, \cp)}))  & \mbox{if (D3) holds}.
\end{array}
\right.
\end{eqnarray}
where $\ol{\xi}_{1}$ and $\ol{\xi}_{2}$ are defined as
\begin{eqnarray*}
  \ol{\xi}_{1}(\theta_{2}) = \max \{ \theta; \gamma(\theta, \theta_{2}) = 1 \}, \qquad \ol{\xi}_{2}(\theta_{1}) = \max \{ \theta; \gamma(\theta_{1}, \theta) = 1 \}.
\end{eqnarray*}
  We will also use the following notation corresponding with them.
\begin{eqnarray*}
  \ul{\xi}_{1}(\theta_{2}) = \min \{ \theta_{1}; \gamma(\theta_{1}, \theta_{2}) = 1 \}, \qquad \ul{\xi}_{2}(\theta_{1}) = \min \{ \theta_{2}; \gamma(\theta_{1}, \theta_{2}) = 1 \}.
\end{eqnarray*}
  
  Then, the domain $\sr{D}$ is obtained as follows.
\begin{theorem}
\label{thr:domain D}
  Under conditions (i)--(iv) and the stability condition, we have
\begin{eqnarray}
\label{eqn:domain D 1}
 \sr{D} = \{\vc{\theta} \in \Gamma_{\max}; \vc{\theta} < \vc{\tau} \},
\end{eqnarray}
  where $\Gamma_{\max} = \{ \vc{\theta} \in \dd{R}^{2}; \exists \vc{\theta}' \in \Gamma, \vc{\theta} < \vc{\theta}' \}$. Furthermore, for $k=1,2$,
\begin{eqnarray}
\label{eqn:domain D 2}
  \tau_{k} = \sup\{ \theta_{k} \ge 0; \varphi_{k}(\theta_{k}) < \infty \}.
\end{eqnarray}
\end{theorem}
\begin{remark}
\label{rem:domain D}
  This result generalizes Theorem 3.1 of \cite{Miya2009} in two respects. First, the domain is fully identified. Secondly, the skip free condition in \cite{Miya2009} is relaxed.
\end{remark}
 
  We next consider the decay rate of $P( \br{\vc{c}, \vc{L}} \ge x)$ as $x$ goes to infinity. In some cases, we also derive its exact asymptotic. From the domain $\sr{D}$ obtained in \thr{domain D}, we can expect that this decay rate is given by
\begin{eqnarray}
\label{eqn:decay rate c}
  \alpha_{\vc{c}} = \sup \{ x \ge 0; x \vc{c} \in \sr{D} \}.
\end{eqnarray}
  One may consider this to be immediate, claiming that the tail decay rate of a distribution on $[0,\infty)$ is obtained by the convergence parameter of its moment generating function. This claim has been used in the literature, but it is not true (see \app{convergence parameter} for a counter example). Thus, we do need to prove \eqn{decay rate c}.

  We simplify notation $\alpha_{\vcn{e}_{k}}$ as $\alpha_{k}$ for $\vcn{e}_{1} = (1,0)$ and $\vcn{e}_{2} = (0,1)$. Note that
\begin{eqnarray}
\label{eqn:alpha k}
  && \alpha_{1} = \left\{\begin{array}{ll}
  \tau_{1}, \quad & \ol{\xi}_{2}(\tau_{1}) \ge 0,\\
  \beta_{1}, & \ol{\xi}_{2}(\tau_{1}) < 0.
  \end{array} \right.
  \quad
  \alpha_{2} = \left\{\begin{array}{ll}
  \tau_{2}, \quad & \ol{\xi}_{1}(\tau_{2}) \ge 0,\\
  \beta_{2}, & \ol{\xi}_{1}(\tau_{2}) < 0.
  \end{array} \right.
\end{eqnarray}
 where $\beta_{k}$ is the positive solution $x$ of $\gamma(x \vcn{e}_{k}) = 1$ (see \fig{marginal 1} below). We now present asymptotic results in two theorems, which will be proved in the next section.
\begin{figure}[h]
 	\centering
	\includegraphics[height=4.3cm]{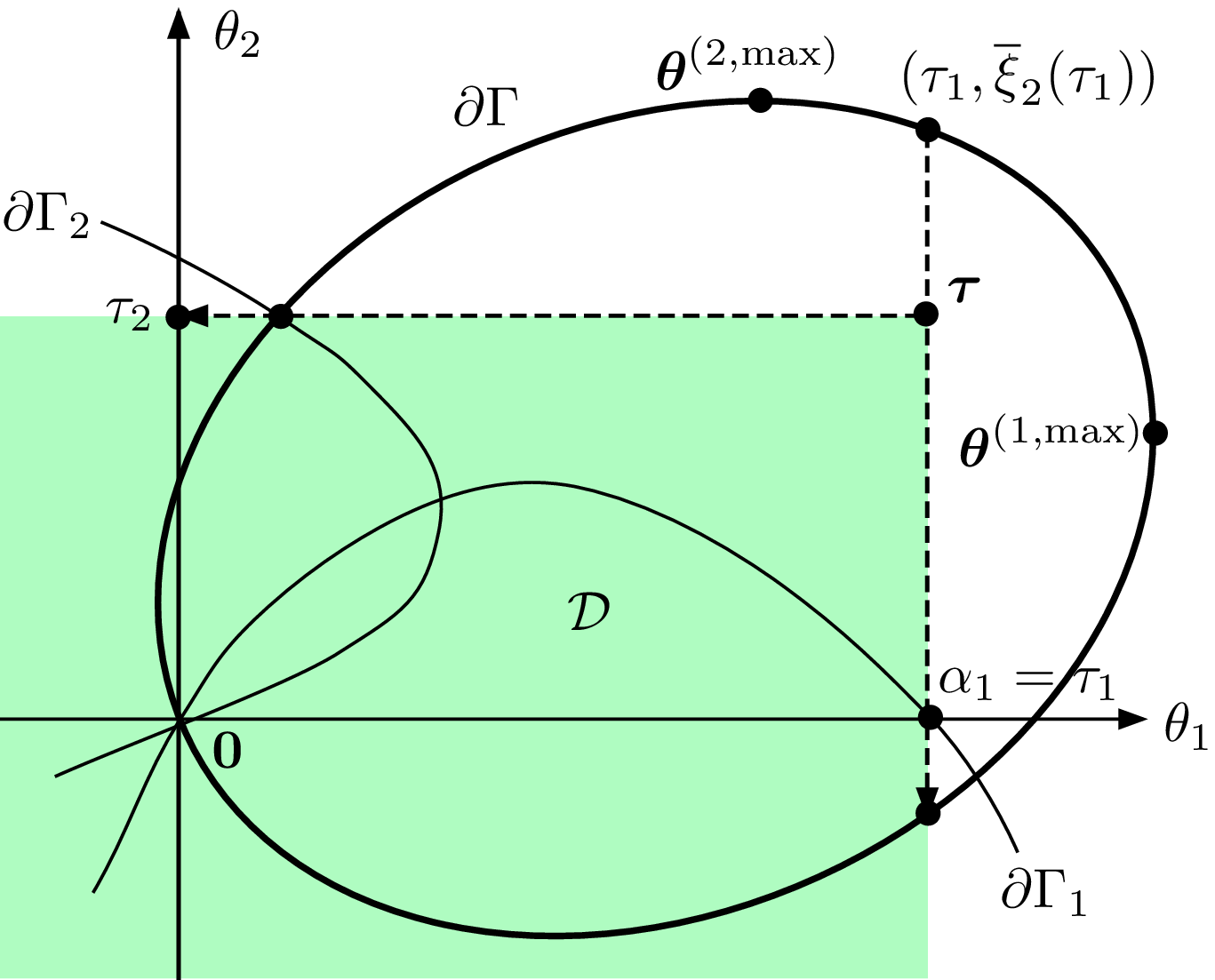} \hspace{2ex}
	\includegraphics[height=4.3cm]{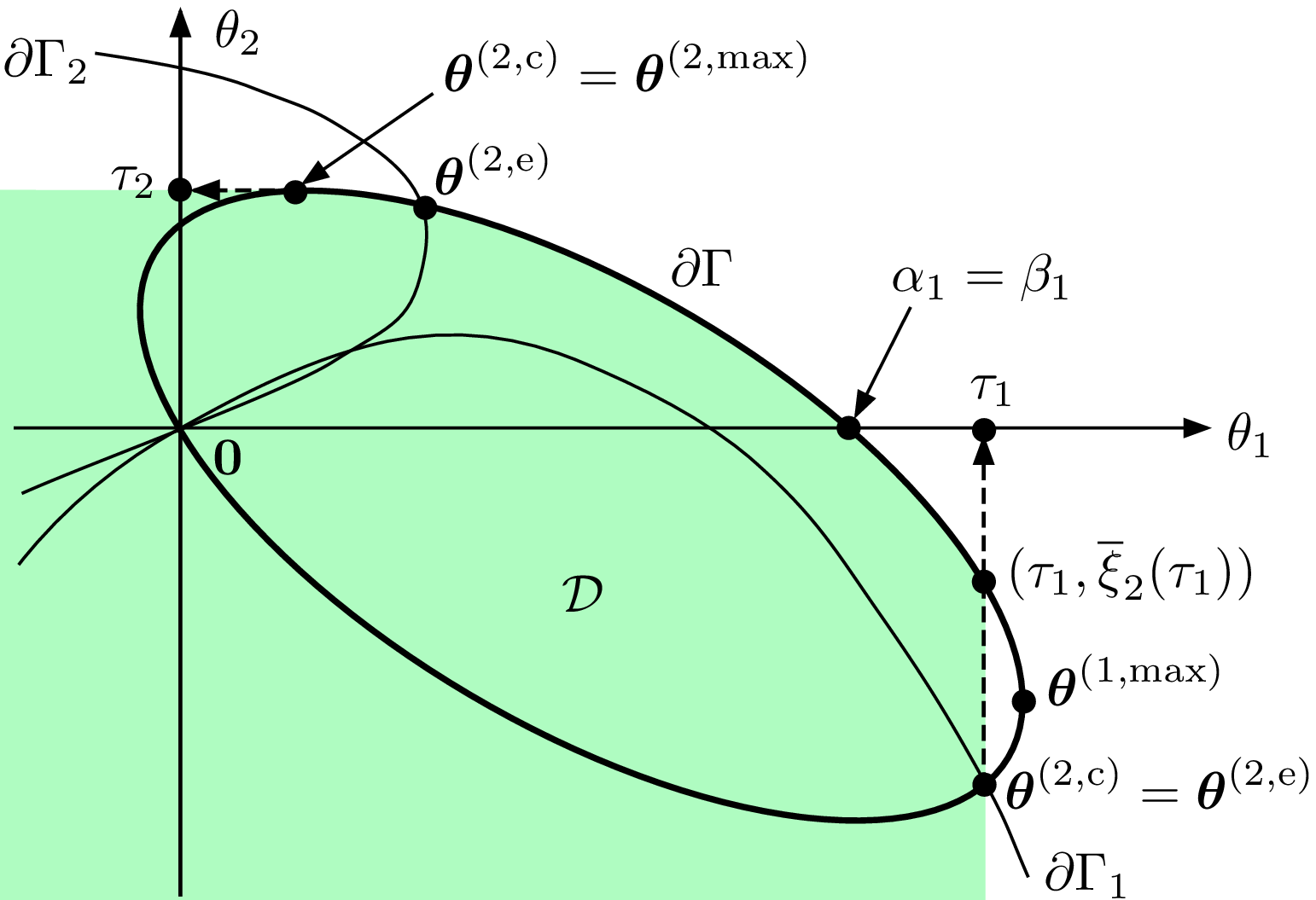}
	\caption{Typical figures for $\alpha_{1} = \tau_{1}$ and $\alpha_{1} = \beta_{1}$ (the colored areas are the domains)}
	\label{fig:marginal 1}
\end{figure}

\begin{theorem}
\label{thr:decay rate 0}
  Under conditions (i)--(iv) and the stability condition, we have,
\begin{eqnarray}
\label{eqn:decay rate 0}
  \lim_{n \to \infty} \frac 1n \log P( L_{k} \ge n, L_{3-k} = i ) = - \tau_{k}, \qquad i \in \dd{Z}_{+}, k=1,2.
\end{eqnarray}
  Furthermore, if $\tau_{k} \ne \theta^{(k, \max)}_{k}$ and if the Markov additive kernel $\{A^{(k)}_{n}; n \ge -1\}$ is 1-arithmetic, then we have the following asymptotics for some constant $b_{ki} > 0$.
\begin{eqnarray}
\label{eqn:exact asymptotic 01}
  \liminf_{n \to \infty} e^{\tau_{k} n} P( L_{k} \ge n, L_{3-k} = i ) \ge b_{ki} , \qquad i \in \dd{Z}_{+}, k=1,2.
\end{eqnarray}
 In particular, for (D1) with $k=1,2$, (D2) with $k=2$ and (D3) with $k=1$, this is refined as
\begin{eqnarray}
\label{eqn:exact asymptotic 02}
  \lim_{n \to \infty} e^{\tau_{k} n} P( L_{k} \ge n, L_{3-k} = i ) = c_{ki} \qquad i \in \dd{Z}_{+} \mbox{ for some $c_{ki} > 0$.}
\end{eqnarray}
\end{theorem}
\begin{remark}
\label{rem:decay rate 0a}
  (a) In the case (D1), if $\tau_{k} = \theta^{(k, \max)}_{k}$ and $\vc{\theta}^{(k, \max)} \ne \vc{\theta}^{(k, \edge)}$, then we can show
\begin{eqnarray}
\label{eqn:exact asymptotic 03}
  \lim_{n \to \infty} e^{\tau_{k} n} P( L_{k} \ge n, L_{3-k} = i ) = 0 .
\end{eqnarray}
  This is immediate from \rem{decay rate 1} (b). We conjecture that \eqn{exact asymptotic 03} is also true for $\tau_{k} = \theta^{(k, \max)}_{k}$ and $\vc{\theta}^{(k, \max)} = \vc{\theta}^{(k, \edge)}$. (b) When we apply the kernel method to the skip free case, $\varphi_{k}(z)$ with complex variable $z$ (or the corresponding generating function) has a branch point at $z = \tau_{k}$ for $\tau_{k} = \theta^{(k, \max)}_{k}$ under the case (D1), which is a dominant singular point of $\varphi_{k}(z)$ (see \cite{KobaMiya2012}). Thus, the rightmost point $\vc{\theta}^{(k, \max)}$ corresponds with a branch point of the complex analysis. This is also the reason why \eqn{exact asymptotic 03} holds because $P( L_{k} = n, L_{3-k} = i )$ has the exact asymptotics of the form $n^{c} e^{-\tau_{k}n}$ with $c = -\frac 12, -\frac 32$ as $n \to \infty$.
\end{remark}

\begin{theorem}
\label{thr:decay rate 1}
  Under the same conditions of \thr{decay rate 0}, we have, for any directional vector $\vc{c} \ge \vc{0}$,
\begin{eqnarray}
\label{eqn:Ldecay rate 1}
 \lim_{x \to \infty} \frac 1x \log P( \br{\vc{c}, \vc{L}} \ge x ) = - \alpha_{\vc{c}},
\end{eqnarray}
  where we recall that $\alpha_{\vc{c}} = \sup \{ x \ge 0; x \vc{c} \in \sr{D} \}$. Further assume that $\gamma(\alpha_{\vc{c}} \vc{c}) = 1$, $\gamma_{k}(\alpha_{\vc{c}} \vc{c}) \ne 1$ and $\alpha_{\vc{c}} c_{k} \ne \tau_{k}$ for $k=1,2$. Then, we have, for some positive constant $b_{\vc{c}}$,
\begin{eqnarray}
\label{eqn:exact asymptotic 2}
  \lim_{x \to \infty} e^{\alpha_{\vc{c}}x} P( \br{\vc{c}, \vc{L}} \ge x ) = b_{\vc{c}},
\end{eqnarray}
  where $x$ runs over $\{\delta n; n \in \dd{Z}_{+}\}$ if $\br{\vc{c}, \vc{X}^{(+)}}$ is $\delta$-arithmetic for some $\delta > 0$, while it runs over $\dd{R}_{+}$ otherwise, which is equivalent to $c_{2} \ne 0$ and $c_{1}/c_{2}$ is not rational.
\end{theorem}

\begin{remark}
\label{rem:decay rate 1}
  (a) Since $\alpha_{k}$ may be less than $\tau_{k}$, the decay rate of $P( L_{k} \ge n )$ may be different from that of $P( L_{k} \ge n, L_{3-k} = 0 )$.
  (b) Similar to \rem{decay rate 0a},  if $\alpha_{k} = \theta^{(k, \max)}_{k}$ and $\vc{\theta}^{(k, \max)} \ne \vc{\theta}^{(k, \edge)}$, then we can show that $\varphi(\vc{\theta}^{(k, \max)}) < \infty$ using the expression \eqn{analytic zc} in the proof of this theorem. This implies that
\begin{eqnarray}
\label{eqn:exact asymptotic 23}
  \lim_{n \to \infty} e^{\alpha_{k}n} P( L_{k} \ge n ) = 0 .
\end{eqnarray}
  This obviously implies \eqn{exact asymptotic 03}. (c)
  If the jumps are skip free, then finer exact asymptotics are obtained for coordinate directions in \cite{KobaMiya2012}, which partially uses the kernel method discussed in \sectn{Moment}. One may think to apply the same approach as in \cite{KobaMiya2012}. Unfortunately, this approach does not work because of the same reason discussed in \rem{kernel method}.
\end{remark}

We now consider how the decay rates $\tau_{k}$ and $\alpha_{\vc{c}}$ are influenced by the modeling primitives. Obviously, if $\gamma(\vc{\theta})$, $\gamma_{1}(\vc{\theta})$ and $\gamma_{2}(\vc{\theta})$ are increased by changing the modeling primitives, then the open sets $\Gamma$, $\Gamma_{1}$ and $\Gamma_{2}$ are diminished. Hence, we have

\begin{lemma}
\label{lem:monotonicity 1}
  Under the assumptions of \thr{decay rate 0}, if the distributions of $\vc{X}^{(+)}$, $\vc{X}^{(1)}$ and $\vc{X}^{(2)}$ are changed to increase $\gamma(\vc{\theta})$, $\gamma_{1}(\vc{\theta})$ and $\gamma_{2}(\vc{\theta})$ for each fixed $\vc{\theta} \in \dd{R}^{2}$, then the decay rates $\tau_{k}$ and $\alpha_{\vc{c}}$ are decreased.
\end{lemma}

Here, decreasing and increasing are used in the weaker sense. This convention will be used throughout the paper. To materialize the monotone property in \lem{monotonicity 1}, we use the following stochastic order for random vectors.

\begin{definition}[M\"{u}ller and Stoyan \cite{MullStoy2002}]
\label{dfn:linear cx}
    (a) For random variables $X$ and $Y$, the distribution of $X$ is said to be less than the distribution of $Y$ in convex order, which is denoted by $X \le_{cx} Y$, if, for any convex function $f$ from $\dd{R}$ to $\dd{R}$,
\begin{eqnarray*}
  E( f(X)) \le E( f(Y)),
\end{eqnarray*}
as long as the expectations exist. (b) For two dimensional random vectors $\vc{X}$ and $\vc{Y}$, if $\br{\vc{\theta},\vc{X}} \le_{cv} \br{\vc{\theta},\vc{Y}}$ for all $\vc{\theta} \in \dd{R}^{2}$, then the distribution of $\vc{X}$ is said to be less than the distribution of $\vc{Y}$ in linear convex order, which is denoted by $\vc{X} \le_{lcx} \vc{Y}$.
\end{definition}

For  linear convex order, Koshevoy and Mosler \cite{KoshMosl1998} give several conditions to be equivalent (see also \cite{Scar1998}). Among them, the Strassen's characterization for convex order visualizes the variability of this order (see Theorem 2.6.6 of \cite{MullStoy2002}). That is, $\vc{X} \le_{lcx} \vc{Y}$ if and only if there is a random variable $U_{\vc{\theta}}$ for each $\vc{\theta} \in \dd{R}^{2}$ such that $E(U_{\vc{\theta}}|\br{\vc{\theta},\vc{X}}) = 0$ and
\begin{eqnarray}
\label{eqn:linear cx}
  \br{\vc{\theta},\vc{Y}} \simeq \br{\vc{\theta},\vc{X}} + U_{\vc{\theta}},
\end{eqnarray}
where $\simeq$ denotes the equality in distribution.

Since $e^{x}$ is a convex function, the following fact is immediate from Theorems \thrt{decay rate 0} and \thrt{decay rate 1} and \lem{monotonicity 1}.

\begin{corollary}
\label{cor:monotonicity 2}
  If the distributions of $\vc{X}^{(+)}$, $\vc{X}^{(1)}$ and $\vc{X}^{(2)}$ are increased in linear convex order, then the decay rate $\tau_{k}$ and $\alpha_{\vc{c}}$ are decreased for $k=1,2$ and any direction $\vc{c} \ge \vc{0}$, where the stability condition is unchanged due to \eqn{linear cx}.
\end{corollary}

This lemma shows how the decay rates are degraded by increasing the variability of the transition jumps $\vc{X}^{(+)}$, $\vc{X}^{(1)}$ and $\vc{X}^{(2)}$.

\section{Proofs of the theorems}
\label{sect:Proofs}
\setnewcounter

The aim of this section is to prove Theorems \thrt{domain D}, \thrt{decay rate 0} and \thrt{decay rate 1}. Before their proofs, we prepare two sets of auxiliary results. One is another representation of the stationary distribution using a Markov additive process. Another is an iteration algorithm for deriving the convergence domain $\sr{D}$.

\subsection{Occupation measure and Markov additive process}
\label{sect:Prelimiaries}

  We first represent the stationary distribution using a so called censoring. That is, the stationary probability of each state is computed by the expected number of visiting times to that state between the returning times to the boundary or its face. The set of these expected numbers is called an occupation measure. We formally introduce them.
  
  Let $U$ be a subset of $S$, and let $\sigma^{U} = \inf\{ \ell \ge 1; \vc{L}(\ell) \in U \}$. For this $U$, define the distribution $g^{U}$ of the first return state and the occupation measure $h^{U}$ as
\begin{eqnarray*}
 && g^{U}(\vc{i}, \vc{j}) = P\left(\left. \vc{L}(\sigma^{U}) = \vc{j}, \sigma^{U} < \infty \right| \vc{L}(0) = \vc{i} \right), \qquad \vc{i}, \vc{j} \in U,\\
 && h^{U}(\vc{i}, \vc{j}) = E\left(\left. 1(\sigma^{U} < \infty) \sum_{\ell=0}^{\sigma^{U}-1} 1(\vc{L}(\ell) = \vc{j}) \right| \vc{L}(0) = \vc{i} \right), \qquad \vc{i}, \vc{j} \in S \setminus U.
\end{eqnarray*}
  Let $G^{U}$ and $H^{U}$ be the matrices whose $(\vc{i}, \vc{j})$ entries are $g^{U}(\vc{i}, \vc{j})$ and $h^{U}(\vc{i}, \vc{j})$, respectively. $H^{U}$ may be also considered as a potential kernel (see, e.g., Section 2.1 of \cite{Numm1984} for this kernel). Note that $G^{U}$ is stochastic since $\{\vc{L}(\ell)\}$ has the stationary distribution $\nu$. Let $\nu^{U}$ be the stationary distribution of $G^{U}$, which is uniquely determined up to a constant multiplier by $\nu^{U} G^{U} = \nu^{U}$. By $E_{\nu^{U}}(\sigma^{U})$, we denote the expectation of $\sigma^{U}$ with respect to $\nu^{U}$. By the existence of the stationary distribution, $E_{\nu^{U}}(\sigma^{U})$ is finite. Then, as discussed in Section 2 of \cite{MiyaZwar2012}, it follows from censoring on the set $U$ that
\begin{eqnarray}
\label{eqn:occupation measure 1}
 && \nu(\vc{j}) = \frac 1{E_{\nu^{U}}(\sigma^{U})} \sum_{\vc{i} \in U} \nu^{U}(\vc{i}) \sum_{\vc{k} \in S \setminus U} p(\vc{i}, \vc{k}) h^{U}(\vc{k}, \vc{j}) , \qquad \vc{j} \in S \setminus U,
\end{eqnarray}
where $p(\vc{i}, \vc{k}) = P \left.\left( \vc{L}(1) = \vc{k} \right| \vc{L}(0) = \vc{i} \right)$.

  We here need the distribution $\nu^{U}$ because $U$ may not be singleton. The censored process is particularly useful when the occupation measure is simple. For example, if we choose $\partial S$ for $U$, then $h^{U}(\vc{i}, \vc{j})$ is obtained from the random walk $\{\vc{Y}(\ell)\}$, which is simpler than $\{\vc{L}(\ell)\}$.
  
  We next choose $U = S_{0} \cup S_{2}$ for $H^{U} \equiv \{h^{U}(\vc{i}, \vc{j}); \vc{i}, \vc{j} \in S \setminus U\}$. This occupation measure has been used to find tail asymptotics of the marginal stationary distribution (see, e.g., \cite{KobaMiyaZhao2010,Miya2009,MiyaZhao2004}). For each $m, n \ge \ell \ge 1$, we denote the $\dd{Z}_{+} \times \dd{Z}_{+}$ matrix whose $(i,j)$ entry is $h^{U_{\ell}}((m,i), (n,j))$ by $H^{(1,\ell)}_{mn}$, where $U_{\ell} = \{0,1,\ldots, \ell-1\} \times \dd{Z}_{+}$. Since this $H^{(1,m)}_{m(m+n)}$ does not depend on $m \ge \ell$, we simply denote it by $H^{(1)}_{n}$ for $n \ge 0$, where $H^{(1)}_{0}$ is the identity matrix.
  
  We need further notation. For nonnegative integers $m, n \ge 0$, let
\begin{eqnarray}
\label{eqn:R matrix 1}
  \lefteqn{r^{(1)}((m,i),(m+n,j))} \nonumber\\
  && = \sum_{\ell=1}^{\infty} P\big( \vc{L}(\ell) = (m+n,j), \nonumber\\
  && \hspace{12ex} m < L_{1}(\ell) \le \min(L_{1}(1), \ldots, L_{1}(\ell-1)) | \vc{L}(0) = (m,i) \big). 
\end{eqnarray}
  For $m \ge 1$, $r^{(1)}((m,i),(m+n,j))$ does not depend on $m$, so, for each $n \ge 1$, we denote the matrix whose $(i,j)$ entry is $r^{(1)}((m,i),(m+n,j))$ by $R^{(1)}_{n}$ for $n \ge 1$. On the other hand, for $m=0$, we denote the corresponding matrix by $R^{(1)}_{0n}$ for $n \ge 1$.
  
  For each $n \ge 0$, let $\vc{\nu}_{n}^{(1)}$ be the vector whose $i$-th entry is $\nu(n,i)$. Similar to \eqn{occupation measure 1}, it follows from censoring with respect to $U = U_{n}$ that
\begin{eqnarray}
\label{eqn:renewal equation 1}
  \vc{\nu}^{(1)}_{n} = \vc{\nu}^{(1)}_{0} R^{(1)}_{0n} + \sum_{\ell=1}^{n-1} \vc{\nu}^{(1)}_{\ell} R^{(1)}_{n-\ell}, \qquad n \ge 1.
\end{eqnarray}
  Using the well known identity that
\begin{eqnarray*}
  \sum_{n=0}^{\infty} s^{n} H^{(1)}_{n} = \left( I - \sum_{n=1}^{\infty} s^{n} R^{(1)}_{n} \right)^{-1},
\end{eqnarray*}
  \eqn{renewal equation 1} can be written as
\begin{eqnarray}
\label{eqn:renewal equation 2}
  \vc{\nu}^{(1)}_{n} = \vc{\nu}^{(1)}_{0} \sum_{\ell=1}^{n} R^{(1)}_{0\ell} H^{(1)}_{n-\ell} \ge \vc{\nu}^{(1)}_{1} H^{(1)}_{n-1}, \qquad n \ge 1.
\end{eqnarray}
  These formulas are also found in \cite{MiyaZhao2004,MiyaZwar2012}. Thus, $\vc{\nu}^{(1)}_{n}$ is given in terms of $\vc{\nu}^{(1)}_{0}$, $\{R^{(1)}_{0n}\}$ and $\{H^{(1)}_{n}\}$. These expressions will be also useful.
  
  We next consider a Markov additive process generated from the reflecting process $\{\vc{L}(\ell)\}$ by removing the boundary transitions on $S_{0} \cup S_{2}$, and present a useful identity, called the Wiener-Hopf factorization.
  
  Denote this Markov additive process by $\{\vc{Z}^{(1)}(\ell)\}$. Specifically, let $\{ A_{n}^{(1)}; n \ge -1\}$ be its Markov additive kernel, that is, $(i,j)$ entry of matrix $A_{n}^{(1)}$ is defined as
\begin{eqnarray*}
  [A^{(1)}_{n}]_{ij} = P(\vc{Z}^{(1)}(\ell+1) = (m+n,j) | \vc{Z}^{(1)}(\ell) = (m,i)), \qquad n, m \in \dd{Z}, i,j \in \dd{Z}_{+},
\end{eqnarray*}
 Obviously, the right-hand side is independent of $m$, and
\begin{eqnarray*}
  [A^{(1)}_{n}]_{ij} = \left\{\begin{array}{ll}
  P(\vc{X}^{(+)} = (n, j - i)), \qquad &  i \ge 1, j \ge i - 1, \\
  P(\vc{X}^{(1)} = (n, j - i)), \qquad &  i = 0, j \ge 0,
  \end{array} \right.
\end{eqnarray*}
  for $n \ge -1$, where $n$ and $i,j$ are referred to as level and background states, respectively. Define matrix $A_{*}^{(1)}(t)$ for $t > 0$ by
\begin{eqnarray*}
  [A^{(1)}_{*}(t)]_{ij} = \sum_{n=-1}^{\infty} t^{n} [A^{(1)}_{n}]_{ij}, \qquad t > 0, i,j \in \dd{Z}_{+},
\end{eqnarray*}
  as long as they exist. This matrix is called a matrix generating function of the transition kernel $A_{n}^{(1)}$. Similarly, we denote the matrix generating function $R^{(1)}_{n}$ by $R^{(1)}_{*}(t)$. Note that \eqn{R matrix 1} is valid also for $\vc{Z}^{(1)}(\ell)$ instead of $\vc{L}(\ell)$. Hence, $R^{(1)}_{*}(t)$ is well defined for the Markov additive process $\{\vc{Z}^{(1)}(\ell)\}$. Then, we have the Wiener-Hopf factorization:
\begin{eqnarray}
\label{eqn:WHF 1}
  I - A_{*}^{(1)}(t) = (I - R_{*}^{(1)}(t)) (I - G^{(1)}_{*}(t)),
\end{eqnarray}
  where $G^{(1)}_{*}(t)$ is the $\dd{Z} \times \dd{Z}$ matrix whose $(i,j)$ entry is given by
\begin{eqnarray*}
  g^{(1)}_{*}(t)(i,j) = E\left( \left. t^{Z^{(1)}_{1}(\sigma^{-0}_{1})} 1(Z^{(1)}_{2}(\sigma^{-0}_{1}) = j) \right| Z^{(1)}_{2}(0) = i \right), 
\end{eqnarray*}
  where $\sigma^{-0}_{1} = \inf \{\ell \ge 1; Z^{(1)}_{1}(\ell) \le 0\}$. The factorization \eqn{WHF 1} goes back to \cite{ArjaSpee1973}, but $t$ is limited to a complex number satisfying $|t| \le 1$ for simplicity. The present version is valid as long as $A^{(1)}_{*}(t)$ exists. This fact is formally proved in \cite{MiyaZwar2012}, but has been often ignored in the literature (see, e.g., \cite{MiyaZhao2004}), but it is crucial to consider the tail asymptotic problems.
  
\subsection{Iteration algorithm and bounds}
\label{sect:Iteration}

  To prove \thr{domain D}, we prepare a series of lemmas. A main body of the proof will be given in the next section. We first construct a sequence in the closure of $\sr{D}$, denoted by $\ol{\sr{D}}$ that converges to $\vc{\tau}$ of \eqn{tau}. For this, we rewrite \eqn{stationary equation 1} as
\begin{eqnarray}
\label{eqn:stationary equation 2}
  \lefteqn{(1 - \gamma(\vc{\theta}))\varphi_{+}(\vc{\theta}) + (1 - \gamma_{k}(\vc{\theta})) \varphi_{k}(\theta_{k}) }\nonumber\\
  && \hspace{10ex} = (\gamma_{3-k}(\vc{\theta}) - 1) \varphi_{3-k}(\theta_{3-k}) + (\gamma_{0}(\vc{\theta}) - 1) \varphi_{0}(0).
\end{eqnarray}
  We expand the confirmed region of $\varphi(\vc{\theta}) < \infty$ using \eqn{stationary equation 1} and \eqn{stationary equation 2} with help of \lem{stationary equation 2}.
 
  Let $\Gamma_{k}^{(0)} = \{ \vc{\theta} \in \Gamma_{k} \cap \Gamma; \theta_{3-k} \leq 0 \}$ for $k=1,2$. Note that $\Gamma_{k}^{(0)}$ is not empty by \lem{geometric 1}. Obviously, $\varphi_{3-k}(\theta_{3-k})$ is finite for $\vc{\theta} \in \Gamma_{k}^{(0)}$. Hence, by \lem{stationary equation 2}, $\varphi_{+}(\vc{\theta})$ and $\varphi_{k}(\theta_{k})$ are finite for $\vc{\theta} \in \Gamma_{k}^{(0)}$. Thus, $\varphi(\vc{\theta})$, $\varphi_{1}(\theta_{1})$ and $\varphi_{2}(\theta_{2})$ are finite for $\vc{\theta} \in \Gamma_{1}^{(0)} \bigcup \Gamma_{2}^{(0)}$. We define $\vc{\theta}^{(\triangle, 0)} \equiv (\theta_{1}^{(\triangle, 0)}, \theta_{2}^{(\triangle, 0)})$ by 
\begin{eqnarray*}
  \theta_{k}^{(\triangle, 0)} = \sup \{ \theta_{k}; (\theta_{1},\theta_{2}) \in \Gamma_{k}^{(0)} \}, \qquad k=1,2. 
\end{eqnarray*}
  By \lem{geometric 1}, at least one of $\theta_{1}^{(\triangle, 0)}$ and $\theta_{2}^{(\triangle, 0)}$ is positive under condition (iv). 

  We next define $\vc{\theta}^{(\triangle, n)} = (\theta_{1}^{(\triangle, n)},\theta_{2}^{(\triangle, n)})$ for $n \geq 1$ by
\begin{eqnarray*}
  \theta_{k}^{(\triangle, n)} = \sup \{ \theta_{k}; \vc{\theta} \in \Gamma_{k} \cap \Gamma, \theta_{3-k} \leq \theta_{3-k}^{(\triangle, n-1)} \} .
\end{eqnarray*}
Then, $\vc{\theta}^{(\triangle, n)}$ is nondecreasing in $n$, and $\vc{\theta}^{(\triangle, n)} \le \vc{\theta}^{\max}$ from our definition. Thus, the sequence $\vc{\theta}^{(\triangle, n)}$ converges to a finite positive vector because $\vc{\theta}^{(\triangle, 1)} > \vc{0}$. Denote this limit by $\vc{\theta}^{(\triangle, \infty)} \equiv (\theta_{1}^{(\triangle, \infty)},\theta_{2}^{(\triangle, \infty)})$. Because $\Gamma_{k} \cap \Gamma$ is a bounded convex set, we can see that
\begin{eqnarray}
\label{eqn:theta-infty}
  \theta_{k}^{(\triangle, \infty)} = \sup \{ \theta_{k}; \vc{\theta} \in \Gamma_{k} \cap \Gamma, \theta_{3-k} \leq \theta_{3-k}^{(\triangle, \infty)} \}, \qquad k = 1,2.
\end{eqnarray}
  This can be considered as a fixed point equation, and we have the following solution, which is proved in \app{tau}.

\begin{lemma}
\label{lem:tau}
  $\vc{\theta}^{(\triangle, \infty)} = \vc{\tau}$, and $\varphi(\vc{\theta})$ is finite for all $\vc{\theta} < \vc{\theta}^{(k, \infty)}$ for $k=1,2$, where
\begin{eqnarray*}
  \vc{\theta}^{(1, \infty)} = (\theta_{1}^{(\triangle, \infty)}, \ul{\xi}_{2}(\theta_{1}^{(\triangle, \infty)})), \qquad \vc{\theta}^{(2, \infty)} = (\ul{\xi}_{1}(\theta_{2}^{(\triangle, \infty)}), \theta_{2}^{(\triangle, \infty)}).
\end{eqnarray*}
\end{lemma}

  We need two more lemmas. Recall that $\vc{c} \in \dd{R}^{2}$ is called a direction vector if $\| \vc{c} \| = 1$. Let $\vc{1}$ be the vector all of whose entries are unit. The dimension of $\vc{1}$ is either 2 or $\infty$, which can be distinguished in the context of its usage.
  
\begin{lemma}
\label{lem:infinite domain}
  Let $\Delta(\vc{a}) = \{\vc{x} \in \dd{R}^{2}; \vc{0} \le \vc{x} < \vc{a} \}$ for $\vc{a} \ge \vc{1}$. Then, under conditions (i), (ii) and (iii), we have, for any direction vector $\vc{c} > \vc{0}$ and any $\vc{a} \ge \vc{1}$,
\begin{eqnarray}
\label{eqn:lower bound 1}
  \liminf_{x \to \infty} \frac 1{x} \log P(\vc{L} \in x \vc{c} + \Delta(\vc{a})) \ge - \sup\{ \br{\vc{\theta}, \vc{c}}; \gamma(\vc{\theta}) \le 1 \},
\end{eqnarray}
  and therefore $\varphi(\vc{\theta})$ is infinite for $\vc{\theta} \not\in \ol{\Gamma}_{\max}$, where $\ol{\Gamma}_{\max}$ is the closure of $\Gamma_{\max}$.
\end{lemma}
\begin{remark}
\label{rem:infinite domain}
  This lemma is valid without condition (iv). Hence, it can be used for $E(X^{(+)}_{1}) = E(X^{(+)}_{2}) = 0$. In this case, the right side of \eqn{lower bound 1} is zero, so the stationary distribution $\nu$ cannot have a light tail.
\end{remark}

The first part of \lem{infinite domain} is obtained in Theorem 3.1 of \cite{BoroMogu2001} (see also Theorem 1.6 there). However, these theorems use Theorem 1.2 there, and its proof are largely omitted for the lower bound. Thus, the results are not well accessible, so we give its proof in \app{infinite domain}. We need one more lemma.
\begin{lemma}
\label{lem:lower bound c}
 For each $k=1,2$, 
\begin{eqnarray}
\label{eqn:LD lower bound 1}
  \liminf_{n \to \infty} \frac 1n \log P(L_{1} \ge n, L_{2} = i) \ge - \theta^{(1,\cp)}_{1}, \qquad i \in \dd{Z}_{+}, 
\end{eqnarray}
 and therefore $\theta_{k} > \theta^{(k,\cp)}_{k}$ implies $\varphi_{k}(\theta_{k}) = \infty$, and therefore $\varphi(\vc{\theta}) = \infty$.
 
\begin{proof}
  By symmetry, it is sufficient to prove the lemma for $k=1$. From \eqn{renewal equation 2}, we have
\begin{eqnarray*}
  P(L_{1} \ge n, L_{2} = i) = \sum_{\ell=n}^{\infty} \left[\vc{\nu}^{(1)}_{n} \right]_{i} \ge \left[\vc{\nu}^{(1)}_{0} R^{(1)}_{01}  \sum_{\ell=n-1}^{\infty} H^{(1)}_{\ell} \right]_{i}, \qquad i \in \dd{Z}_{+}.
\end{eqnarray*}
  Hence, \eqn{LD lower bound 1} follows from Theorem 4.1 of \cite{KobaMiyaZhao2010}.
\end{proof}
\end{lemma}

We are now ready to prove \thr{domain D}.

\subsection{Proof of \thr{domain D}}
\label{sect:Theorem 3.1}

  We first prove that $\{ \vc{\theta} \in \Gamma_{\max}; \vc{\theta} < \vc{\tau} \} \subset \sr{D}$. Since $\vc{\theta} \in \Gamma_{\max}$ implies the existence of $\vc{\theta}' \in \Gamma$ such that $\vc{\theta}' > \vc{\theta}$, this is sufficient to prove that
\begin{eqnarray*}
  \Gamma_{\vc{\tau}} \equiv \{ \vc{\theta} \in \Gamma; \vc{\theta} < \vc{\tau} \} \subset \sr{D}.
\end{eqnarray*}
  For cases (D2) and (D3), this $\Gamma_{\vc{\tau}}$ is a subset of $\{\vc{\theta} \in \dd{R}^{2}; \vc{\theta} < \vc{\theta}^{(2, \infty)} \}$ and $\{\vc{\theta} \in \dd{R}^{2}; \vc{\theta} < \vc{\theta}^{(1, \infty)} \}$, respectively, on which $\varphi(\vc{\theta})$ is finite by \lem{tau}. Hence, we only need to consider the case (D1). In this case, $\varphi_{k}(\theta) < \infty$ for $\theta < \theta^{(k,\infty)}_{k} = \tau_{k}$ $(k=1,2)$, and therefore $\varphi(\vc{\theta}) < \infty$ for $\vc{\theta} \in \Gamma$ by \lem{stationary equation 2}. Thus, we have $\Gamma_{\vc{\tau}} \subset \sr{D}$. Obviously, this also implies that
\begin{eqnarray}
\label{eqn:domain D 2a}
  \tau_{k} = \theta^{(k, \infty)}_{k} \le \sup\{ \theta_{k} \ge 0; \varphi_{k}(\theta_{k}) < \infty \}.
\end{eqnarray}
  
  We next prove that $\sr{D} \subset \{ \vc{\theta} \in \Gamma_{\max}; \vc{\theta} < \vc{\tau} \}$. Because of the symmetric roles of $\theta_{1}$ and $\theta_{2}$, it is sufficient to prove that either $\theta_{1} > \tau_{1}$ or $\gamma(\vc{\theta}) > 1$ with $\vc{\theta} \not\in \Gamma_{\max}$ implies $\varphi(\vc{\theta}) = \infty$. The latter is immediate from \lem{infinite domain}, and therefore we only need to prove that
\begin{eqnarray}
\label{eqn:diverge 1}
  \theta_{1} > \tau_{1} \mbox{ implies } \varphi_{1}(\theta_{1}) = \infty,
\end{eqnarray}
   which together with \eqn{domain D 2a} also verify \eqn{domain D 2}. We prove \eqn{diverge 1} for the cases (D1), (D2) and (D3) separately.
  
  We first consider the case that (D1) or (D3) holds. In this case, $\theta^{(1,\infty)}_{1} = \theta_{1}^{(1, \cp)}$, and therefore \lem{lower bound c} verifies the claim. 
  
  We next consider the case that (D2) holds. In this case, $\theta^{(1, \infty)}_{1} = \ol{\xi}_{1}(\theta_{2}^{(2, \cp)})$, and $\theta^{(2, \cp)}_{2} = \theta^{(2, \infty)}_{2}$. Then, applying the same argument for $\theta_{1}$, we have $\varphi(\vc{\theta}) = \infty$ and $\varphi_{2}(\theta_{2}) = \infty$ for $\theta_{2} > \theta^{(2, \infty)}_{2}$. If $\theta^{(1, \infty)}_{1} = \theta^{(1, \cp)}_{1}$, then \lem{lower bound c} again verifies the claim \eqn{diverge 1}. Hence, we assume that $\theta^{(1, \infty)}_{1} < \theta^{(1, \cp)}_{1}$.
  
  In what follows, we consider the stationary equation \eqn{stationary equation 1} for $\theta_{1} \in (0, \theta^{(1, \infty)}_{1})$. Let $\theta_{2} = \ul{\xi}_{2}(\theta_{1})$. Since $\varphi(\vc{\theta}) < \infty$ for this $\vc{\theta}$, \eqn{stationary equation 1} is valid, and yields
\begin{eqnarray}
\label{eqn:stationary equation 2c}
  \lefteqn{(1 - \gamma_{1}(\theta_{1}, \ul{\xi}_{2}(\theta_{1}))) \varphi_{1}(\theta_{1})} \nonumber \hspace{8ex}\\
 && = (\gamma_{2}(\theta_{1}, \ul{\xi}_{2}(\theta_{1})) - 1) \varphi_{2}(\ul{\xi}_{2}(\theta_{1})) + (\gamma_{0}(\theta_{1}, \ul{\xi}_{2}(\theta_{1})) - 1) \varphi_{0}(0).
\end{eqnarray}
  We increase $\theta_{1}$ up to $\theta^{(1,\infty)}_{1}$ in this equation. Note that we can find $\epsilon_{0} > 0$ such that $\gamma_{1}(\theta_{1}, \ul{\xi}_{2}(\theta_{1})) \ne 1$ and $\gamma_{2}(\theta_{1}, \ul{\xi}_{2}(\theta_{1})) \ne 1$ for $\theta_{1} \in [\theta^{(1,\infty)}_{1} - \epsilon_{0}, \theta^{(1, \cp)}_{1}]$. Since $\ul{\xi}_{2}(\theta^{(1, \infty)}_{1}) = \theta^{(2, \infty)}_{2}$ and $\ul{\xi}_{2}(\theta)$ is increasing for $\theta \in (\theta^{(1,\infty)}_{1} - \epsilon_{0}, \theta^{(1, \cp)}_{1})$, we have, for any $\epsilon \in (0, \epsilon_{0})$,
\begin{eqnarray*}
  \lim_{\theta_{1} \uparrow \theta^{(1, \infty)}_{1} + \epsilon} \varphi_{2}(\ul{\xi}_{2}(\theta_{1})) = \infty.
\end{eqnarray*}
  This and \eqn{stationary equation 2c} verify \eqn{diverge 1}. Hence, the proof of \thr{domain D} is completed.
 
\subsection{Proof of Theorem \thrt{decay rate 0}}
\label{sect:Theorem 3.2}

   We present two lemmas. They provide suitable bounds for the tail probabilities, and will immediately proves \thr{decay rate 0}.
  
\begin{lemma}
\label{lem:LD upper bound 1}
  Under the assumptions of \thr{domain D}, 
\begin{eqnarray}
\label{eqn:LD upper bound 1}
 && \limsup_{n \to \infty} \frac 1n \log P( L_{k} \ge n, L_{3-k} = i ) \le - \tau_{k}, \qquad i \in \dd{Z}_{+}, k=1,2,
\end{eqnarray}
  and, for any directional vector $\vc{c} \ge \vc{0}$,
\begin{eqnarray}
\label{eqn:LD upper bound 2}
 \limsup_{x \to \infty} \frac 1x \log P( \br{\vc{c}, \vc{L}} \ge x ) \le - \sup \{ u \ge 0; u \vc{c} \in \sr{D} \}.
\end{eqnarray}
\begin{proof}
  \eqn{LD upper bound 1} is immediate from \eqn{domain D 1} and \eqn{domain D 2}. To see \eqn{LD upper bound 2}, we use Markov inequality:
\begin{eqnarray*}
  e^{u x} P(\br{\vc{c}, \vc{L}} \ge x) \le E(e^{\br{u\vc{c}, \vc{L}}}), \qquad u \ge 0, n = 0,1,\ldots.
\end{eqnarray*}
  Taking logarithm of both sides, dividing by $x > 0$ and letting $x \to \infty$, we have \eqn{LD upper bound 2}.
\end{proof}
\end{lemma}

\begin{lemma}
\label{lem:LD lower bound 1}
  Under the assumptions of \thr{domain D}, 
\begin{eqnarray}
\label{eqn:lower bound 5}
  \liminf_{n \to \infty} \frac 1n \log P( L_{1} \ge n, L_{2} = i ) \ge - \tau_{1}, \qquad i \in \dd{Z}_{+}.
\end{eqnarray}
  Further assume that the Markov additive kernel $\{A^{(k)}_{n}; n \ge -1\}$ is 1-arithmetic concerning the additive component. If either (D1) with $\tau_{1} = \theta^{(1, \cp)}_{1} < \theta^{(1, \max)}_{1}$ or (D3) holds, then, for some positive vector $\vc{b}$,
\begin{eqnarray}
\label{eqn:reflecting 0}
 && \lim_{n \to \infty} e^{\tau_{1} n} \vc{\nu}^{(1)}_{n} = \vc{b},
\end{eqnarray}
  while, if (D2) with $\tau_{1} < \theta^{(1, \cp)}_{1}$ holds, then, for some constants $b' > 0$,
\begin{eqnarray}
\label{eqn:reflecting 1}
 && \liminf_{n \to \infty} e^{\tau_{1} n} \vc{\nu}^{(1)}_{n} \ge b' \vc{x},
\end{eqnarray}
where the limit is taken component-wise and $\vc{x}$ is the left invariant vector of $A_{*}^{(1)}(\tau_{1})$. 
\end{lemma}

\begin{proof}
  If $\tau_{1} = \theta^{(1, \max)}_{1}$, then we obviously have \eqn{lower bound 5} by \lem{infinite domain}. Otherwise, \eqn{lower bound 5} follows from \eqn{reflecting 0} and \eqn{reflecting 1} or their $\delta$-arithmetic versions for each positive integer $\delta$. Thus, it remains to prove \eqn{reflecting 0} and \eqn{reflecting 1}.
  
  We first consider the case where either (D1) with $\theta^{(1, \cp)}_{1} < \theta^{(1, \max)}_{1}$ or (D3) holds. In this case, we have $\tau_{1} = \theta^{(1, \edge)}_{1} = \theta^{(1, \cp)}_{1}$ and therefore $\tau_{1} < \theta^{(1, \max)}_{1}$ by the definition \eqn{tau} of $\tau_{1}$. By Lemma 4.2 of \cite{KobaMiyaZhao2010}, we already know that $A_{*}^{(1)}(e^{\theta^{(1, \edge)}_{1}})$ is positive recurrent. Furthermore, $\varphi_{2}(\theta_{2}) < \infty$ for $\theta_{2} < \tau_{2}$ and $\gamma_{1}(\vc{\theta}^{(1, \edge)}) = 1$. Hence, using the notation $\vc{\nu}^{(1)}_{n}$ of \sectn{Double}, we can verify all the conditions of Theorem 4.1 of \cite{MiyaZhao2004} since $\{A^{(k)}_{n}\}$ is 1-arithmetic. Thus, we get \eqn{reflecting 0}. 
  
  We next consider the case where (D2) with $\tau_{1} < \theta^{(1, \cp)}_{1}$ holds. We use the same idea which is applied to the double QBD process, a special case of the double $M/G/1$-type process, in \cite{Miya2009} (see Proposition 3.1 there and Lemma 2.2.1 of \cite{LiMiyaZhao2007}). However, the skip free condition for the reflecting process is crucial in the arguments of \cite{LiMiyaZhao2007,Miya2009}, which is not satisfied for the present model. Thus, we need some more ideas.
  
  Similar to the case (D3) for $k=1$, (D2) implies that $\tau_{2} = \theta_{2}^{(2, \edge)}$ and, for a positive constant $b^{(2)}_{\ell}$ for each $\ell \in \dd{Z}_{+}$,
\begin{eqnarray*}
  \lim_{i \to \infty} e^{\tau_{2} i} \nu(\ell, i) = b^{(2)}_{\ell}.
\end{eqnarray*}
   Furthermore, let $x_{i}$ be the $i$-th entry of the left invariant vector $\vc{x}$ of $A_{*}^{(1)}(e^{\tau_{1}})$, where $\tau_{2} = \ul{\xi}_{1}(\tau_{1})$, then it follows from Theorem A.1 and (A.11) of \cite{KobaMiyaZhao2010} that the complex variable generating function $x_{*}(z) \equiv \sum_{\ell=0}^{\infty} z^{\ell} x_{\ell}$ has a simple pole at $z = \tau_{2}$ and no other pole on $|z| = \tau_{2}$ because of the 1-arithmetic condition, and therefore, for some positive constant $c_{0}$,
\begin{eqnarray}
\label{eqn:x i convergence}
  \lim_{i \to \infty} e^{\tau_{2} i} x_{i} = c_{0}.
\end{eqnarray}
  Hence, we have
\begin{eqnarray}
\label{eqn:pi to x}
 && \lim_{i \to \infty} \vc{\nu}_{\ell}^{(1)}(i) x_{i}^{-1} = c_{0 }\lim_{i \to \infty} e^{\tau_{2} i} \nu(\ell, i) = c_{0} b^{(2)}_{\ell}, \qquad \ell \in \dd{Z}_{+}.
\end{eqnarray}

  Recall the Markov additive  process $\{\vc{Z}^{(1)}(\ell)\}$, which is introduced in \sectn{Prelimiaries}. We now change the measure of this process  using $\tau_{1}$ and $\vc{x}$ in such a way that the new kernel is defined as
\begin{eqnarray}
\label{eqn:twisted A 1}
  \tilde{A}^{(1)}_{n} =\Delta_{\vc{x}}^{-1} (e^{\tau_{1} n} A_{n}^{(1)})^{\rs{t}} \Delta_{\vc{x}}, \qquad n \ge -1,
\end{eqnarray}
  where $\Delta_{\vc{x}}$ is the diagonal matrix whose $i$-th diagonal entry is $x_{i}$. Obviously, $\tilde{A}^{(1)}_{n}(i,j)$ is a non-defective Markov additive kernel. We denote the Markov additive process with this kernel by $\{\tilde{\vc{Z}}^{(1)}(\ell)\} \equiv \{(\tilde{Z}^{(1)}_{1}(\ell), \tilde{Z}^{(1)}_{2}(\ell)) \}$. Let
\begin{eqnarray*}
  \tilde{A}^{(1)} \equiv \sum_{n=-1}^{\infty} \tilde{A}_{n}^{(1)},
\end{eqnarray*}
  which is the transition probability matrix of the background process $\{\tilde{Z}^{(1)}_{2}(n)\}$. By Lemma A.2 of \cite{KobaMiyaZhao2010}, $\tilde{A}^{(1)}$ must be transient because $\tau_{1} < \theta^{(1,\cp)}_{1} = \theta^{(1,\edge)}_{1}$. Under this change of measure, $R^{(1)}_{n}$ and $H^{(1)}_{n}$ are similarly changed to $\tilde{R}_{n}^{(1)}$ and $\tilde{H}_{n}^{(1)}$. Namely,
\begin{eqnarray*}
  \tilde{R}^{(1)}_{n} = \Delta_{\vc{x}}^{-1} (e^{\tau_{1} n} R_{n}^{(1)})^{\rs{t}} \Delta_{\vc{x}}, \quad n \ge 1, \qquad \tilde{H}^{(1)}_{n} = \Delta_{\vc{x}}^{-1} (e^{\tau_{1} n} H_{n}^{(1)})^{\rs{t}} \Delta_{\vc{x}} \quad n \ge 0.
\end{eqnarray*}

  It is notable that $\tilde{R}^{(1)} \equiv \sum_{n=1}^{\infty} \tilde{R}_{n}^{(1)}$ is stochastic because $\vc{x}$ is also the left invariant vector of $R^{(1)}_{*}(e^{\tau_{1}})$ by the Wiener-Hopf factorization \eqn{WHF 1}. For $n =1, 2, \ldots$, let $\tilde{\sigma}^{(1)}_{1}(0) = 0$ and let
\begin{eqnarray*}
  \tilde{\sigma}^{(1)}_{1}(n) = \inf \{\ell \ge \tilde{\sigma}^{(1)}_{1}(n-1); \tilde{Z}^{(1)}_{1}(\ell) -  \tilde{Z}^{(1)}_{1}(\tilde{\sigma}^{(1)}_{1}(n-1)) \ge 1 \},
\end{eqnarray*}
   then we can see from the version of \eqn{WHF 1} for $\{\tilde{A}_{n}\}$ that $\tilde{R}_{n}^{(1)}$ is the transition probability matrix at the first ascending ladder epoch, that is, its $(i,j)$ entry is given by
\begin{eqnarray*}
  \tilde{r}_{n}^{(1)}(i,j) = P( \tilde{Z}^{(1)}_{1}(\tilde{\sigma}^{(1)}_{1}(1)) - \tilde{Z}^{(1)}_{1}(0) = n, \tilde{Z}^{(1)}_{2}(\tilde{\sigma}^{(1)}_{1}(1)) = j, | \tilde{Z}^{(1)}_{2}(0) = i), \quad n \ge 1,
\end{eqnarray*}
  Because $\tilde{R}^{(1)}$ is stochastic, this implies that $\tilde{Z}_{1}(\tilde{\sigma}^{(1)}_{1}(n))$ drifts to $\infty$ as $n \to \infty$ with probability one. Let $\tilde{h}^{(1)}_{n}(i,j)$ be the $(i,j)$ entry of $\tilde{H}^{(1)}_{n}$. Since
\begin{eqnarray*}
  \tilde{h}^{(1)}_{n}(i,j) = 1(n=0) I + \sum_{\ell=1}^{n} \sum_{k=0}^{\infty} \tilde{r}^{(1)}_{\ell}(i,k) \tilde{h}^{(1)}_{n-\ell}(k,j),
\end{eqnarray*}
   we have, for $n \ge 1$,
\begin{eqnarray*}
  \lefteqn{\tilde{h}^{(1)}_{n}(i,j) = \sum_{\ell = 1}^{n} P( \tilde{Z}^{(1)}_{1}(\tilde{\sigma}^{(1)}_{1}(\ell)) - \tilde{Z}^{(1)}_{1}(0) = n, \tilde{Z}^{(1)}_{2}(\tilde{\sigma}^{(1)}_{1}(\ell)) = j | \tilde{Z}^{(1)}_{2}(0) = i)}\\
 && = P( \cup_{\ell=1}^{n} \{\tilde{Z}^{(1)}_{1}(\tilde{\sigma}^{(1)}_{1}(\ell)) - \tilde{Z}^{(1)}_{1}(0) = n\} \cap \{ \tilde{Z}^{(1)}_{2}(\tilde{\sigma}^{(1)}_{1}(\ell)) = j \} | \tilde{Z}^{(1)}_{2}(0) = i),
\end{eqnarray*}
where the second equality is obtained because $\tilde{Z}^{(1)}_{1}(\tilde{\sigma}^{(1)}_{1}(\ell))$ is increasing in $\ell$. Thus, $\tilde{h}^{(1)}_{n}(i,j)$ is obtained as the solution of the Markov renewal equation, which is uniformly bounded by unit. However, we can not apply the standard Markov renewal theorem because its background kernel $\tilde{A}^{(1)}$ is transient. Nevertheless, we can show that, for some constant $a > 0$,
\begin{eqnarray}
\label{eqn:H limit}
  \lim_{n \to \infty} \tilde{H}^{(1)}_{n} \vc{1} = \frac 1{a} \vc{1}.
\end{eqnarray}
  Intuitively, this may be obvious because both entries of $\tilde{\vc{Z}}^{(1)}(n)$ go to infinity as $n \to \infty$ and they behave like a random walk asymptotically when they get large. However, we need to prove \eqn{H limit}. Since this proof is quite technical, we defer it to \app{renewal theorem}.

  It follows from \eqn{renewal equation 2} using the vector row $\vc{y} = \{e^{-\tau_{2} \ell}; \ell \ge 0\}$ that
\begin{eqnarray}
\label{eqn:pi n 1}
  e^{\tau_{1} n} \vc{\nu}^{(1)}_{n} &\ge& \vc{\nu}^{(1)}_{1} \Delta_{\vc{x}}^{-1} \Delta_{\vc{x}} (e^{\tau_{1} n} H^{(1)}_{n-1}) \Delta_{\vc{x}}^{-1} \Delta_{\vc{x}} \nonumber\\
  &=& e^{\tau_{1}} \left(\Delta_{\vc{x}} \tilde{H}^{(1)}_{n-1} \Delta_{\vc{x}}^{-1} (\vc{\nu}^{(1)}_{1})^{\rs{t}} \right)^{\rs{t}}.
\end{eqnarray}
  Hence, applying the bounded convergence theorem, \eqn{x i convergence} and \eqn{H limit}, we get \eqn{reflecting 1}.
\end{proof}

\bigskip

  We are now ready to prove Theorems \thrt{decay rate 0} and \thrt{decay rate 1}.

\begin{proof*}{The proof of \thr{decay rate 0}}

  Because of symmetry, we only prove for $k=1$. The rough asymptotic \eqn{decay rate 0} is immediate from Lemmas \lemt{LD upper bound 1} and \lemt{LD lower bound 1}. The remaining part is also immediate from the second part of \lem{LD lower bound 1}.
\end{proof*}

\bigskip

\subsection{Proof of Theorem \thrt{decay rate 1}}
\label{sect:Theorem 3.3}
  
   We first consider where the ray $x \vc{c}$ with $x > 0$ intersects $\sr{D}$ in the $(\theta_{1},\theta_{2})$-plane. By \thr{domain D}, there are three cases:
\begin{mylist}{0}
\item [(a)] It intersects the vertical line $\theta_{1} = \tau_{1}$.
\item [(b)] It intersects the horizontal line $\theta_{2} = \tau_{2}$.
\item [(c)] It intersects $\partial \Gamma_{\max} \equiv \mbox{the boundary of } \Gamma_{\max}$ (see \fig{case (c)} below).
\end{mylist}
 (a) and (b) are symmetric, so we only need to consider cases (a) and (c).
 
  We first consider case (a). In this case, 
\begin{eqnarray*}
  \alpha_{\vc{c}} = \sqrt{\left(\tau_{1} \right)^{2} + \left(\frac {c_{2}} {c_{1}} \tau_{1}\right)^{2} } = \frac 1{c_{1}} \tau_{1},
\end{eqnarray*}
  because $\br{\vc{c}, \vc{c}} = 1$. On the other hand, it follows from \eqn{lower bound 5} that
\begin{eqnarray*}
  \liminf_{x \to \infty} \frac 1x \log P( \br{\vc{c}, \vc{L}} > x ) & \ge & \frac 1{c_{1}} \liminf_{x \to \infty} \frac {c_{1}} x \log P\left( L_{1} > \frac 1{c_{1}} x, L_{2} = 0 \right)\\
  & \ge & - \frac 1{c_{1}} \tau_{1}.
\end{eqnarray*}
  Hence, combining this with the upper bound \eqn{LD upper bound 2} of \lem{LD upper bound 1}, we have \eqn{Ldecay rate 1}.
  
  We next consider case (c). In this case, \eqn{Ldecay rate 1} is immediate from Lemmas \lemt{infinite domain} and \lemt{LD upper bound 1}.
\begin{figure}[h]
 	\centering
	\includegraphics[height=4.6cm]{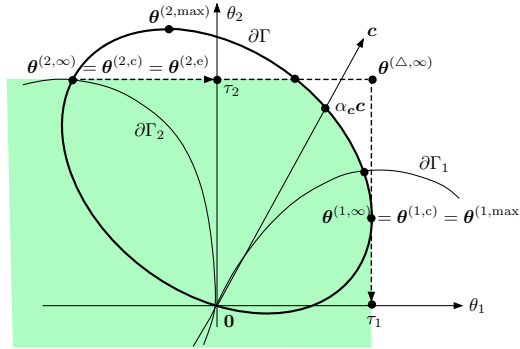}
	\caption{Typical figures for case (c)}
	\label{fig:case (c)}
\end{figure}
  
  It remains to prove \eqn{exact asymptotic 2}. For this, we consider a singular point of the analytic function  $\varphi(z \vc{c})$ obtained from the moment generating function of $\br{\vc{c}, \vc{L}}$. From \eqn{stationary equation 1},
\begin{eqnarray}
\label{eqn:analytic zc}
  (1 - \gamma(z \vc{c})) \varphi(z \vc{c}) = \sum_{k=1}^{2} (\gamma_{k}(z \vc{c}) - \gamma(z \vc{c})) \varphi_{k}(c_{k} z) + (\gamma_{0}(z \vc{c}) - \gamma(z \vc{c})) \varphi_{0}(0),
\end{eqnarray}
  for $(\Re z) \vc{c} \in \sr{D}$. From the assumptions on $\alpha_{\vc{c}}$, we first observe that the right hand side of \eqn{analytic zc} is analytic for $\Re z < \equiv \min( \frac 1 {c_{1}}\tau_{1}, \frac 1 {c_{2}}\tau_{2})$, which is greater than $\alpha_{\vc{c}}$ (see \fig{case (c)}).
  
  We next consider the root of equation $1 - \gamma(z \vc{c}) = 0$ for $\Re z = \alpha_{\vc{c}}$ in the complex number field. We first note that the root $z=\alpha_{\vc{c}}$ is simple because $\gamma(t \vc{c})$ is a convex function of $t \in \dd{R}$ and $\gamma(\vc{0}) = 1$. If $\br{\vc{c}, \vc{X}^{(+)}}$ is $\delta$-arithmetic for some $\delta > 0$, then $\br{\vc{c}, \vc{X}^{(+)}}/\delta$ is integer valued and $1$-arithmetic, and therefore the roots are of the form that $\alpha_{\vc{c}} + 2\pi i k/\delta$ for $k \in \dd{Z}$, where $i = \sqrt{-1}$. If there is no $\delta > 0$ for $\br{\vc{c}, \vc{X}^{(+)}}$ to be $\delta$-arithmetic, then we can see that there is no other root than $z = \alpha_{\vc{c}}$ because  $\gamma(z)$ is the function of $e^{z}$.
  
  We now consider the two cases separately. First assume that $\br{\vc{c}, \vc{X}^{(+)}}$ is $\delta$-arithmetic for some integer $\delta > 0$. Since $\br{\vc{c}, \vc{X}^{(+)}}/\delta$ is integer valued, we consider generating functions instead of moment generating functions. That is, we change complex variable $z$ to $w = e^{\delta z}$ in $\varphi(z \vc{c})$, $\varphi_{k}(c_{k} z)$, $\gamma(z\vc{c})$ and $\gamma_{k}(z\vc{c})$, and denote them, respectively, by $f(w)$, $f_{k}(w)$, $g(w)$ and $g_{k}(w)$, which are generating functions. Since the analytic properties of the original functions are transformed to these generating functions, we can see that $f(w)$ is analytic for $|w| < e^{\delta \alpha_{\vc{c}}}$, and there is no singular point on the circle $|w| = e^{\delta \alpha_{\vc{c}}}$ except for $w = e^{\delta \alpha_{\vc{c}}}$. Since the circle in the complex plane is compact and $g(w) = 1$ has a simple root at $w = e^{\delta \alpha_{\vc{c}}}$, we can analytically expand $f(w)$ given by
\begin{eqnarray*}
  f(w) = \frac 1{1 - g(w)} \left(\sum_{k=1}^{2} (g_{k}(w) - g(w)) f_{k}(w) + (g_{0}(w \vc{c}) - g(w)) \varphi_{0}(0) \right)
\end{eqnarray*}
to the region $\{w \in \dd{C}; |w| < e^{\delta \alpha_{\vc{c}}} + \epsilon \}$ for some $\epsilon > 0$ except for $w = e^{\delta \alpha_{\vc{c}}}$. Since $h$ must be singular at $w = e^{\delta \alpha_{\vc{c}}}$ and $1 - g(w) = 0$ has a single root there, $f(w)$ has a simple pole at $w = e^{\delta \alpha_{\vc{c}}}$. Hence, we can apply the asymptotic  inversion technique for a generating function of a 1-arithmetic distribution (e.g., see Theorem VI.5 of \cite{FlajSedg2009}). Thus, we have \eqn{exact asymptotic 2}.
  
  We now assume that $\br{\vc{c}, \vc{X}^{(+)}}$ is not $\delta$-arithmetic for any $\delta > 0$. In this case, we back to moment generating function. We already observed that there $\gamma(z \vc{c}) = 1$ has no other root for $\Re z = \alpha_{\vc{c}}$ than $z = \alpha_{\vc{c}}$ in the complex number field. Furthermore, the circle $\{e^{z} \in \dd{C}; \Re z = \alpha_{\vc{c}} \}$ is a compact set and $1 = |\gamma(z\vc{c})| \le \gamma(\Re z \vc{c})$. Hence, for any sequence of complex numbers $z_{n}$ for $n=1,2, \ldots$ such that $\gamma(z_{n} \vc{c}) = 1$, we have $\Re z_{n} > \alpha_{\vc{c}}$, and $\Re z_{n}$ can not converge to $\alpha_{\vc{c}}$ as $n \to \infty$. Namely, it were a converging sequence, then $\{e^{z_{n}} \in \dd{C}; n =1,2,\ldots \}$ has a converging subsequence, and therefore $\gamma(z \vc{c}) = 1$ for all $z \in \dd{C}$ by the uniqueness of analytic extension, which is a contradiction. This proves that $1 - \gamma(z \vc{c}) = 0$ has a single root at $z = \alpha_{\vc{c}}$ in the complex region $\{ z \in \dd{C}; 0 < \Re z < \alpha_{\vc{c}} + \epsilon\}$ for some $\epsilon > 0$. Hence, $\varphi(z \vc{c})$ is analytically extendable for $\Re z < \alpha_{\vc{c}} + \epsilon$ except at the simple pole at $z = \alpha_{\vc{c}}$. To this analytic function, we apply Lemma 6.2 of \cite{DaiMiya2013}, which is an adaptation of the asymptotic inversion due to Doetsch \cite{Doet1974}. Then, we can get \eqn{exact asymptotic 2}.

\section{Application to queueing networks}
\label{sect:Application}
\setnewcounter

  In this section, we consider a two node Markovian network with batch arrivals. This network generalizes the Jackson network so that each node may have batch arrivals at once. By applying \thr{decay rate 1}, it is easy to compute the decay rates at least numerically not only for this modification but also for further modification such that the service rates are changed when either one of the nodes is empty because they can be formulated as the double $M/G/1$-type process. It may be notable that special cases of those models have been studied in the literature (e.g, see \cite{FlatHahn1984,GuilKnesLeeu2013,GuilPinc2004,GuilLeeu2011,KobaMiya2012}). Exact asymptotics have been studied there, while \thr{decay rate 1} only answers rough asymptotics.
  
  The aim of this section is twofold. First, we consider the influence of the variability of the batch size distributions to the decay rates. Secondly, we examine whether the upper bound of \cite{MiyaTayl1997} for the stationary distribution is tight concerning the decay rate, where it is assumed that there is no simultaneous arrival as in \cite{MiyaTayl1997}. This tightness has been never considered.

\subsection{Two node Markovian network with batch arrivals}
\label{sect:Jackson}
  
Consider a continuous time Markovian network with two nodes, numbered as $1$ and $2$. We assume the following arrivals and service. Batches of customers arrive either at one of the two nodes or simultaneously at both nodes from outside, which we describe by the compound Poisson process with rate $\lambda > 0$ and joint batch size distribution $F$. The service times at node $k$ $(k = 1,2)$ are independent and have the common exponential distribution with mean $\mu_{k}^{-1}$, which are also independent of the batch arrivals. Because of this exponential assumption, service discipline is irrelevant as long as servers are busy. A customer who completes service at node $1$ moves to node $2$ with probability $p_{12}$, or leaves the network with $1-p_{12}$. Similarly, a departing customer from node $2$ goes to node $1$ with probability $p_{21}$, or outside with probability $1-p_{21}$. To exclude trivial cases, we assume
\begin{mylist}{3}
\item [(\sect{Application}a)] $0 < p_{12}< 1$ and $0 < p_{21}< 1$.
\end{mylist}
  We assume without loss of generality that $\lambda + \mu_{1} + \mu_{2} = 1$.
 
  Let $\vc{B} \equiv (B_{1}, B_{2})$ be a random vector subject to the joint batch size distribution $F$, and denote its moment generating function by $\hat{F}$. We assume
\begin{mylist}{3}
\item [(\sect{Application}b)] For each non-zero $\vc{\theta} \ge \vc{0}$, $\sup \{t > 0; \hat{F}(t\vc{\theta}) < \infty\} = \infty$.
\end{mylist}
   Thus, $F$ has light tails. The simplest model of this type is a two parallel queues with two simultaneous arrivals (no batch arrival at each node), whose exact tail asymptotics were studied by \cite{FlatHahn1984}. This result is recently generalized for a network but without batch arrival in \cite{KobaMiya2012}. The present batch arrival network is more general than these model, but we can only derive the decay rates except for some cases.

  Let $L_{tk}$ be the queue length of node $k$ at time $t$. Clearly, $(L_{t1}, L_{t2})$ is a continuous time Markov chain. We assume the intuitive stability condition:
\begin{eqnarray}
\label{eqn:stable}
\rho_{1} \equiv \frac{\lambda (b_{1} + b_{2} p_{21})}{(1-p_{12} p_{21})\mu_{1}} < 1 , \qquad \rho_{2} \equiv \frac {\lambda (b_{2} + b_{1} p_{12})} {(1-p_{12} p_{21})\mu_{2}} < 1,
\end{eqnarray}
  where $b_{k} = E(B_{k})$. One can verify that this condition is identical with the stability condition given in \lem{stability d=2} using fact that
\begin{eqnarray*}
 && \br{\vc{m}, \vc{m}^{(1)}_{\bot}} = \mu_{2}(\lambda (b_{1} + b_{2} p_{21}) - \mu_{1}(1 - p_{12} p_{21})),\\
 && \br{\vc{m}, \vc{m}^{(2)}_{\bot}} = \mu_{1}(\lambda (b_{2} + b_{1} p_{12}) - \mu_{2}(1 - p_{12} p_{21})).
\end{eqnarray*}
Thus, \eqn{stable} is indeed the stability condition for $(L_{t1}, L_{t2})$ to have the stationary distribution, which is denoted by $\nu$.
 
  By the well known uniformization, we reformulate the continuous time Markov chain $\{(L_{t1}, L_{t2}); t \ge 0\}$ as a discrete time one with the same stationary distribution $\nu$. This discrete time Markov chain is a double $M/G/1$-type process, so denoted by $\{(L_{1}(\ell), L_{2}(\ell)); \ell = 0,1,\ldots\}$.
 
  This queueing network is more general than the model studied in \cite{MiyaTayl1997} in the sense that simultaneous arrivals at both nodes may occur. If there is no simultaneous arrival at both nodes, then the model becomes a special case of the network of \cite{MiyaTayl1997} because batch departures are not allowed. We will consider this special case in \sectn{Proof T33} for considering the quality of the upper bound of \cite{MiyaTayl1997}.
  
  It is known that batch arrivals and/or simultaneous arrivals make it very hard to get the stationary distribution, while it is obvious if there is no such arrival. The latter network is the Jackson network, which has the stationary distribution of a product form as is well known.

\subsection{Influence of batch size distributions}
\label{sect:Influence}

  For the double $M/G/1$-type process for the batch arrival network, we compute $\gamma(\vc{\theta})$ and $\gamma_{k}(\vc{\theta})$ for $k=1,2$. 
\begin{eqnarray}
\label{eqn:batch gamma +}
  && \gamma(\vc{\theta}) = \lambda \hat{F} (\vc{\theta}) + \mu_{1} e^{-\theta_{1}} (1-p_{12}+ p_{12}e^{\theta_{2}}) + \mu_{2} e^{-\theta_{2}} (1-p_{21}+ p_{21}e^{\theta_{1}}), \hspace{5ex} \\
\label{eqn:batch gamma 1}
  && \gamma_{1}(\vc{\theta}) = \lambda \hat{F} (\vc{\theta}) + \mu_{1} e^{-\theta_{1}} (1-p_{12}+ p_{12}e^{\theta_{2}}) + \mu_{2},\\
\label{eqn:batch gamma 2}
  && \gamma_{2}(\vc{\theta}) = \lambda \hat{F} (\vc{\theta}) + \mu_{1} + \mu_{2} e^{-\theta_{2}} (1-p_{21}+ p_{21}e^{\theta_{1}}).
\end{eqnarray}

Hence, if the distribution of $\vc{B}$ is increased in linear convex order, then $\gamma(\vc{\theta})$, $\gamma_{1}(\vc{\theta})$ and $\gamma_{2}(\vc{\theta})$ are increased for each fixed $\vc{\theta} \in \dd{R}^{2}$ as long as they exist. Hence, \cor{monotonicity 2} yields the following fact.

\begin{proposition}
\label{pro:batch monotonicity}
  For the batch arrival Markovian network satisfying the conditions (\sect{Application}a), (\sect{Application}b) and \eqn{stable} if the distribution of $\vc{B}$ is increased in linear convex order, then the decay rates $\tau_{1}$, $\tau_{2}$ in \thr{decay rate 0} and $\alpha_{\vc{c}}$  in \thr{decay rate 1} are decreased.
\end{proposition}

To get the decay rates, we need to find $\vc{\theta}^{(k, \edge)}$ and $\vc{\theta}^{(k, \max)}$ for $k=1,2$, which are the roots of the equations $\gamma(\vc{\theta}) = \gamma_{k}(\vc{\theta}) = 1$ and $\gamma(\vc{\theta}) = 1$ satisfying $\frac {d\theta_{k}} {d \theta_{3-k}} = 0$, respectively. We here note that $\gamma(\vc{\theta}) = \gamma_{1}(\vc{\theta}) = 1$ is equivalent to $\gamma_{1}(\vc{\theta}) = 1$ and
\begin{eqnarray*}
  e^{\theta_{2}} = 1-p_{21}+ p_{21}e^{\theta_{1}},
\end{eqnarray*}
  which follows from $\gamma(\vc{\theta}) = \gamma_{1}(\vc{\theta})$. Thus, the numerical values of the decay rates are easily computed using a software such as Mathematica, but their analytical expressions are very hard to get except for the skip free case. The model studied by \cite{FlatHahn1984} is the simplest skip-free case. Theorems \thrt{decay rate 0} and \thrt{decay rate 1} are fully compatible with their asymptotic results.

\subsection{Stochastic upper bound of Miyazawa and Taylor}
\label{sect:Proof T33}
  
  We next consider the stochastic upper bound for the stationary distribution $\nu$, obtained by Miyazawa and Taylor \cite{MiyaTayl1997}. Since their model does not allow simultaneous arrival, we assume that either one of $B_{1}$ and $B_{2}$ is zero. Thus, the joint batch size distribution $F$ can be written as
\begin{eqnarray*}
  F(x_{1}, x_{2}) = F(x_{1}, 0) + F(0, x_{2}), \qquad x_{1}, x_{2} \ge 0.
\end{eqnarray*}
  Let $F_{1}(x) = F(x, 0)/F(\infty, 0)$ and $F_{2}(x) = F(0, x)/F(0, \infty)$, and let $\lambda_{1} = \lambda F(\infty, 0)$ and $\lambda_{2} = \lambda F(0, \infty)$. For computational convenience, we switch to generating functions from moment generating functions. Let $\tilde{F}_{k}$ be the generating functions of $F_{k}$.  We present the upper bound of \cite{MiyaTayl1997} using our notation.

\begin{proposition}[Corollary 3.2 and Theorem 4.1 of \cite{MiyaTayl1997}]
\label{pro:stochastic bound}
  If (\sect{Application}a), (\sect{Application}b) and the stability condition \eqn{stable} hold, then the equations:
\begin{eqnarray}
\label{eqn:r1}
&&\lambda_{1} \left(\tilde{F}_{1}(s_{1}) - 1 \right) + \mu_{1} s_{1}^{-1} (1- s_{1})  = \mu_{2} p_{21} s_{2}^{-1} (1 - s_{1}),\\
\label{eqn:r2}
&& \lambda_{2} \left(\tilde{F}_{2}(s_{2}) - 1 \right) + \mu_{2} s_{2}^{-1} (1-s_{2})  = \mu_{1} p_{12} s_{1}^{-1} (1 - s_{2}),
\end{eqnarray}
  have solutions $(s_{1}, s_{2}) > \vc{0}$. Let $(h_{1}, h_{2})$ be the maximal solution among them, then,
\begin{eqnarray*}
\label{eqn:upeer bound}
P(\vc{L} \ge \vc{n}) \leq h_{1}^{ - n_{1}} h_{2}^{- n_2}, \qquad \vc{n} = (n_{1}, n_{2}) \ge \vc{0}, 
\end{eqnarray*}
where $\vc{L}$ is a random vector subject to the stationary distribution $\nu$.
\end{proposition}

  To compare $h_{k}$ with the decay rate, we let
\begin{eqnarray*}
  \eta_{k} = \log h_{k}, \qquad k=1,2.
\end{eqnarray*}
  By \thr{decay rate 1} and \pro{stochastic bound}, we have that $\eta_{k} \le \alpha_{k}$, where we recall that $\alpha_{k} = \alpha_{\vcn{e}_{k}}$.
  
  Let $s_{k} = e^{\theta_{k}}$ in \eqn{batch gamma +}, \eqn{batch gamma 1}, \eqn{batch gamma 2}, then $\gamma(\vc{\theta})$, $\gamma_{1}(\vc{\theta})$ and $\gamma_{2}(\vc{\theta})$ can be written as
\begin{eqnarray}
\label{eqn:mgf g}
  && \lambda_{1} \tilde{F}_{1}(s_{1}) + \mu_{1} s_{1}^{-1} (1-p_{12}+ p_{12} s_{2}) + \lambda_{2} \tilde{F}_{2}(s_{2}) + \mu_{2} s_{2}^{-1} (1-p_{21}+ p_{21}s_{1}) = 1, \hspace{5ex}\\
\label{eqn:mgf g1}
  && \lambda_{1} \tilde{F}_{1}(s_{1}) + \mu_{1} s_{1}^{-1} (1-p_{12}+ p_{12}s_{2}) + \lambda_{2} \tilde{F}_{2}(s_{2}) + \mu_{2} = 1,\\
\label{eqn:mgf g2}
  && \lambda_{1} \tilde{F}_{1}(s_{1}) + \mu_{1} + \lambda_{2} \tilde{F}_{2}(s_{2}) + \mu_{2} s_{2}^{-1} (1-p_{21}+ p_{21}s_{1}) = 1.
\end{eqnarray}

  Note that \eqn{r1} and \eqn{r2} imply \eqn{mgf g}. That is, $(h_{1}, h_{2})$ satisfies the equation \eqn{mgf g} for variable $(s_{1},s_{2})$. In other words, the point $(\eta_{1}, \eta_{2})$ is on the curve $\partial \Gamma$.
  
  Our question is when $(\eta_{1}, \eta_{2})$ is identical with $(\alpha_{1}, \alpha_{2})$, that is, when the upper bounds agree with the decay rates of the marginal distributions in coordinate directions. We answer it by the following theorem.

\begin{theorem}
\label{thr:Jackson tightness}
  Under (\sect{Application}a), (\sect{Application}b) and the stability condition \eqn{stable}, the decay rate $\eta_{k} \equiv \log {t}_{k}$ of the stochastic upper bound of \cite{MiyaTayl1997} is identical with the decay rate $\alpha_{k}$ for $k=1,2$ if and only if both nodes have no batch arrival.
  
\begin{proof}
  The sufficiency of the single arrivals is immediate from the well known product form solution for the Jackson network. Thus, we only need to prove the necessity. Assume that $(\eta_{1}, \eta_{2}) = (\alpha_{1}, \alpha_{2})$. As we already observed, this implies that  $(\alpha_{1}, \alpha_{2}) \in \partial \Gamma$. Hence, from \thr{decay rate 1} and \eqn{alpha k}, we can see that neither $\alpha_{1} = \beta_{1}$ nor $\alpha_{2} = \beta_{2}$ is possible because $\alpha_{k} = \beta_{k}$ implies that $\alpha_{2-k} = 0$, where $\beta_{k}$ is defined after \eqn{alpha k}. Furthermore, we cannot simultaneously have $\alpha_{1} = \theta^{(1, \max)}_{1}$ and $\alpha_{2} = \theta^{(2, \max)}_{2}$ because of $(\alpha_{1}, \alpha_{2}) \in \partial \Gamma$. Hence, we must have either $\alpha_{1} = \theta^{(1, \edge)}_{1}$ or $\alpha_{2} = \theta^{(2, \edge)}_{2}$.
  
  Suppose that $\alpha_{1} = \theta^{(1, \edge)}_{1}$, which implies that $\eta_{1} =  \theta^{(1, \edge)}_{1}$ by our assumption. For convenience, we introduce notations:
\begin{eqnarray*}
  t^{(k,\edge)}_{i} = e^{\theta^{(k, \edge)}_{i}}, \qquad i, k = 1,2.
\end{eqnarray*}
  Then, $\eta_{1} =  \theta^{(1, \edge)}_{1}$ is equivalent to $h_{1} =  t^{(1, \edge)}_{1}$, and $(t^{(1, \edge)}_{1}, t^{(1, \edge)}_{2})$ is the solution of the equations \eqn{mgf g} and \eqn{mgf g1} for variable $(s_{1}, s_{2})$. Thus, it follows from \eqn{mgf g} and \eqn{mgf g1} that
\begin{eqnarray}
\label{eqn:t 12}
   t^{(1, \edge)}_{2} = 1 - p_{21}+ p_{21} t^{(1, \edge)}_{1}.
\end{eqnarray}
  Substituting \eqn{t 12} into \eqn{mgf g1} with $(s_{1}, s_{2}) = (t^{(1, \edge)}_{1}, t^{(1, \edge)}_{2})$, we have
\begin{eqnarray*}
  \lambda_{1} \left(\tilde{F}_{1}(t^{(1, \edge)}_{1}) -1 \right) + \lambda_{2} \left(\tilde{F}_{2}(1 - p_{21}+ p_{21} t^{(1, \edge)}_{1}) -1 \right) + \mu_{1} (1 - p_{12} p_{21})( (t^{(1, \edge)}_{1})^{-1} - 1) = 0.
\end{eqnarray*}
  Thus, $t^{(1, \edge)}_{1}$ is obtained as the unique positive solution of this equation. Since $h_{1} = t^{(1, \edge)}_{1}$, this implies that
\begin{eqnarray}
\label{eqn:batch 3}
  \lambda_{1} \left(\tilde{F}_{1}(h_{1}) -1 \right) + \lambda_{2} \left(\tilde{F}_{2}(1 - p_{21}+ p_{21}h_{1}) -1 \right) + \mu_{1} (1 - p_{12} p_{21})(h_{1}^{-1} - 1) = 0.
\end{eqnarray}

  From \eqn{r1} with $(s_{1}, s_{2}) = (h_{1}, h_{2})$ and \eqn{batch 3}, we have
\begin{eqnarray*}
  \lambda_{2} \tilde{F}_{2}( 1 - p_{21}+ p_{21} h_{1}) = \mu_{1} p_{12} p_{21} h_{1}^{-1} (1 - h_{1}) - \mu_{2} p_{21} h_{2}^{-1} (1- h_{1}).
\end{eqnarray*}
  This yields
\begin{eqnarray*}
  \lambda_{2} \frac {\tilde{F}_{2}(1 - p_{21}+ p_{21} h_{1}) - 1} {(1 - p_{21}+ p_{21} h_{1}) - 1} = \mu_{2} h_{2}^{-1} - \mu_{1} p_{12} h_{1}^{-1} .
\end{eqnarray*}
  On the other hand, it follows from \eqn{r2} with $(s_{1}, s_{2}) = (h_{1}, h_{2})$ that
\begin{eqnarray*}
  \lambda_{2} \frac {\tilde{F}_{2}(h_{2}) - 1} {h_{2} - 1} = \mu_{2} h_{2}^{-1} - \mu_{1} p_{12} h_{1}^{-1} .
\end{eqnarray*}
  Hence, we must have
\begin{eqnarray*}
  \frac {\tilde{F}_{2}(1 - p_{21}+ p_{21} h_{1}) - 1} {(1 - p_{21}+ p_{21} h_{1}) - 1} = \frac {\tilde{F}_{2}(h_{2}) - 1} {h_{2} - 1}.
\end{eqnarray*}
  Since $\tilde{F}_{2}(s)$ is a strictly increasing convex function, this equation is true only when either $\tilde{F}_{2}(s) = s$, which is equivalent to no batch arrivals at node $2$, or 
\begin{eqnarray}
\label{eqn:batch 4a}
  h_{2} = 1 - p_{21}+ p_{21} h_{1}.
\end{eqnarray}

  Suppose that node 2 has batch arrivals. Then, \eqn{batch 4a} holds, and therefore \eqn{t 12} implies $h_{2} = t^{(1, \edge)}_{2}$. This is equivalent to that $(\eta_{1}, \eta_{2}) = \vc{\theta}^{(1, \edge)}$. Hence, the assumption that $(\eta_{1}, \eta_{2}) = (\alpha_{1}, \alpha_{2})$ implies that we must have the case (D2), which in turn implies that $h_{2} = t^{(1, \edge)}_{2} = t^{(2, \edge)}_{2}$. Hence, by the same arguments, we have either $\tilde{F}_{1}(s) = s$ or 
\begin{eqnarray}
\label{eqn:batch 4b}
  h_{1} = 1 - p_{12}+ p_{12} h_{2}.
\end{eqnarray}
  However, if both of \eqn{batch 4a} and \eqn{batch 4b} hold, then $h_{1} = h_{2} = 1$, which is impossible. Hence, we must have that $\tilde{F}_{1}(s) = s$, that is, node $1$ has no batch arrivals.
  
  Applying $\tilde{F}_{1}(h_{1}) = h_{1}$ to \eqn{mgf g1} with $(s_{1}, s_{2}) = (h_{1}, h_{2})$ and using the fact that $h_{1} > 1$, we have
\begin{eqnarray}
\label{eqn:h 1}
  \mu_{1} h_{1}^{-1} = \lambda_{1} + \mu_{2} p_{21} h_{2}^{-1} .
\end{eqnarray}
  On the other hand, from \eqn{r2} with $(s_{1}, s_{2}) = (h_{1}, h_{2})$ and \eqn{batch 4a}, we have
\begin{eqnarray*}
  \lefteqn{\lambda_{2} \left(\tilde{F}_{2}(h_{2}) - 1 \right) = \mu_{1} p_{12} h_{1}^{-1} (1 - h_{2}) - \mu_{2} h_{2}^{-1} (1-h_{2}) } \hspace{10ex}\\
 && = (1 - h_{2}) \left(\mu_{1} h_{1}^{-1}  p_{12}- \mu_{2} h_{2}^{-1} \right)\\
 && = (1 - h_{2}) \left( (\lambda_{1} + \mu_{2} p_{21} h_{2}^{-1}) p_{12}- \mu_{2} h_{2}^{-1} \right),
\end{eqnarray*}
  where the last equality is obtained by substituting \eqn{h 1}. Rearranging terms in this equation and recalling that $\rho_{2} = (\lambda_{1} + \lambda_{2} p_{12})/((1-p_{12} p_{21}) \mu_{1})$, we have
\begin{eqnarray}
\label{eqn:h 2}
  \lambda_{2} \left(\tilde{F}_{2}(h_{2}) - h_{2} \right) &=& (1 - h_{2}) \left( (\lambda_{2} + \lambda_{1} p_{12}) - (1 - p_{21} p_{12}) \mu_{2} h_{2}^{-1} \right) \nonumber\\
 & = & (1 - h_{2}) (1 - p_{21} p_{12}) \mu_{2} \left( \rho_{2} - h_{2}^{-1} \right).
\end{eqnarray}
  Since this $\rho_{2}$ is identical with the geometric decay rate of the Jackson network with single arrivals at both nodes, we obviously have $\rho_{2} \le h_{2}^{-1}$. Thus, the right-hand side of \eqn{h 2} is not positive, but its left-hand side must be positive because node 2 has batch arrivals, which implies $\tilde{F}_{2}(h_{2}) - h_{2} > 0$. This is a contradiction, and node 2 can not have batch arrivals. This concludes that both nodes do not have batch arrivals. By symmetry, we have the same conclusion for $\alpha_{2} = \theta^{(2,\edge)}_{2}$. Hence, we have completed the proof.
\end{proof}
\end{theorem}

  This theorem shows that the stochastic bound of \cite{MiyaTayl1997} cannot be tight even for the decay rates. However, this may not exclude the case where one of the upper bounds is tight. In our proof, the tightness leads to a contradiction if either node 2 has batch arrivals or (D2) holds. This suggests that, if node 2 has no batch arrivals and if (D2) does not hold, then node 1 with $\eta_{1} = \alpha_{1}$ may have batch arrivals. The following corollary affirmatively answers it.
  
\begin{corollary}
\label{cor:Jackson tightness}
  Under the same assumptions of \thr{Jackson tightness}, $\eta_{1} = \alpha_{1}$ holds if there is no batch arrival at node 2 and either (D1) with $\vc{\theta}^{(1,\max)} \not\in \ol{\Gamma}_{1}$ or (D3) holds. These conditions are also necessary for $\eta_{1} = \alpha_{1}$ if node 1 has batch arrivals. Similarly, $\eta_{2} = \alpha_{2}$ holds if and only if there is no batch arrival at node 1 and either (D1) with $\vc{\theta}^{(2,\max)} \not\in \ol{\Gamma}_{2}$ or (D2) holds. These conditions are also necessary if node 2 has batch arrivals.
\end{corollary}

\begin{proof}
  By symmetry, we only need to prove the first two claims. If there is no batch arrival at node 2, then $(h_{1}, h_{2})$ is obtained as the solution of the equations \eqn{r1} and \eqn{r2} with $\tilde{F}_{2}(s_{2}) = s_{2}$, that is,
\begin{eqnarray}
\label{eqn:r1 t}
&&\lambda_{1} \left(\tilde{F}_{1}(s_{1}) - 1 \right) + \mu_{1} s_{1}^{-1} (1- h_{1})  = \mu_{2} p_{21} s_{2}^{-1} (1 - s_{1}),\\
\label{eqn:r2 t}
&& \mu_{2} s_{2}^{-1}  = \lambda_{2} + \mu_{1} p_{12} s_{1}^{-1}.
\end{eqnarray}
 Substituting $\mu_{2} s_{2}^{-1}$ of \eqn{r2 t} into \eqn{r1 t} implies
\begin{eqnarray}
\label{eqn:batch 5}
 \qquad \lambda_{1} \left(\tilde{F}_{1}(s_{1}) -1 \right) - \lambda_{2} p_{21} \left(1 - s_{1} \right) + \mu_{1} (1 - p_{12} p_{21})(1 - s_{1}) s_{1}^{-1} = 0. \hspace{-5ex}
\end{eqnarray}
  On the other hand, this equation also follows from \eqn{mgf g}, \eqn{mgf g1} and the single arrivals at node 2, and therefore we have $h_{1} = t^{(1, \edge)}_{1}$, equivalently, $\eta_{1} = \theta^{(1, \edge)}_{1}$. Hence, $\eta_{1} = \alpha_{1}$ holds if either (D1) with $\vc{\theta}^{(1,\max)} \not\in \ol{\Gamma}_{1}$ or (D3) holds, because these conditions imply $\alpha_{1} = \theta^{(1,\edge)}_{1}$. This proves the first claim. To prove the necessity, we note that $\eta_{1} = \alpha_{1}$ implies $\alpha_{1} = \theta^{(1,\edge)}_{1}$. Assume that node 1 has batch arrivals, then neither the batch arrivals at node 2 nor (D2) is possible as shown in the proof of \thr{Jackson tightness}. Hence, it is required that node 2 has no batch arrivals and (D2) does not holds. The latter together with $\alpha_{1} = \theta^{(1,\edge)}_{1}$ implies that either (D1) with $\vc{\theta}^{(1,\max)} \not\in \ol{\Gamma}_{1}$ or (D3) holds. Thus, the second claim is proved.
\end{proof}
  
\section{Concluding remarks}
\label{sect:Concluding}
\setnewcounter

  In this paper, we have studied the tail decay asymptotics of the marginal stationary distributions for an arbitrary direction under conditions (i)--(iv) and the stability condition. Among these conditions, (i) is most restrictive for applications. For example, it excludes a priority queue with two classes of customers. However, it can be relaxed as remarked in Section 7 of \cite{Miya2009}. Hence, (i) is not a crucial restriction. As we already noted, condition (iii) can also be relaxed to (iii)' for obtaining the decay rate. Thus, the rough tail asymptotics can be obtained under the minimum requisites.
  
  What we have not studied in this paper is of other types of asymptotic behaviors of the stationary distribution. In particular, we have not fully studied exact asymptotics. We have recently studied this problem for the skip free reflecting random walk in \cite{KobaMiya2012}. For the unbounded jump case, this is a challenging problem. We hope the convergence domain which we obtained in this paper may be helpful for the kernel method as it proved to be in \cite{KobaMiya2012}.
  
\vspace{3ex}

\appendix
\setcounter{section}{0}
\setnewcounter

\noindent {\bf \Large Appendix}\vspace{-1ex}

\section{Proof of \lem{geometric 1}}
\label{app:geometric 1}

  We first claim that, for each $k=1,2$, there is a $\vc{\theta} \in \Gamma \cap \Gamma_{k}$ such that $\theta_{k} > 0$ if and only if
\begin{eqnarray}
\label{eqn:theta mean 1}
  \br{\vc{\theta}, \vc{m}} < 0, \qquad \br{\vc{\theta}, \vc{m}^{(k)}} < 0 \quad \mbox{and} \quad \theta_{k} > 0.
\end{eqnarray}
 Because of symmetry, we only prove this for $k=1$. Define functions $f$ and $f_{1}$ as
\begin{eqnarray*}
  f(u) = E(e^{u \br{\vc{\theta}, \vc{X}^{(+)}}}), \qquad f_{1}(u) = E(e^{u \br{\vc{\theta}, \vc{X}^{(1)}}}), \qquad u \in \dd{R},
\end{eqnarray*}
  as long as $f(u)$ is finite for each fixed $\vc{\theta} \in \dd{R}$.  Obviously, by condition (iii),
\begin{eqnarray*}
  E(\br{\vc{\theta}, \vc{X}^{(+)}}) = \theta_{1} E(X^{(+)}_{1}) + \theta_{2} E(X^{(+)}_{2})
\end{eqnarray*}
   exists and finite for any $\vc{\theta} \in \dd{R}$. Choose a $\vc{\theta} \in \dd{R}$ such that $\theta_{1} > 0$ and $f(1), f_{1}(1) < \infty$. This $\vc{\theta}$ exists by (iii). Since $f(0) = 1$ and $f(u)$ is convex in $u$, 
\begin{eqnarray}
\label{eqn:theta mean 2}
  \br{\vc{\theta}, \vc{m}} = E(\br{\vc{\theta}, \vc{X}^{(+)}}) = f'(0) < 0
\end{eqnarray}
 is necessary and sufficient to have that $f(u_{0}) < 1$ for some $u_{0} > 0$, which is equivalent to that $u_{0} \vc{\theta} \in \Gamma$ and $u_{0} \theta_{1} > 0$. Let $\vc{\theta}' = u_{0} \vc{\theta}$, then \eqn{theta mean 2} is equivalent to that $\br{\vc{\theta}', \vc{m}} < 0$. Using this $\vc{\theta}'$, we apply the same arguments to function $f_{1}$ and can see that $\br{\vc{\theta}', \vc{m}^{(1)}} < 0$ holds if and only if there is a $u_{1} > 0$ such that $u_{1} \vc{\theta}' \in \Gamma_{1}$ and $u_{1} \theta_{1}' > 0$. This implies that $\min(1,u_{1}) \vc{\theta}' \in \Gamma \cap \Gamma_{1}$ since $\Gamma$ and $\Gamma_{1}$ are convex sets. Thus, the claim is proved. 
 
  We next show that \eqn{theta mean 1} for $k=1$ follows from either one of the stability conditions (\rmn{1}),  (\rmn{2}) and  (\rmn{3}). We first assume (\rmn{1}). Since $m^{(1)}_{2} \ge 0$, we consider the possibility that $m^{(1)}_{2} = 0$. In this case, $m^{(1)}_{1} < 0$ by the first inequality in (\rmn{1}) and $m_{2} < 0$. Hence, $\Gamma_{1} \equiv \{ \vc{\theta} \in \dd{R}^{2}; \varphi_{1}(\vc{\theta}) < 1 \}$ is the region between two straight lines $\theta_{1} = 0$ and $\theta_{1} = a$ for some $a > 0$. Since $\Gamma$ is not empty by (iii), we must have \eqn{theta mean 1} for some $\vc{\theta}$ such that $\theta_{1} > 0$ and $\theta_{2} < 0$. We next assume that $m^{(1)}_{2} > 0$. We put $\vc{\theta} = (\theta_{1}, \theta_{2})$ as
\begin{eqnarray*}
  \theta_{1} = m^{(1)}_{2}, \qquad \theta_{2} = \left\{\begin{array}{ll}
  - m^{(1)}_{1} - \epsilon, \quad & m^{(1)}_{1} \ge 0,\\
  - \epsilon, & m^{(1)}_{1} < 0,
  \end{array} \right.
\end{eqnarray*}
  where $\epsilon > 0$ is chosen so that $\br{\vc{m}, \vc{m}^{(1)}_{\bot}} - \epsilon\, m_{2} < 0$ and $m^{(1)}_{2} m_{1} - \epsilon\, m_{2} < 0$, which is possible by  (\rmn{1}) and $m_{1} < 0$. Note that $\theta_{1} > 0$ and $\theta_{2} < 0$ in this definition. Then, we have
\begin{eqnarray*}
  && \br{\vc{\theta}, \vc{m}} = m^{(1)}_{2} m_{1} - ((m^{(1)}_{1} + \epsilon) 1(m^{(1)}_{1} \ge 0) + \epsilon 1(m^{(1)}_{1} < 0)) m_{2}\\
  && \hspace{8ex} = \left(\br{\vc{m}, \vc{m}^{(1)}_{\bot}} - \epsilon\, m_{2} \right) 1(m^{(1)}_{1} \ge 0) + (m^{(1)}_{2} m_{1}-\epsilon\, m_{2}) 1(m^{(1)}_{1} < 0) < 0,\\
  && \br{\vc{\theta}, \vc{m}^{(1)}} = m^{(1)}_{2} m^{(1)}_{1} - ((m^{(1)}_{1} + \epsilon) 1(m^{(1)}_{1} \ge 0) + \epsilon 1(m^{(1)}_{1} < 0))  m^{(1)}_{2}\\
  && \hspace{8ex} = m^{(1)}_{2} m^{(1)}_{1} 1(m^{(1)}_{1} < 0) - \epsilon\, m^{(1)}_{2} < 0.
\end{eqnarray*}
  Thus, we have \eqn{theta mean 1} for $k=1$.
  
  We next assume (\rmn{2}). We consider the possibility that $m_{1} = 0$. In this case, we put $\vc{\theta} = (\theta_{1}, \theta_{2})$ as
\begin{eqnarray*}
  \theta_{1} = - m_{2} - \epsilon, \qquad \theta_{2} = \epsilon,
\end{eqnarray*}
  where $\epsilon > 0$ is chosen so that $- m_{2} - \epsilon > 0$ and $\br{\vc{m}, \vc{m}^{(1)}_{\bot}} - \epsilon (m^{(1)}_{1} - m^{(1)}_{2}) < 0$, which is possible by  (\rmn{2}). In this case, $\theta_{1}, \theta_{2} > 0$. Then, we have
\begin{eqnarray*}
  && \br{\vc{\theta}, \vc{m}} = - (m_{2} + \epsilon) m_{1} + \epsilon\, m_{2} = \epsilon\, m_{2} < 0, \\
  && \br{\vc{\theta}, \vc{m}^{(1)}} = - (m_{2} + \epsilon)  m^{(1)}_{1} + \epsilon\, m^{(1)}_{2} = \br{\vc{m}, \vc{m}^{(1)}_{\bot}} - \epsilon (m^{(1)}_{1} - m^{(1)}_{2}) < 0.
\end{eqnarray*}
  Thus, we can assume that $m_{1} > 0$. We choose $\epsilon > 0$ such that $\br{\vc{m}, \vc{m}^{(1)}_{\bot}} + \epsilon\, m_{1} < 0$, which is possible by (\rmn{2}), and put
\begin{eqnarray*}
  \theta_{1} = m^{(1)}_{2} + \epsilon, \qquad \theta_{2} = - m^{(1)}_{1}.
\end{eqnarray*}
  Then, $\theta_{1} > 0$ and $\theta_{2} > 0$ since $m^{(1)}_{2} \ge 0$ and $m^{(1)}_{1} < 0$ by (\rmn{2}). Hence, we have
\begin{eqnarray*}
  && \br{\vc{\theta}, \vc{m}} = (m^{(1)}_{2} + \epsilon) m_{1} - m^{(1)}_{1} m_{2} = \br{\vc{m}, \vc{m}^{(1)}_{\bot}} + \epsilon\, m_{1} < 0, \\
  && \br{\vc{\theta}, \vc{m}^{(1)}} = (m^{(1)}_{2} + \epsilon)  m^{(1)}_{1} - m^{(1)}_{1} m^{(1)}_{2} = \epsilon\, m^{(1)}_{1} < 0.
\end{eqnarray*}

  We finally assume (\rmn{3}). In this case, if $m^{(1)}_{2} = 0$, then $m^{(1)}_{1} < 0$, and therefore $\Gamma_{1}$ is the region between $\theta_{1} = 0$ and $\theta_{1} = b$ for some $b > 0$.  Since $m_{1} < 0$ and $m_{2} \ge 0$, we can easily see that \eqn{theta mean 1} holds true for $\theta_{1} > 0$ and $\theta_{2} \le 0$. Thus, we can assume that $m^{(1)}_{2} > 0$, and therefore we can choose $\epsilon > 0$ such that
\begin{eqnarray*}
  \epsilon\, m^{(1)}_{1} + m^{(1)}_{2} m_{1} < 0,
\end{eqnarray*}
  and put $\vc{\theta} = (\epsilon, m_{1})$. Then, $\theta_{1} > 0$, $\theta_{2} < 0$, and
\begin{eqnarray*}
  && \br{\vc{\theta}, \vc{m}} = \epsilon\, m_{1} + m_{1} m_{2} < 0, \\
  && \br{\vc{\theta}, \vc{m}^{(1)}} = \epsilon\, m^{(1)}_{1} + m_{1} m^{(1)}_{2} < 0.
\end{eqnarray*}
  Thus, we have shown \eqn{theta mean 1} for $k=1$. Furthermore, $\theta_{2} \le 0$ for (\rmn{1}) and (\rmn{3}). By symmetric arguments, \eqn{theta mean 1} for $k=2$ is obtained, and $\theta_{1} < 0$ for (\rmn{1}) and (\rmn{2}). Thus, either one of the stability conditions of \lem{stability d=2} implies \eqn{theta mean 1}. The converse is immediate from \eqn{theta mean 1} and the observation that is presented just before this lemma.

\section{Proof of \lem{K c}}
\label{app:K c}
\setnewcounter

Obviously, if either $c_{1}$ or $c_{2}$ vanishes, then $K_{\vc{c}}$ is arithmetic. Hence, we assume that $c_{1} \ne 0$ and $c_{2} \ne 0$. If $c_{1}/c_{2}$ is rational, we obviously see that $K_{\vc{c}}$ is arithmetic. Thus, we only need to prove that $K_{\vc{c}}$ is asymptotically dense at infinity if $c_{1}/c_{2}$ is irrational. For this, we combine the ideas which are used for the proofs of Lemma 2 and Corollary in Section V.4a of \cite{Fell1971}.

Assume that $c_{1}/c_{2}$ is irrational and $c_{1} < c_{2}$. The latter can be assumed without loss of generality because the roles of $c_{1}$ and $c_{2}$ are symmetric. For each positive integer $n$, let
\begin{eqnarray*}
  A(n) = \{ c_{1} m_{1} - c_{2} m_{2} ; 0 \le c_{1} m_{1} - c_{2} m_{2} \le c_{2}, m_{2} \le n, m_{1}, m_{2} \in \dd{Z}_{+} \}.
\end{eqnarray*}
then the number of elements of $A(n)$ is strictly increased as $n$ is increased because of the irrationality. Hence, for each $\epsilon > 0$, we can find positive integer $n$ such that there are $u, u' \in A(n)$ such that $|u - u'| < \epsilon$ because $A(n)$ is a subset of the interval $[0,c_{2}]$. Since we can find $m_{1}, m_{2}, m_{1}', m_{2}'$ such that $m_{1} > m_{1}'$, $u = c_{1} m_{1} - c_{2} m_{2}$ and $u' = c_{1} m_{1}' - c_{2} m_{2}'$, we have
\begin{eqnarray*}
  |c_{1} (m_{1} - m_{1}') - c_{2}( m_{2} - m_{2}')| = |u - u'| < \epsilon.
\end{eqnarray*}
Since $m_{1} > m_{1}'$, we obviously require that $m_{2} > m_{2}'$. Hence, we put $a = c_{2}( m_{2} - m_{2}')$, then, for each $x \ge a$, we have $|x - y| < \epsilon$ for some $y \in K_{\vc{c}}$.

\section{The proof of \lem{stationary equation 2}}
\label{app:stationary inequality}

  We only prove this lemma when (\sect{Double}a) is satisfied because the other cases are similarly proved. We immediately have \eqn{stationary equation 1} if $\varphi_{+}(\vc{\theta}) < \infty$. For proving this finiteness, we apply truncation arguments for \eqn{stationary equation 0}. For each $n=1,2,\ldots$, let
\begin{eqnarray*}
  f_{n}(x) = \min(x,n), \qquad x \in \dd{R}.
\end{eqnarray*}
  If $x \le n$, then $f_{n}(x+y) \le x+y = f_{n}(x) + y$. Otherwise, if $x > n$, then $f_{n}(x+y) \le n = f_{n}(x)$. Hence, for any $x \ge 0$ and $y \in \dd{R}$,
\begin{eqnarray}
\label{eqn:truncation 1}
  f_{n}(x+y) \le f_{n}(x) + \left\{\begin{array}{ll}
  y, \quad & x \le n,\\
  0, & x > n.
  \end{array} \right.
\end{eqnarray}
  From \eqn{stationary equation 0}, we have
\begin{eqnarray}
\label{eqn:stationary equation 0a}
 && \br{\vc{\theta}, \vc{L}} \simeq \br{\vc{\theta}, \vc{L}} + \br{\vc{\theta}, \vc{X}^{(+)}} 1(\vc{L} \in S_{+}) + \sum_{k \in \{0, 1, 2\}} \br{\vc{\theta}, \vc{X}^{(k)}} 1(\vc{L} \in S_{k}).
\end{eqnarray}
  Hence, we have, using the independence of $\vc{L}, \vc{X}^{(0)}, \vc{X}^{(1)}$ and $\vc{X}^{(2)}$,
\begin{eqnarray*}
  E(e^{f_{n}(\br{\vc{\theta}, \vc{L}})}) &\le& E(e^{f_{n}(\br{\vc{\theta}, \vc{L}})} 1(\vc{L} \in S_{+}, \br{\vc{\theta}, \vc{L}} \le n)) E(e^{f_{n}\br{\vc{\theta}, \vc{X}^{(+)}})})\\
  && + E(e^{f_{n}(\br{\vc{\theta}, \vc{L}})} 1(\vc{L} \in S_{+}, \br{\vc{\theta}, \vc{L}} > n))\\
  && + \sum_{k \in \{0, 1, 2\}} E(e^{f_{n}(\br{\vc{\theta}, \vc{L}})} 1(\vc{L} \in S_{k}, \br{\vc{\theta}, \vc{L}} \le n)) E(e^{f_{n}\br{\vc{\theta}, \vc{X}^{(k)}})})\\
  && +\sum_{k \in \{0, 1, 2\}} E(e^{f_{n}(\br{\vc{\theta}, \vc{L}})} 1(\vc{L} \in S_{k}, \br{\vc{\theta}, \vc{L}} > n)).
\end{eqnarray*}
  Rewriting the left side as
\begin{eqnarray*}
  E(e^{f_{n}(\br{\vc{\theta}, \vc{L}})} 1(\vc{L} \in S_{+})) + \sum_{k \in \{0,1,2\}} E(e^{f_{n}(\br{\vc{\theta}, \vc{L}})} 1(\vc{L} \in S_{k})),
\end{eqnarray*}
  we have
\begin{eqnarray*}
  \lefteqn{ (1 - E(e^{f_{n}\br{\vc{\theta}, \vc{X}^{(+)}})})) E(e^{f_{n}(\br{\vc{\theta}, \vc{L}})} 1(\vc{L} \in S_{+}, \br{\vc{\theta}, \vc{L}} \le n))} \hspace{5ex}\\
  && \le \sum_{k \in \{0, 1, 2\}} (E(e^{f_{n}\br{\vc{\theta}, \vc{X}^{(k)}})}) - 1) E(e^{f_{n}(\br{\vc{\theta}, \vc{L}})} 1(\vc{L} \in S_{k}, \br{\vc{\theta}, \vc{L}} \le n)).
\end{eqnarray*}
  Let $n$ go to infinity for this inequality, then the monotone convergence theorem and the finiteness of $\varphi_{1}(\vc{\theta})$ and $\varphi_{2}(\vc{\theta})$ yield
\begin{eqnarray*}
  0 \le (1 - \gamma(\vc{\theta}))\varphi_{+}(\vc{\theta}) \le \sum_{k \in \{1,2\}} (\gamma_{k}(\vc{\theta}) - 1) \varphi_{k}(\theta_{k}) + (\gamma_{0}(\vc{\theta}) - 1) \varphi_{0}(0),
\end{eqnarray*}
  since $\vc{\theta} \in \Gamma$ and $f_{n}(x)$ is nondecreasing in $x$. The right side of this inequality is finite and $1 - \gamma(\vc{\theta}) > 0$, we must have $\varphi_{+}(\vc{\theta}) < \infty$.

\section{Convergence parameter and decay rate}
\label{app:convergence parameter}
\setnewcounter

Let us consider a nonnegative integer valued random variable $Z$ with light tail. Let
\begin{eqnarray}
\label{eqn:finite alpha}
  \alpha^{*} = \sup\{\alpha \ge 0; E \exp(\alpha Z) < \infty \}.
\end{eqnarray}
  $\alpha^{*} > 0$ by the light tail condition. We give an example such that 
\begin{eqnarray}
\label{eqn:rough rate}
  \lim_{x \to \infty} x^{-1} \log P(Z > x) = -\alpha^{*}
\end{eqnarray}
is not true. One can easily see that
\begin{eqnarray*}
  \limsup_{x \to \infty} \frac 1x \log P( Z > x) = - \alpha^{*}.
\end{eqnarray*}
Hence, the problem is the limit infimum. Define the distribution function $F$ of a random variable $Z$ by
\begin{eqnarray*}
  \ol{F}(x) = \left\{\begin{array}{ll}
  1, \quad & x \le 1,\\
  e^{-\alpha^{*} 2^{n}}, \quad & 2^{n-1} < x \le 2^{n}, n= 1,\ldots,
  \end{array} \right.
\end{eqnarray*}
where $\ol{F}(x) = 1 - F(x)$, then
\begin{eqnarray*}
 && \lim_{n \to \infty} \frac 1{2^{n}} \log \ol{F}(2^{n}) = \lim_{n \to \infty} \frac 1{2^{n}} (-\alpha^{*} 2^{n}) = - \alpha^{*},\\
 && \lim_{n \to \infty} \frac 1{2^{n}+1} \log \ol{F}(2^{n}+1) = \lim_{n \to \infty} \frac 1{2^{n} + 1} (-\alpha^{*} 2^{n+1}) = - 2 \alpha^{*}.
\end{eqnarray*}
  Thus, \eqn{finite alpha} holds, but we have
\begin{eqnarray*}
  \liminf_{x \to \infty} \frac 1x \log P( Z > x) = - 2 \alpha^{*} < - \alpha^{*} = \limsup_{x \to \infty} \frac 1x \log P( Z > x).
\end{eqnarray*}
  The distribution function $F(x)$ is a little tricky because its increasing points are sparse as $x$ goes to infinity. However, it is not very difficult to make a small change for it to increase at all positive integers. Similar examples are obtained in the literature (e.g., see Section 2.3 of \cite{Naka2004}). 
  
\section{Proof of \lem{tau}}
\label{app:tau}
\setnewcounter

  We first prove $\vc{\theta}^{(\triangle, \infty)} = \vc{\tau}$ for the three cases separately.  For (D1), suppose that $\theta^{(k, \infty)}_{k} < \theta^{(k, \cp)}_{k}$ for $k=1,2$. Note that $\theta^{(2, \infty)}_{1} < \theta^{(1, \infty)}_{1}$ and $\theta^{(1, \infty)}_{2} < \theta^{(2, \infty)}_{2}$ hold by (D1). Since $\vc{\theta}^{(1, \infty)} \le \vc{\theta}^{(\triangle, \infty)}$ and $\vc{\theta}^{(2, \infty)} \le \vc{\theta}^{(\triangle, \infty)}$, at least one of $\theta^{(1, \infty)}_{1}$ and $\theta^{(2, \infty)}_{2}$ must be increased by the right hand side of \eqn{theta-infty}. This contradicts the supposition. Hence, either $\theta^{(1, \infty)}_{1} = \theta^{(1, \cp)}_{1}$ or $\theta^{(2, \infty)}_{2} = \theta^{(2, \cp)}_{2}$ holds. Suppose that $\theta^{(1, \infty)}_{1} = \theta^{(1, \cp)}_{1}$ and $\theta^{(2, \infty)}_{2} < \theta^{(2, \cp)}_{2}$. Then, from condition (D1), $\theta^{(2, \infty)}_{2}$ must be increased again by the right hand side of \eqn{theta-infty}. This is a contradiction. Similarly, it is impossible that $\theta^{(1, \infty)}_{1} < \theta^{(1, \cp)}_{1}$ but $\theta^{(2, \infty)}_{2} = \theta^{(2, \cp)}_{2}$. Thus, we must have that $\theta^{(k, \infty)}_{k} = \theta^{(k, \cp)}_{k}$ for $k=1,2$.
  
  For (D2), we can apply similar arguments as above if we replace $\theta^{(1, \cp)}_{1}$ by $\ol{\xi}_{1}(\theta^{(2, \cp)}_{2})$. For (D3), we replace $\theta^{(2, \cp)}_{2}$ by $\ol{\xi}_{2}(\theta^{(1, \cp)}_{1})$. Thus, we get \eqn{tau} for all the three cases since $\tau_{1} = \theta^{(1, \infty)}_{1}$ and $\tau_{2} = \theta^{(2, \infty)}_{2}$.
  
    The remaining part of this lemma is immediate since $\varphi(\vc{\theta}) < \infty$ for all $\vc{\theta} < (\theta_{1}^{(\triangle, n)}, \ul{\xi}_{2}(\theta_{1}^{(\triangle, n)}))$ and for all $\vc{\theta} < (\ul{\xi}_{1}(\theta_{2}^{(\triangle, n}), \theta_{2}^{(\triangle, n)})$ are inductively obtained by \lem{stationary equation 2}.

\section{The proof of \lem{infinite domain}}
\label{app:infinite domain}
\setnewcounter

  We will use the random walk $\{\vc{Y}(\ell)\}$ introduced in \sectn{Double}. We apply the permutation arguments in Lemma 5.6 of \cite{BoroMogu2001} twice. Then, we have, for any positive integer $n$ and any $\vc{x} > \vc{0}$,
\begin{eqnarray}
\label{eqn:lower bound 2}
  \lefteqn{P( \vc{Y}(n) \in \vc{x} + \Delta(\vc{a}), \min_{1 \le \ell \le n} Y_{1}(\ell) > 0, \min_{1 \le \ell \le n} Y_{2}(\ell) > 0 | \vc{Y}(0) = \vc{0} )} \hspace{10ex}\\
  && \ge \frac 1n P( \vc{Y}(n) \in \vc{x} + \Delta(\vc{a}), \min_{1 \le \ell \le n} Y_{2}(\ell) > 0 | \vc{Y}(0) = \vc{0} ) \nonumber\\
  && \ge \frac 1{n^{2}} P( \vc{Y}(n) \in \vc{x} + \Delta(\vc{a}) | \vc{Y}(0) = \vc{0}).\nonumber
\end{eqnarray}
  We next note the well known Cram\'{e}r's theorem (e.g., see Theorem 2 of \cite{BoroMogu2000} and Section 3.5 of \cite{DupuElli1997}).
\begin{eqnarray}
\label{eqn:Cramer 1}
  \lim_{n \to \infty} \frac 1{n} \log P(\vc{Y}(n) \in n \vc{x} + \Delta(\vc{a}))) = - \Lambda(\vc{x}),
\end{eqnarray}
  where $\Lambda(\vc{x}) = \sup_{\vc{\theta} \in \dd{R}^{2}} \{ \br{\vc{\theta}, \vc{x}} - \log \varphi(\vc{\theta}) \}$.
  
  Since the random walk $\{\vc{Y}(\ell)\}$ is identical that of $\{\vc{L}(\ell)\}$ as long as they are inside of the quadrant $S$, \eqn{lower bound 2} can be written as, for $\vc{y} \in S_{+}$,
\begin{eqnarray}
\label{eqn:lower bound 3}
  P( \vc{L}(n) \in \vc{x} + \Delta(\vc{a}), \sigma_{0} > n | \vc{L}(0) = \vc{y} ) \ge \frac 1{n^{2}} P( \vc{Y}(n) \in \vc{x} - \vc{y} + \Delta(\vc{a}) | \vc{Y}(0) = \vc{0}). 
\end{eqnarray}
  where $\sigma_{0} = \inf\{ \ell \ge 1; \vc{L}(\ell) \in \partial S \}$. It follows from the representation \eqn{occupation measure 1} with $B = \partial S$ for the stationary distribution that there are some $\vc{y}_{0} \in \partial S$ and $\vc{y}_{1} \in S_{+}$ such that $p(\vc{y}_{0}, \vc{y}_{1}) > 0$ and, for any $m \ge 1$,
\begin{eqnarray*}
  \lefteqn{P( \vc{L} \in n \vc{c} + \Delta(\vc{a}) )}\\
  && = \frac 1{E_{\nu}(\sigma_{0})} \sum_{\vc{y} \in \partial S} \sum_{\vc{y}' \in S_{+}} \nu(\vc{y}) p(\vc{y}, \vc{y}') \sum_{\ell = 1}^{\infty} P( \vc{L}(\ell) \in n \vc{c} + \Delta(\vc{a}), \sigma_{0} > \ell| \vc{L}(0) = \vc{y}' ) \\
  && \ge \frac 1{E_{\nu}(\sigma_{0})} P( \vc{L}(m) \in n \vc{c} + \Delta(\vc{a}), \sigma_{0} > m| \vc{L}(0) = \vc{y}_{1} ) \nu(\vc{y}_{0}) p(\vc{y}_{0}, \vc{y}_{1})\\
  && \ge \frac 1{m^{2} E_{\nu}(\sigma_{0})} P( \vc{Y}(m) \in n \vc{c} - \vc{y}_{1} + \Delta(\vc{a}) | \vc{Y}(0) = \vc{0}) \nu(\vc{y}_{0}) p(\vc{y}_{0}, \vc{y}_{1}).
\end{eqnarray*}

  Thus, for each $t > 0$, letting $m, n \to \infty$ in such a way that $n/m \to t$, we have
\begin{eqnarray*}
  \lim_{n \to \infty} \frac 1{n} \log P( \vc{L} \in n \vc{c} + \Delta(\vc{a}) ) &\ge&  \lim_{n \to \infty} \frac m{n} \frac 1m \log P\left( \vc{Y}(m) \in m \frac nm \vc{c} - \vc{y}_{1} + \Delta(\vc{a}) \right) \\
  &=& - \frac 1t \Lambda(t\vc{c}).
\end{eqnarray*}
  Since $t> 0$ can be arbitrary, this implies that
\begin{eqnarray*}
  \lim_{n \to \infty} \frac 1{n} \log P( \vc{L} \in n \vc{c} + \Delta(\vc{a}) ) \ge - \inf_{ t > 0} \frac 1t \Lambda(t\vc{c}) = - \sup\{ \br{\vc{\theta}, \vc{c}}; \gamma(\vc{\theta}) \le 1 \},
\end{eqnarray*}
  where the last equality is obtained from Theorem 1 of \cite{BoroMogu1996a} (see also Theorem 13.5 of \cite{Rock1970}).
  
\begin{figure}[h]
 	\centering
	\includegraphics[height=4.5cm]{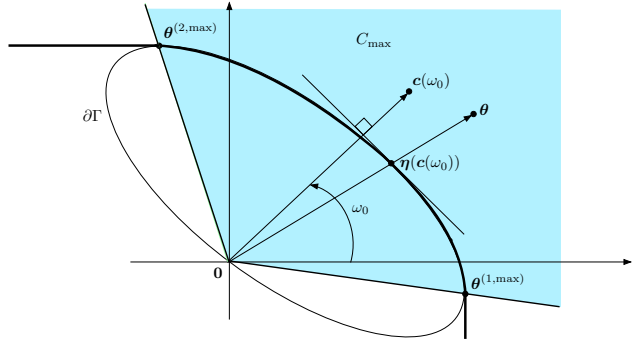}
	\caption{The positions of $\vc{\theta}$, $\vc{c}(\omega_{0})$, $\vc{\eta}(\vc{c}(\omega_{0}))$ and the cone $C_{\max}$ (blue colored area)}
	\label{fig:domain 1}
\end{figure}

  It remains to prove that $\vc{\theta} \not\in \ol{\Gamma}_{\max}$ implies $\varphi(\vc{\theta}) = \infty$. Define the cone $C_{max}$ as
\begin{eqnarray*}
  C_{max} = \{\vc{x} \in \dd{R}^{2}; \vc{x} = s \vc{\theta}^{(1, \max)} + t \vc{\theta}^{(2, \max)}, s,t \ge 0\}.
\end{eqnarray*}
  If $\vc{\theta} \not\in C_{\max} \setminus \ol{\Gamma}_{\max}$, then either $\theta_{1} > \theta^{(1,\max)}_{1}$ or $\theta_{2} > \theta^{(2,\max)}_{2}$ holds. Hence, $\varphi(\vc{\theta}) = \infty$ in this case by \lem{lower bound c}. Thus, we only need to consider $\vc{\theta} \in C_{\max} \setminus \ol{\Gamma}_{\max}$.
  
  Since $\ol{\Gamma}$ is a closed convex set that contains $\vc{0}$, there exists a unique $\vc{\eta} \in \ol{\Gamma}$ that maximizes $\br{\vc{\eta}, \vc{c}}$ for each $\vc{c}$. Denote this $\vc{\eta}$ by $\vc{\eta}(\vc{c})$. That is,
\begin{eqnarray*}
  \br{\vc{\eta}(\vc{c}), \vc{c}} = \sup\{ \br{\vc{\eta}, \vc{c}}; \gamma(\vc{\eta}) \le 1 \}.
\end{eqnarray*}
  Let $\vc{c}(\omega) = (\cos \omega, \sin \omega)$, then the $\vc{\eta}(\vc{c}(\omega))$ continuously moves on $\partial \Gamma \cap C_{\max}$ from $\vc{\theta}^{(1, \max)}$ to $\vc{\theta}^{(2, \max)}$ as $\omega$ is increased on $(-\frac 12 \pi, \pi)$. Then, as can been seen in \fig{domain 1}, there is an $\omega_{0} \in (-\frac 12 \pi, \pi)$ for $\vc{\theta} \in C_{\max} \setminus \ol{\Gamma}_{\max}$ such that $\vc{\eta}(\vc{c}(\omega_{0}))$ has the same direction as $\vc{\theta}$, which implies that  $\vc{\theta} = a \vc{\eta}(\vc{c}(\omega_{0}))$ for some $a > 1$. For this $\omega_{0}$, we have
\begin{eqnarray*}
  \br{\vc{\theta}, \vc{c}(\omega_{0})} = a \br{\vc{\eta}(\vc{c}(\omega_{0})), \vc{c}(\omega_{0})} > \br{\vc{\eta}(\vc{c}(\omega_{0})), \vc{c}(\omega_{0})}.
\end{eqnarray*}
   Hence, the Markov's inequality,
\begin{eqnarray}
\label{eqn:Markov inequality 1}
  \varphi(\vc{\theta}) \ge e^{n \br{\vc{\theta}, \vc{c}(\omega_{0})}} P( \vc{L} \in n \vc{c}(\omega_{0}) + \Delta(\vc{a})),
\end{eqnarray}
  and \eqn{lower bound 1} yield that
\begin{eqnarray*}
  \liminf_{n \to \infty} \frac 1n \log \varphi(\vc{\theta}) \ge \br{\vc{\theta}, \vc{c}(\omega_{0})} - \br{\vc{\eta}(\vc{c}(\omega_{0})), \vc{c}(\omega_{0})} > 0.
\end{eqnarray*}
  This concludes $\varphi(\vc{\theta}) = \infty$, which completes the proof.

\section{Proof of \eqn{H limit}}
\label{app:renewal theorem}
\setnewcounter

  For changing measures for the random walk $\{\vc{Y}(\ell)\}$, we define $\{\hat{\vc{Y}}^{(1)}(\ell)\}$ by
\begin{eqnarray*}
   P(\hat{\vc{Y}}^{(1)}(\ell+1) = (n,j) | \hat{\vc{Y}}^{(1)}(\ell) = (m,i)) = e^{\tau_{1} (n-m) - \tau_{2} (i - j)} P(\vc{X}^{(+)} = (n-m, -(i-j))),
\end{eqnarray*}
  which is well defined since because $\gamma(\vc{\tau}) = 1$ because of (D2). Note that, for $k=1,2$,
\begin{eqnarray*}
  E\left( \left. \hat{Y}^{(1)}_{k}(\ell+1) \right| \hat{Y}^{(1)}_{k}(\ell) = 0 \right) = (-1)^{k-1} E\left( X^{(+)}_{k} e^{\br{\vc{\tau}, \vc{X}^{(+)}}} \right) = (-1)^{k-1} \left. \frac {\partial}{\partial \theta_{k}} \gamma(\vc{\theta}) \right|_{\vc{\theta} = \vc{\tau}}
\end{eqnarray*}
  is finite and positive because of (D2) and the assumption that $\tau_{1} < \theta^{(1, \cp)}_{1}$. We denote this expectation by $\mu_{k}$.
  
  Recall that $\tilde{\sigma}^{(1)}_{1}(n)$ is the $n$-th increasing instant of the Markov additive process $\{\tilde{\vc{Z}}^{(1)}(\ell)\}$. We introduce similar instants for $\{\hat{\vc{Y}}(\ell)\}$. Let
\begin{eqnarray*}
  \hat{\zeta}^{(1)}_{1}(n) = \inf \{\ell \ge \hat{\zeta}^{(1)}_{1}(n-1); \hat{Y}^{(1)}_{1}(\ell) -  \hat{Y}^{(1)}_{1}(\hat{\zeta}^{(1)}_{1}(n-1)) \ge 1 \}.
\end{eqnarray*}  
  For convenience, we also introduce the following events. Let, for $1 \le n_{0} \le n$, $m \ge 1$ and $j \ge 0$,
\begin{eqnarray*}
 && \tilde{A}^{\vc{Z}}_{n_{0},\ell}(m, k) = \{\tilde{\vc{Z}}^{(1)}(\tilde{\sigma}^{(1)}_{1}(\ell)) - \tilde{\vc{Z}}^{(1)}(\tilde{\sigma}^{(1)}_{1}(n_{0})) = (m,k) \}, \\
 && \hat{A}^{\vc{Y}}_{n_{0},\ell}(m, k) = \{\tilde{\vc{Y}}^{(1)}(\tilde{\sigma}^{(1)}_{1}(\ell)) - \tilde{\vc{Y}}^{(1)}(\tilde{\sigma}^{(1)}_{1}(n_{0})) = (m,k) \}, \\  
 && \tilde{B}^{Z_{1}}_{n_{0},n}(m) = \cup_{\ell = n_{0}+1}^{n} \cup_{k=0}^{\infty} \tilde{A}^{\vc{Z}}_{n_{0},\ell}(m, k), \qquad \hat{B}^{Y_{1}}_{n_{0},n}(m) = \cup_{\ell = n_{0}+1}^{n} \cup_{k=-\infty}^{\infty} \hat{A}^{\vc{Y}}_{n_{0},\ell}(m, k),\\
 && \tilde{C}^{Z_{2}}_{n_{0},n}(j) = \Big\{\inf_{\tilde{\sigma}^{(1)}_{1}(n_{0}) < \ell \le \tilde{\sigma}^{(1)}_{1}(n)} \tilde{Z}^{(1)}_{2}(\ell) \ge 1, \tilde{Z}^{(1)}_{2}(\tilde{\sigma}^{(1)}_{1}(n_{0})) = j \Big\},\\
 && \hat{C}^{Y_{2}}_{n_{0},n}(j) = \Big\{\inf_{\hat{\zeta}^{(1)}_{1}(n_{0}) < \ell \le \hat{\zeta}^{(1)}_{1}(n)} \hat{Y}^{(1)}_{2}(\ell) \ge 1, \hat{Y}^{(1)}_{2}(\hat{\zeta}^{(1)}_{1}(n_{0})) = j \Big\}.
\end{eqnarray*}
  Since $\tilde{A}^{(1)}$ is transient, $\tilde{Z}^{(1)}_{2}(n)$ goes to $+\infty$ as $n \to \infty$ with probability one. This fact is alternatively verified by using $\{\hat{\vc{Y}}^{(1)}(\ell)\}$ (see \eqn{zeta Z 5} and \eqn{x asymptotic} below). Similarly, for $k=1,2$, $\hat{Y}^{(1)}_{k}(n)$ diverges as $n \to \infty$ by $\mu_{k} > 0$. Hence, we can find a positive integer $n_{0}$ for any $i, j_{0}, k_{0}, \ell_{0} \ge 1$ such that, for any $n \ge n_{0}$,
\begin{eqnarray}
\label{eqn:zeta Z 1}
 && \sum_{j \ge j_{0}} P(\tilde{C}_{n_{0}, n}^{Z_{2}}(j)|\tilde{Z}^{(1)}_{2} (0) = i) > 1 - \epsilon,\\
\label{eqn:zeta Z 2}
 && \sum_{j \ge j_{0}} \sum_{k \ge k_{0}} P(\hat{C}_{n_{0},n}^{Y_{2}}(j) \cap \{\hat{Y}^{(1)}_{2}(\hat{\zeta}^{(1)}_{1}(n)) = k\} | \hat{\vc{Y}}^{(1)}(0) = \vc{0} ) > 1 - \epsilon,\\
\label{eqn:zeta Z 3}
 && P( \tilde{\vc{Z}}^{(1)}(\tilde{\sigma}^{(1)}_{1}(n)) \ge (\ell_{0},j_{0}) |\tilde{Z}^{(1)}_{2} (0) = i) > 1 - \epsilon.
\end{eqnarray}
We further choose $n_{1} > n_{0}$ such that, for any $n \ge n_{1}$,
\begin{eqnarray}
\label{eqn:zeta Z 4}
  P(\tilde{B}_{0,n_{0}}^{Z_{1}}(n) |\tilde{Z}^{(1)}_{2} (0) = i) < \epsilon.
\end{eqnarray}
  because $\tilde{Z}^{(1)}_{1}(\tilde{\sigma}^{(1)}_{1}(n_{0}))$ is finite with probability 1.

  We now compute, for $n \ge n_{0}$ and $j \ge 1$,
\begin{eqnarray}
\label{eqn:zeta Z 5}
  \lefteqn{P( \tilde{B}_{n_{0},n}^{Z_{2}}(n) \cap \tilde{C}^{Z_{2}}_{n_{0}, n}(j) |\tilde{Z}^{(1)}_{2} (0) = i)} \nonumber \\
 &&  = \sum_{\ell \ge 1} P( \tilde{B}_{n_{0},n}^{Z_{1}}(n) \cap \tilde{C}^{Z_{2}}_{n_{0}, n}(j) | \tilde{\vc{Z}}^{(1)}(\tilde{\sigma}^{(1)}_{1}(n_{0})) = (\ell,j)) \nonumber \\
 && \hspace{25ex} \times P( \tilde{\vc{Z}}^{(1)}(\tilde{\sigma}^{(1)}_{1}(n_{0})) = (\ell,j) |\tilde{Z}^{(1)}_{2} (0) = i) \nonumber \\
 &&  = \sum_{\ell \ge 1} \sum_{k \ge 1} \sum_{s = n_{0}+1}^{n} \frac {e^{-\tau_{2}k}} {x_{k}} P(\hat{A}_{n_{0},s}^{\vc{Y}}(n,k) \cap \hat{C}^{Y_{2}}_{n_{0}, n}(j) ) \frac {x_{j}} {e^{-\tau_{2}j}} \nonumber \\
 && \hspace{25ex} \times P( \tilde{\vc{Z}}^{(1)}(\tilde{\sigma}^{(1)}_{1}(n_{0})) = (\ell,j) |\tilde{Z}^{(1)}_{2} (0) = i),
\end{eqnarray}
  where the last equality is obtained from the definitions of $\{\tilde{\vc{Z}}^{(1)}(n)\}$ and $\{\hat{\vc{Y}}^{(1)}(n)\}$ and the fact that $\{\hat{\vc{Y}}^{(1)}(n)\}$ is a random walk.

We next consider to apply the renewal theorem. For this let
\begin{eqnarray*}
  \mu^{(1)}_{\sigma} = E(\hat{Y}^{(1)}_{1}(\hat{\zeta}^{(1)}_{1}(1))|\hat{Y}^{(1)}_{1}(0) = 0)
\end{eqnarray*}
  then $\mu^{(1)}_{\sigma}$ is finite because $\hat{Y}^{(1)}_{1}(n)$ has drifts to $+\infty$ by $\mu^{(1)} > 0$ and its increments have a finite expectation (see Theorem 2.4 in Chapter VIII of \cite{Asmu2003}). Hence, it follows from the renewal theorem that
\begin{eqnarray*}
  \lim_{n \to \infty} \sum_{\ell=1}^{n} P\left( \hat{Y}^{(1)}_{1}(\hat{\zeta}^{(1)}_{1}(\ell)) = n \right) = \frac 1{\mu^{(1)}_{\sigma}}.
\end{eqnarray*}
  Obviously, this yields, for each fixed $n_{0} \ge 0$,
\begin{eqnarray}
\label{eqn:renewal 1}
  \lim_{n \to \infty} P\left( \hat{B}^{Y_{1}}_{n_{0},n}(n) \right) = \frac 1{\mu^{(1)}_{\sigma}}.
\end{eqnarray}

  From \eqn{zeta Z 1} and \eqn{zeta Z 4}, it follows that, for sufficiently large $n$,
\begin{eqnarray}
\label{eqn:zeta Z 6}
  \left|P( \tilde{B}_{n_{0},n}^{Z_{1}}(n)|\tilde{Z}^{(1)}_{2} (0) = i) - \sum_{j \ge j_{0}} P( \tilde{B}_{n_{0},n}^{Z_{1}}(n) \cap \tilde{C}^{Z_{1}}_{n_{0}, n}(j) |\tilde{Z}^{(1)}_{2} (0) = i) \right| < \epsilon.
\end{eqnarray}
  Similarly,
\begin{eqnarray}
\label{eqn:zeta Z 7}
  \left|P( \hat{B}_{n_{0},n}^{Y_{1}}(n)) - \sum_{j \ge j_{0}} P( \hat{B}_{n_{0},n}^{Y_{1}}(n) \cap \hat{C}^{Y_{1}}_{n_{0}, n}(j) ) \right| < \epsilon.
\end{eqnarray}

  By \eqn{x i convergence}, we can choose sufficiently large $ j_{0}, k_{0}$ such that
\begin{eqnarray}
\label{eqn:x asymptotic}
  \left| 1 - \frac {e^{-\tau_{2}k}} {x_{k}} \right| \left|1 - \frac {x_{j}} {e^{-\tau_{2}j}} \right| < \epsilon, \qquad j \ge j_{0}, k \ge k_{0}.
\end{eqnarray}
  We then sum both side of \eqn{zeta Z 5} for $j \ge j_{0}$. Further, using \eqn{zeta Z 2} and \eqn{zeta Z 3} for sufficiently large $\ell_{0}$, we apply \eqn{renewal 1} to the sum as $n \to \infty$, then letting $\epsilon \to 0$ yield
\begin{eqnarray*}
  \lim_{n \to \infty} \sum_{j \ge j_{0}} P( \tilde{B}_{n_{0},n}^{Z_{1}}(n) \cap \tilde{C}^{Z_{1}}_{n_{0}, n}(j) |\tilde{Z}^{(1)}_{2} (0) = i) = \frac 1{\mu^{(1)}_{\sigma}}.
\end{eqnarray*}
  Hence, by \eqn{zeta Z 4}, \eqn{zeta Z 6} and \eqn{zeta Z 7}, we have
\begin{eqnarray*}
  \limsup_{n \to \infty} \left|P( \tilde{B}_{n_{0},n}^{Z_{1}}(n)|\tilde{Z}^{(1)}_{2} (0) = i) - \frac 1{\mu^{(1)}_{\sigma}} \right| < \epsilon.
\end{eqnarray*}
  Thus, we have \eqn{H limit} by letting $\epsilon \downarrow 0$.

\section*{Acknowledgements} This work is supported in part by Japan Society for the Promotion of Science under grant No.\ 24310115. It has been partly done during the second author visited the Isaac Newton Institute in Cambridge under the program titled, Stochastic Processes in Communication Sciences, from February and March in 2010. The second author is grateful to the institute and program committee for providing stimulus environment for research. 

\medskip

\bibliographystyle{ims}
\bibliography{dai05152013a1}

\def\cprime{$'$} \def\cprime{$'$} \def\cprime{$'$} \def\cprime{$'$}
  \def\cprime{$'$} \def\cprime{$'$} \def\cprime{$'$}
\begin{thebibliography}{34}
\expandafter\ifx\csname natexlab\endcsname\relax\def\natexlab#1{#1}\fi
\expandafter\ifx\csname url\endcsname\relax
  \def\url#1{\texttt{#1}}\fi
\expandafter\ifx\csname urlprefix\endcsname\relax\def\urlprefix{URL }\fi
\providecommand{\eprint}[2][]{\url{#2}}

\bibitem[{Arjas and Speed(1973)}]{ArjaSpee1973}
\textsc{Arjas, E.} and \textsc{Speed, T.~P.} (1973).
\newblock Symmetric wiener-hopf factorisations in {Markov} additive processes.
\newblock \textit{Probability Theory and Related Fields}, \textbf{26} 105--118.

\bibitem[{Asmussen(2003)}]{Asmu2003}
\textsc{Asmussen, S.} (2003).
\newblock \textit{Applied probability and queues}, vol.~51 of
  \textit{Applications of Mathematics (New York)}.
\newblock 2nd ed. Springer-Verlag, New York.
\newblock Stochastic Modelling and Applied Probability.

\bibitem[{Borovkov and Mogul{\cprime}ski{\u\i}(1996)}]{BoroMogu1996a}
\textsc{Borovkov, A.~A.} and \textsc{Mogul{\cprime}ski{\u\i}, A.~A.} (1996).
\newblock The second function of deviations and asymptotic problems of the
  reconstruction and attainment of a boundary for multidimensional random
  walks.
\newblock \textit{Sibirsk. Mat. Zh.}, \textbf{37} 745--782, i.
\newblock \urlprefix\url{http://dx.doi.org/10.1007/BF02104660}.

\bibitem[{Borovkov and Mogul{\cprime}ski{\u\i}(2000)}]{BoroMogu2000}
\textsc{Borovkov, A.~A.} and \textsc{Mogul{\cprime}ski{\u\i}, A.~A.} (2000).
\newblock Integro-local limit theorems for sums of random vectors that include
  large deviations. {II}.
\newblock \textit{Teor. Veroyatnost. i Primenen.}, \textbf{45} 5--29.
\newblock \urlprefix\url{http://dx.doi.org/10.1137/S0040585X97978026}.

\bibitem[{Borovkov and Mogul{\cprime}ski{\u\i}(2001)}]{BoroMogu2001}
\textsc{Borovkov, A.~A.} and \textsc{Mogul{\cprime}ski{\u\i}, A.~A.} (2001).
\newblock Large deviations for {Markov} chains in the positive quadrant.
\newblock \textit{Russian Mathematical Surveys}, \textbf{56} 803--916.
\newblock \urlprefix\url{http://dx.doi.org/10.1070/RM2001v056n05ABEH000398}.

\bibitem[{Dai and Miyazawa(2011)}]{DaiMiya2011}
\textsc{Dai, J.~G.} and \textsc{Miyazawa, M.} (2011).
\newblock Reflecting brownian motion in two dimensions: Exact asymptotics for
  the stationary distribution.
\newblock \textit{Stochastic Systems}, \textbf{1} 146--208.

\bibitem[{Dai and Miyazawa(2013)}]{DaiMiya2013}
\textsc{Dai, J.~G.} and \textsc{Miyazawa, M.} (2013).
\newblock Stationary distribution of a two-dimensional srbm: geometric views
  and boundary measures.
\newblock \textit{Queueing Systems}, \textbf{74} 181--217.

\bibitem[{Doetsch(1974)}]{Doet1974}
\textsc{Doetsch, G.} (1974).
\newblock \textit{Introduction to the theory and application of the {L}aplace
  transformation}.
\newblock Springer-Verlag, New York.
\newblock Translated from the second German edition by Walter Nader.

\bibitem[{Dupuis and Ellis(1997)}]{DupuElli1997}
\textsc{Dupuis, P.} and \textsc{Ellis, R.~S.} (1997).
\newblock \textit{A weak convergence approach to the theory of large
  deviations}.
\newblock Wiley Series in Probability and Statistics: Probability and
  Statistics, John Wiley \& Sons Inc., New York.
\newblock A Wiley-Interscience Publication.

\bibitem[{Fayolle(1989)}]{Fayo1989}
\textsc{Fayolle, G.} (1989).
\newblock On random walks arising in queueing systems: as {L}yapounov
  functions-part 1.
\newblock \textit{Queueing Systems}, \textbf{5} 167--184.

\bibitem[{Fayolle et~al.(1999)Fayolle, Iasnogorodski and
  Malyshev}]{FayoIasnMaly1999}
\textsc{Fayolle, G.}, \textsc{Iasnogorodski, R.} and \textsc{Malyshev, V.}
  (1999).
\newblock \textit{Random Walks in the Quarter-Plane: Algebraic Methods,
  Boundary Value Problems and Applications}.
\newblock Springer, New York.

\bibitem[{Fayolle et~al.(1995)Fayolle, Malyshev and
  Menshikov}]{FayoMalyMens1995}
\textsc{Fayolle, G.}, \textsc{Malyshev, V.} and \textsc{Menshikov, M.} (1995).
\newblock \textit{Topics in the constructive theory of countable {Markov}
  chains}.
\newblock Cambridge University Press, Cambridge, UK.

\bibitem[{Feller(1971)}]{Fell1971}
\textsc{Feller, W.} (1971).
\newblock \textit{An introduction to probability theory and its applications.
  {v}ol. {II}.}
\newblock Second edition, John Wiley \& Sons Inc., New York.

\bibitem[{Flajolet and Sedgewick(2009)}]{FlajSedg2009}
\textsc{Flajolet, P.} and \textsc{Sedgewick, R.} (2009).
\newblock \textit{Analytic combinatorics}.
\newblock Cambridge University Press, Cambridge, UK.

\bibitem[{Flatto and Hahn(1984)}]{FlatHahn1984}
\textsc{Flatto, L.} and \textsc{Hahn, S.} (1984).
\newblock Two parallel queues created by arrivals with two demands i.
\newblock \textit{SIAM Journal on Applied Mathematics}, \textbf{44} 1041--1053.

\bibitem[{Foley and McDonald(2005)}]{FoleMcDo2005}
\textsc{Foley, R.~D.} and \textsc{McDonald, D.~R.} (2005).
\newblock Large deviations of a modified {J}ackson network: stability and rough
  asymptotics.
\newblock \textit{Ann. Appl. Probab.}, \textbf{15} 519--541.
\newblock \urlprefix\url{http://dx.doi.org/10.1214/105051604000000666}.

\bibitem[{Guillemin et~al.(2013)Guillemin, Knessl and van
  Leeuwaarden}]{GuilKnesLeeu2013}
\textsc{Guillemin, F.}, \textsc{Knessl, C.} and \textsc{van Leeuwaarden, J.}
  (2013).
\newblock Wireless 3-hop networks with stealing {II}: exact solutions through
  boundary value problems.
\newblock \textit{Queueing Systems}, \textbf{74} 235--272.

\bibitem[{Guillemin and Pinchon(2004)}]{GuilPinc2004}
\textsc{Guillemin, F.} and \textsc{Pinchon, D.} (2004).
\newblock Analysis of generalized processor- sharing systems with two classes
  of customers and exponential services.
\newblock \textit{Journal of Applied Probability}, \textbf{41} 832--858.

\bibitem[{Guillemin and van Leeuwaarden(2011)}]{GuilLeeu2011}
\textsc{Guillemin, F.} and \textsc{van Leeuwaarden, J.~S.} (2011).
\newblock Rare event asymptotics for a random walk in the quarter plane.
\newblock \textit{Queueing Systems}, \textbf{67} 1--32.

\bibitem[{Kobayashi and Miyazawa(2012)}]{KobaMiya2012}
\textsc{Kobayashi, M.} and \textsc{Miyazawa, M.} (2012).
\newblock Revisit to the tail asymptotics of the double {QBD} process:
  Refinement and complete solutions for the coordinate and diagonal directions.
\newblock In \textit{Matrix-Analytic Methods in Stochastic Models} (G.~Latouche
  and M.~S. Squillante, eds.). Springer, 147--181.
\newblock ArXiv:1201.3167.

\bibitem[{Kobayashi et~al.(2010)Kobayashi, Miyazawa and
  Zhao}]{KobaMiyaZhao2010}
\textsc{Kobayashi, M.}, \textsc{Miyazawa, M.} and \textsc{Zhao, Y.~Q.} (2010).
\newblock Tail asymptotics of the occupation measure for a {M}arkov additive
  process with an {$M/G/1$}-type background process.
\newblock \textit{Stochastic Models}, \textbf{26} 463--486.

\bibitem[{Koshevoy and Mosler(1998)}]{KoshMosl1998}
\textsc{Koshevoy, G.} and \textsc{Mosler, K.} (1998).
\newblock Lift zonoids, random convex hulls and the variability of random
  vectors.
\newblock \textit{Bernoulli}, \textbf{4} 377--399.

\bibitem[{Li et~al.(2007)Li, Miyazawa and Zhao}]{LiMiyaZhao2007}
\textsc{Li, H.}, \textsc{Miyazawa, M.} and \textsc{Zhao, Y.~Q.} (2007).
\newblock Geometric decay in a {QBD} process with countable background states
  with applications to a join-the-shortest-queue model.
\newblock \textit{Stoch. Models}, \textbf{23} 413--438.
\newblock \urlprefix\url{http://dx.doi.org/10.1080/15326340701471042}.

\bibitem[{Li and Zhao(2011)}]{LiZhao2011}
\textsc{Li, H.} and \textsc{Zhao, Y.~Q.} (2011).
\newblock Tail asymptotics for a generalized two-demand queueing model--a
  kernel method.
\newblock \textit{Queueing Systems}, \textbf{69} 77--100.
\newblock \urlprefix\url{http://dx.doi.org/10.1007/s11134-011-9227-0}.

\bibitem[{Miyazawa(2009)}]{Miya2009}
\textsc{Miyazawa, M.} (2009).
\newblock Tail decay rates in double {QBD} processes and related reflected
  random walks.
\newblock \textit{Math. Oper. Res.}, \textbf{34} 547--575.
\newblock \urlprefix\url{http://dx.doi.org/10.1287/moor.1090.0375}.

\bibitem[{Miyazawa(2011)}]{Miya2011}
\textsc{Miyazawa, M.} (2011).
\newblock Light tail asymptotics in multidimensional reflecting processes for
  queueing networks.
\newblock \textit{TOP}, \textbf{19} 233--299.

\bibitem[{Miyazawa and Taylor(1997)}]{MiyaTayl1997}
\textsc{Miyazawa, M.} and \textsc{Taylor, P.~G.} (1997).
\newblock A geometric product-form distribution for a queueing network with
  non-standard batch arrivals and batch transfers.
\newblock \textit{Advances in Applied Probability}, \textbf{29} 523--544.
\newblock \urlprefix\url{http://dx.doi.org/10.2307/1428015}.

\bibitem[{Miyazawa and Zhao(2004)}]{MiyaZhao2004}
\textsc{Miyazawa, M.} and \textsc{Zhao, Y.~Q.} (2004).
\newblock The stationary tail asymptotics in the {$GI/G/1$}-type queue with
  countably many background states.
\newblock \textit{Adv. in Appl. Probab.}, \textbf{36} 1231--1251.
\newblock \urlprefix\url{http://dx.doi.org/10.1239/aap/1103662965}.

\bibitem[{Miyazawa and Zwart(2012)}]{MiyaZwar2012}
\textsc{Miyazawa, M.} and \textsc{Zwart, B.} (2012).
\newblock Wiener-hopf factorizations for a multidimensional {Markov} additive
  process and their applications to reflected processes.
\newblock \textit{Stochastic Systems}, \textbf{2} 67--114.

\bibitem[{M{\"u}ller and Stoyan(2002)}]{MullStoy2002}
\textsc{M{\"u}ller, A.} and \textsc{Stoyan, D.} (2002).
\newblock \textit{Comparison methods for stochastic models and risks}.
\newblock Wiley Series in Probability and Statistics, John Wiley \& Sons Ltd.,
  Chichester.

\bibitem[{Nakagawa(2004)}]{Naka2004}
\textsc{Nakagawa, K.} (2004).
\newblock On the exponential decay rate of the tail of a discrete probability
  distribution.
\newblock \textit{Stochastic Models}, \textbf{20} 31--42.

\bibitem[{Nummelin(1984)}]{Numm1984}
\textsc{Nummelin, E.} (1984).
\newblock \textit{General irreducible {Markov} chains and non-negative
  operators}.
\newblock Cambridge University Press.

\bibitem[{Rockafellar(1970)}]{Rock1970}
\textsc{Rockafellar, R.~T.} (1970).
\newblock \textit{Convex analysis}.
\newblock Princeton Mathematical Series, No. 28, Princeton University Press,
  Princeton, N.J.

\bibitem[{Scarsini(1998)}]{Scar1998}
\textsc{Scarsini, M.} (1998).
\newblock Multivariate convex orderings, dependence, and stochastic equality.
\newblock \textit{Journal of Applied Probability}, \textbf{35} 93--103.

\end{thebibliography}

\end{document}